\documentclass[11pt]{article}
\usepackage[obeyspaces,hyphens,spaces]{url}
\usepackage{framed} %for shaded
\usepackage[title]{appendix}

\providecommand{\rras}{\rightrightarrows}

\usepackage{amsthm,amsmath,amsfonts,mathtools,wasysym,empheq}
\usepackage{amssymb} %comment out if use mtpro2

\usepackage{mathpazo}

\usepackage[mathscr]{eucal} % since mathrsfs supports \mathscr, consider not using this

\usepackage[margin=1.1in]{geometry}
\usepackage[title]{appendix}

\usepackage[table, dvipsnames]{xcolor} %

\usepackage[labelsep=space]{caption}  % or period

\usepackage{graphicx, soul}
\usepackage{float,booktabs,multirow}
\restylefloat{table*}
\usepackage{tikz}
\usepackage{pgfplots}
\usetikzlibrary{calc}
\usetikzlibrary{decorations.pathreplacing}
\usetikzlibrary{shapes,intersections}
\usepgfplotslibrary{fillbetween}
\usepgfplotslibrary{colormaps}
\usetikzlibrary{decorations.markings,arrows}

\usepackage{todonotes}

\definecolor{labelkey}{rgb}{.1,.1,.8}
\definecolor{refkey}{rgb}{0,0.6,0.0}

\usepackage[shortlabels,inline]{enumitem}
\setlist[enumerate]{itemsep = 0.5pt,topsep=2pt}

\definecolor{dgreen}{rgb}{0.00,0.49,0.00}
\definecolor{dblue}{rgb}{0,0.08,0.75}
\colorlet{myblue}{dblue}
\colorlet{mygreen}{dgreen}
\definecolor{myfirstblue}{rgb}{.8, .8, 1}

\newcommand*\mybluebox[1]{%
    \colorbox{RoyalBlue!20}{\hspace{1em}#1\hspace{1em}}}

\usepackage{hyperref}
\hypersetup{
	linktocpage=true,
	colorlinks= true,
	urlcolor = magenta,
	citecolor = mygreen,
	linkcolor = dblue
}

\allowdisplaybreaks

\usepackage[capitalize, nameinlink]{cleveref} % color in name

\crefdefaultlabelformat{#2\textup{#1}#3}
\crefrangeformat{enumi}{#3\textup{#1}#4\textup{\textendash}#5\textup{#2}#6}
\crefname{equation}{}{}

\crefname{chapter}{Appendix}{chapters}
\crefname{item}{}{items}
\crefname{figure}{Figure}{Figures}
\crefname{theorem}{\protect\theoremname}{Theorems}
\crefname{lemma}{\protect\lemmaname}{\protect\lemmaname}
\crefname{proposition}{\protect\propositionname}{\protect\propositionname}
\crefname{corollary}{\protect\corollaryname}{\protect\corollaryname}
\crefname{definition}{\protect\definitionname}{\protect\definitionname}
\crefname{fact}{\protect\factname}{\protect\factname}
\crefname{example}{\protect\examplename}{\protect\examplename}
\crefname{algorithm}{Algorithm}{Algorithms}
\crefname{remark}{\protect\remarkname}{\protect\remarkname}
\crefname{case}{\protect\casename}{\protect\casename}
\crefname{question}{\protect\questionname}{\protect\questionname}
\crefname{claim}{\protect\claimname}{\protect\claimname}
\crefname{enumi}{}{}
\crefname{appsec}{Appendix}{Appendices}
\makeatletter

\g@addto@macro\normalsize{%
  \setlength\abovedisplayskip{6pt}
  \setlength\belowdisplayskip{6pt}
  \setlength\abovedisplayshortskip{6pt}
  \setlength\belowdisplayshortskip{6pt}
}

\let\orgdescriptionlabel\descriptionlabel
\renewcommand*{\descriptionlabel}[1]{%
	\let\orglabel\label
	\let\label\@gobble
	\phantomsection
	\edef\@currentlabel{#1}%
	\let\label\orglabel
	\orgdescriptionlabel{#1}%
}

\let\leq\leqslant
\let\geq\geqslant

\def\th@plain{%
	\thm@notefont{} % same as heading font
	\itshape % body font
}
\def\th@definition{%
	\thm@notefont{}% same as heading font
	\normalfont % body font
}

\g@addto@macro\th@remark{\thm@headpunct{}}
\g@addto@macro\th@definition{\thm@headpunct{}}
\g@addto@macro\th@plain{\thm@headpunct{}}

\makeatother

\usepackage[framemethod=tikz]{mdframed}

\definecolor{ocre}{RGB}{243,102,25}
\definecolor{mygray}{RGB}{243,244,254}

\theoremstyle{plain}

\newtheorem{theorem}{\protect\theoremname}[section]
\newtheorem{corollary}[theorem]{\protect\corollaryname}
\newtheorem{lemma}[theorem]{\protect\lemmaname}
\newtheorem{proposition}[theorem]{\protect\propositionname}

\theoremstyle{definition}

\newtheorem{remark}[theorem]{\protect\remarkname}
\newtheorem{definition}[theorem]{\protect\definitionname}
\newtheorem{example}[theorem]{\protect\examplename}

\newtheorem{fact}[theorem]{\protect\factname}

\theoremstyle{remark}

\newcommand{\hilbert}{X}
\newcommand{\id}{{\rm{Id}}}
\newcommand{\idparam}{\alpha}
\newcommand{\Nparam}{\beta}

\providecommand{\theoremname}{Theorem}
\providecommand{\propositionname}{Proposition}
\providecommand{\corollaryname}{Corollary}
\providecommand{\factname}{Fact}
\providecommand{\lemmaname}{Lemma}
\providecommand{\assumptionname}{Assumption}
\providecommand{\algorithmname}{Algorithm}
\providecommand{\siff}{\Leftrightarrow}

\providecommand{\rras}{\rightrightarrows}
\providecommand{\fady}{\varnothing}
\providecommand{\gra}{\operatorname{gra}}
\providecommand{\IN}{I-N decomposition}
\providecommand{\definitionname}{Definition}
\providecommand{\notationname}{Notation}
\providecommand{\remarkname}{Remark}
\providecommand{\RA}{\Rightarrow}

\providecommand{\examplename}{Example}
\providecommand{\claimname}{Claim}

\providecommand{\algorithmname}{Algorithm}
\providecommand{\openprobname}{Open Problem}
\newcommand{\nnn}{\ensuremath{{n\in{\mathbb N}}}}

\providecommand{\pars}{\partial_{\mathmakebox[0.4em][l]{\textrm{\#}}}}

\let\originalleft\left
\let\originalright\right
\renewcommand{\left}{\mathopen{}\mathclose\bgroup\originalleft}
\renewcommand{\right}{\aftergroup\egroup\originalright}

\DeclarePairedDelimiter{\abs}{\lvert}{\rvert}

\DeclarePairedDelimiter{\norm}{\lVert}{\rVert}
\DeclarePairedDelimiterX{\scal}[2]{\langle}{\rangle}{  #1 \, \delimsize \vert \, \mathopen{}  #2  }

\DeclarePairedDelimiterX{\fnorm}[1]{\lVert}{\rVert_\ensuremath{\mathsf{F}}}{#1}

\DeclarePairedDelimiterX\menge[2]{ \{ }{ \} }{ {#1} ~ \delimsize \vert ~ \mathopen{}  {#2} }  % set
\DeclarePairedDelimiterX\fa[2]{ ( }{ )_{#2} }{#1}  % indexed family of vectors
\DeclarePairedDelimiterX\set[2]{ \{ }{ \}_{#2} }{#1}  % for indexed set

\newcommand{\RR}{\ensuremath{\mathbb R}}
\newcommand{\NN}{\ensuremath{\mathbb N}}

\DeclareMathOperator*{\argmin}{argmin}

\newcommand{\dom}{\ensuremath{\operatorname{dom}}}
\newcommand{\ri}{\ensuremath{\operatorname{ri}}}
\newcommand{\sri}{\ensuremath{\operatorname{sri}}}

\newcommand{\Fix}{\ensuremath{\operatorname{Fix}}}
\newcommand{\Id}{\ensuremath{\operatorname{Id}}}
\newcommand{\prox}[1]{\ensuremath{\mathop{\operatorname{Prox}_{#1}}}}
\newcommand{\grad}[1]{\ensuremath{\mathop{\nabla {#1} }   }}

\newcommand{\zer}{{\ensuremath{\operatorname{zer\,}}}} % emptyset
\newcommand{\weakly}{{\ensuremath{\,\rightharpoonup\,}}} % emptyset
\newcommand{\fix}{\ensuremath{\operatorname{Fix}}}

\newcommand{\marker}[2]% [centerX,centerY,radius,angle1,angle2]
{
  \pgfgettransformentries{\myxscale}{\@tempa}{\@tempa}{\myyscale}{\@tempa}{\@tempa}
  \draw[thick] ($#1+#2*(1/\myxscale,1/\myyscale)$)--($#1-#2*(1/\myxscale,1/\myyscale)$);
  \draw[thick] ($#1+#2*(-1/\myxscale,1/\myyscale)$)--($#1-#2*(-1/\myxscale,1/\myyscale)$);
}

\xdef\zero{0}
\xdef\one{1}
\xdef\two{2}

\newcommand{\compositioncc}[8]{

\xdef\alphavar{#1}
\xdef\betavar{#2}
\xdef\sigmavar{#3}
\xdef\drawcontr{#4}
\xdef\negavgcircle{#5}
\xdef\drawconscontr{#6}
\xdef\writecontr{#7}
\xdef\animate{#8}

\xdef\avgone{#2}
\xdef\avgtwo{#3}
\draw (0,0) circle (1);
\draw ({(1-\alphavar)+\alphavar*(1-\avgone)},0) circle (\alphavar*\avgone);
\marker{(1,0)}{0.08};
\marker{(0,0)}{0.08};

\draw (1,-1)--(1,1);

\ifx\animate\one
\fi

\begin{scope}[shift={({1-\alphavar},0)},scale={\alphavar}]
\foreach \s in {0,...,20}
         {
           \pgfmathsetmacro{\circx}{{((1-\avgone)-\avgone*cos(\s*360/20))}}
           \pgfmathsetmacro{\circy}{{(\avgone*sin(\s*360/20))}}
           \pgfmathsetmacro{\circr}{{sqrt((\circx)^2+(\circy)^2)}}
           \draw[dashed] ({\circx*(1-\avgtwo)},{\circy*(1-\avgtwo)}) circle ({\avgtwo*\circr});
            \marker{(\circx,\circy)}{0.08};
            \ifx\animate\one
         \fi
         }
\ifx\negavgcircle\one
\pgfmathsetmacro{\resavg}{{(1/\sigmavar+\betavar)/(1/\sigmavar+\betavar+1)}}
\draw[thick,gray] (\resavg-1,0) circle (\resavg);
\fi
\end{scope}

\ifx\drawconscontr\one
\pgfmathsetmacro{\conscontr}{{abs(1-2*\alphavar+\alphavar*(1/\sigmavar+\betavar)/(1/\sigmavar+\betavar+1))+\alphavar*(1/\sigmavar+\betavar)/(1/\sigmavar+\betavar+1)}}
\draw[thick] (0,0) circle (\conscontr);
\fi

\pgfmathsetmacro{\contrDrankone}{{sqrt(abs(((1 - \alphavar)*(-\alphavar + \alphavar^2 - \sigmavar + 2*\alphavar*\sigmavar - \betavar*\sigmavar + 2*\alphavar*\betavar*\sigmavar - 2*\alphavar^2*\betavar*\sigmavar - \sigmavar^2 + 4*\alphavar*\sigmavar^2 - 4*\alphavar^2*\sigmavar^2 - \betavar*\sigmavar^2 + 2*\alphavar*\betavar*\sigmavar^2 - \alphavar*\betavar^2*\sigmavar^2 + \alphavar^2*\betavar^2*\sigmavar^2))/(\sigmavar*(-1 + \alphavar - \betavar + \alphavar*\betavar - \sigmavar + 2*\alphavar*\sigmavar - \betavar*\sigmavar + \alphavar*\betavar*\sigmavar))))}}

\pgfmathsetmacro{\lambdaone}{{\alphavar^2 + \alphavar*\sigmavar - 2*\alphavar^2*\sigmavar + \alphavar*\betavar*\sigmavar - \alphavar^2*\betavar*\sigmavar}}

\pgfmathparse{\lambdaone<0? 1: \contrDrankone};
\xdef\contrDrankone{\pgfmathresult};

\pgfmathsetmacro{\lambdatwo}{{-(((-\alphavar + \alphavar^2)*\betavar*(-\alphavar - \sigmavar + 2*\alphavar*\sigmavar - \betavar*\sigmavar + \alphavar*\betavar*\sigmavar))/(-1 + \alphavar - \betavar + \alphavar*\betavar - \sigmavar + 2*\alphavar*\sigmavar - \betavar*\sigmavar + \alphavar*\betavar*\sigmavar))}}

\pgfmathparse{\lambdatwo<0? 1: \contrDrankone};
\xdef\contrDrankone{\pgfmathresult};

\pgfmathsetmacro{\contrPrankone}{{max(abs(1-2*\alphavar),abs(1-2*\alphavar*\sigmavar/(1+\sigmavar)),abs(1-2*\alphavar/(1+\betavar)),abs(1-2*\alphavar*(\sigmavar+\betavar)/((1+\sigmavar)*(1+\betavar))))}}

\ifx\drawcontr\one
    \draw[thick] (0,0) circle (\contrPrankone);
\ifx\writecontr\one
    \node at (0,-1.2) {Factor: \pgfmathprintnumber[fixed,precision=2]{\contrPrankone}};
\fi
\else
\ifx\drawcontr\two
    \draw[thick] (0,0) circle (\contrDrankone);
\ifx\writecontr\one
    \node at (0,-1.2) {Factor: \pgfmathprintnumber[fixed,precision=2]{\contrDrankone}};
\fi
\fi
\fi
}

\newcommand{\composition}[8]{

 \xdef\thetavar{#1}
\xdef\alphaone{#2}
\xdef\betaone{#3}
\xdef\alphatwo{#4}
\xdef\betatwo{#5}
\xdef\drawcircle{#6}
\xdef\writeprop{#7}
\xdef\fillgray{#8}

\ifx\fillgray\zero
\xdef\nbrinnercirc{20} % number of second operator circles
\begin{scope}[shift={({1-\thetavar},0)},scale={\thetavar}]
\foreach \s in {0,...,\nbrinnercirc}
         {
           \pgfmathsetmacro{\circx}{{\alphaone-\betaone*cos(\s*360/20)}}
           \pgfmathsetmacro{\circy}{{\betaone*sin(\s*360/20)}}
           \pgfmathsetmacro{\distc}{{sqrt((\circx)^2+(\circy)^2)}}
           \draw[dashed] ({\alphatwo*\circx},{\alphatwo*\circy}) circle ({\distc*\betatwo});
           \marker{(\circx,\circy)}{0.08};
         }
       \end{scope}
       \fi

\ifx\fillgray\one
\xdef\nbrinnercircquad{15} % number of second operator circles per quarter
 \begin{scope}[shift={({1-\thetavar},0)},scale={\thetavar}]
\foreach \s in {0,...,\nbrinnercircquad}
         {
           \pgfmathsetmacro{\circx}{{\alphaone-\betaone*cos(\s*90/\nbrinnercircquad))}}
           \pgfmathsetmacro{\circy}{{(\betaone*sin(\s*90/\nbrinnercircquad))}}
           \pgfmathsetmacro{\distc}{{sqrt((\circx)^2+(\circy)^2)}}
           \draw[name path=circleone] ({\alphatwo*\circx+\distc*\betatwo/2},{\alphatwo*\circy+\distc*\betatwo/2})--({\alphatwo*\circx-\distc*\betatwo/2},{\alphatwo*\circy-\distc*\betatwo/2});
           \draw[color=gray!25!white,fill=gray!25!white] ({\alphatwo*\circx},{\alphatwo*\circy}) circle ({\distc*\betatwo});
           
           \pgfmathsetmacro{\circx}{{\alphaone-\betaone*cos((\s+\nbrinnercircquad)*90/\nbrinnercircquad))}}
           \pgfmathsetmacro{\circy}{{(\betaone*sin((\s+\nbrinnercircquad)*90/\nbrinnercircquad))}}
           \pgfmathsetmacro{\distc}{{sqrt((\circx)^2+(\circy)^2)}}
           \draw[name path=circletwo] ({\alphatwo*\circx+\distc*\betatwo/2},{\alphatwo*\circy-\distc*\betatwo/2})--({\alphatwo*\circx-\distc*\betatwo/2},{\alphatwo*\circy+\distc*\betatwo/2});
           \draw[color=gray!25!white,fill=gray!25!white] ({\alphatwo*\circx},{\alphatwo*\circy}) circle ({\distc*\betatwo});

           \pgfmathsetmacro{\circx}{{\alphaone-\betaone*cos((\s+2*\nbrinnercircquad)*90/\nbrinnercircquad))}}
           \pgfmathsetmacro{\circy}{{(\betaone*sin((\s+2*\nbrinnercircquad)*90/\nbrinnercircquad))}}
           \pgfmathsetmacro{\distc}{{sqrt((\circx)^2+(\circy)^2)}}
           \draw[name path=circlethree] ({\alphatwo*\circx+\distc*\betatwo/2},{\alphatwo*\circy+\distc*\betatwo/2})--({\alphatwo*\circx-\distc*\betatwo/2},{\alphatwo*\circy-\distc*\betatwo/2});
           \draw[color=gray!25!white,fill=gray!25!white] ({\alphatwo*\circx},{\alphatwo*\circy}) circle ({\distc*\betatwo});
           
           \pgfmathsetmacro{\circx}{{\alphaone-\betaone*cos((\s+3*\nbrinnercircquad)*90/\nbrinnercircquad))}}
           \pgfmathsetmacro{\circy}{{(\betaone*sin((\s+3*\nbrinnercircquad)*90/\nbrinnercircquad))}}
           \pgfmathsetmacro{\distc}{{sqrt((\circx)^2+(\circy)^2)}}
           \draw[name path=circlefour] ({\alphatwo*\circx+\distc*\betatwo/2},{\alphatwo*\circy-\distc*\betatwo/2})--({\alphatwo*\circx-\distc*\betatwo/2},{\alphatwo*\circy+\distc*\betatwo/2});
           \draw[color=gray!25!white,fill=gray!25!white] ({\alphatwo*\circx},{\alphatwo*\circy}) circle ({\distc*\betatwo});
           
           \tikzfillbetween[of=circletwo and circlefour] {fill=gray!25!white};
           \tikzfillbetween[of=circleone and circlethree] {fill=gray!25!white};
         }

 \end{scope}
\fi
\ifx\fillgray\zero
\begin{scope}[shift={({1-\thetavar},0)},scale={\thetavar}]
  \draw (\alphaone,0) circle (\betaone);
  \end{scope}
\fi
\unitcirc{1}
 
 \pgfmathsetmacro{\deltaone}{{(\alphaone)/(1-\alphaone)*(1-((1-\alphatwo)^2-(\betatwo)^2)/(1-\alphatwo))}}
 \pgfmathsetmacro{\deltatwo}{{(\alphatwo)/(1-\alphatwo)}}
 \pgfmathsetmacro{\deltathree}{{1-((((1-\alphaone)^2-(\betaone)^2)/(1-\alphaone))*(1-(((1-\alphatwo)^2-(\betatwo)^2)/(1-\alphatwo)))+((1-\alphatwo)^2-(\betatwo)^2)/(1-\alphatwo))}}
 \pgfmathsetmacro{\deltafour}{{(\deltaone*\deltatwo)/(\deltaone+\deltatwo)}}
 \pgfmathsetmacro{\alphacomp}{{(1-\thetavar)+\thetavar*\deltafour/(1+\deltafour)}}
 \pgfmathsetmacro{\betacomp}{{\thetavar*sqrt(\deltathree-\deltafour+\deltathree*\deltafour)/(1+\deltafour)}}
 \pgfmathsetmacro{\alphacomptwo}{{(1-\thetavar)+\thetavar*(\alphaone+\betaone)*(\alphatwo+\betatwo)*(1-((\betaone*\alphatwo)+(\betatwo*\alphaone))/(\alphaone*\alphatwo+\alphaone*\betatwo+\alphatwo*\betaone))}}
 \pgfmathsetmacro{\betacomptwo}{{\thetavar*(\alphaone+\betaone)*(\alphatwo+\betatwo)*(((\betaone*\alphatwo)+(\betatwo*\alphaone))/(\alphaone*\alphatwo+\alphaone*\betatwo+\alphatwo*\betaone))}}
  \pgfmathsetmacro{\alphacompthree}{{(1-\thetavar)+\thetavar*(abs(\alphaone)+\betaone)*(abs(\alphatwo)+\betatwo)*(1-((\betaone*abs(\alphatwo))+(\betatwo*abs(\alphaone)))/(abs(\alphaone)*abs(\alphatwo)+abs(\alphaone)*\betatwo+abs(\alphatwo)*\betaone))}}
 \pgfmathsetmacro{\betacompthree}{{\thetavar*(abs(\alphaone)+\betaone)*(abs(\alphatwo)+\betatwo)*(((\betaone*abs(\alphatwo))+(\betatwo*abs(\alphaone)))/(abs(\alphaone)*abs(\alphatwo)+abs(\alphaone)*\betatwo+abs(\alphatwo)*\betaone))}}
 \ifx\drawcircle\one 
 \draw[thick] (\alphacomptwo,0) circle (\betacomptwo);
 \fi
 
 \ifx\writeprop\one
 \node at (-2,-1) {\deltaone};
 \node at (-2,-0.5) {\deltatwo};
 \node at (-2,0) {\deltathree};
 \node at (-2,0.5) {\deltafour};
 \node at (-2,1) {\alphacomp};
 \node at (-2,1.5) {\betacomp};
 \node at (-1,1) {\alphacomptwo};
 \node at (-1,1.5) {\betacomptwo};
 \fi
 
}

\newcommand{\computealphabetacomp}[5]{

  \xdef\thetavar{#1}
  \xdef\alphaone{#2}
  \xdef\betaone{#3}
  \xdef\alphatwo{#4}
  \xdef\betatwo{#5}

  \pgfmathsetmacro{\alphacomptwo}{{(1-\thetavar)+\thetavar*(\alphaone+\betaone)*(\alphatwo+\betatwo)*(1-((\betaone*\alphatwo)+(\betatwo*\alphaone))/(\alphaone*\alphatwo+\alphaone*\betatwo+\alphatwo*\betaone))}}
  \pgfmathsetmacro{\betacomptwo}{{\thetavar*(\alphaone+\betaone)*(\alphatwo+\betatwo)*(((\betaone*\alphatwo)+(\betatwo*\alphaone))/(\alphaone*\alphatwo+\alphaone*\betatwo+\alphatwo*\betaone))}}
  
}

\newcommand{\unitcirc}[1]{

  \xdef\markerlabel{#1}
  
  \marker{(0,0)}{0.08}
  \marker{(1,0)}{0.08}
  \draw[very thick] (0,0) circle (1);

  \ifx\markerlabel\one
  \node[right] at (1,0) {\footnotesize $x-y$};
  \node[left] at (0,0) {\footnotesize $0$};

  \fi
  }

\newcommand{\shiftLip}[3]{

  \xdef\alphavar{#1}
  \xdef\betavar{#2}
  \xdef\colormix{#3}

  \draw[fill=gray!\colormix!white] (\alphavar,0) circle (\betavar);
  }

  \newcommand{\cococirc}[2]{

    \xdef\betavar{#1}
    \xdef\colormix{#2}
    \shiftLip{{1/(2*\betavar)}}{{1/(2*\betavar)}}{\colormix}
    
  }

  \newcommand{\coniccirc}[2]{

    \xdef\thetavar{#1}
    \xdef\colormix{#2}
    \shiftLip{{(1-\thetavar)}}{{\thetavar}}{\colormix}

  }

  \newcommand{\lipcirc}[2]{
    \xdef\betavar{#1}
    \xdef\colormix{#2}
    \shiftLip{{0}}{{\betavar}}{\colormix}
  }

  \newcommand{\mumono}[2]{
    \xdef\muvar{#1}
    \xdef\colormix{#2}
    \draw[draw=gray!\colormix!white,fill=gray!\colormix!white] (\muvar,-1.1) rectangle (1.35,1.1);
    \draw (\muvar,-1.1)--(\muvar,1.1);
  }

  \newcommand{\legend}[7]{

    \xdef\nbrentries{#1}
    \xdef\xpos{#2}
    \xdef\xspace{#3}
    \xdef\yspace{#4}
    \xdef\param{#5}
    \xdef\labels{#6}
    \xdef\colormix{#7}

    \pgfmathsetmacro{\lastelement}{\nbrentries-1}

    \foreach \i in {0,...,\lastelement}{
      \pgfmathsetmacro{\currentcolormix}{array(\colormix,\i)}
      
      \draw[fill=gray!\currentcolormix!white] (\xpos,{\yspace*((\nbrentries-1)/2-\i)}) circle (0.05);
      \pgfmathsetmacro{\currentlabel}{array(\labels,\i)};
      \node[right] at (\xpos+\xspace,{\yspace*((\nbrentries-1)/2-\i)}) {\footnotesize $\param=\;$\currentlabel};

    }
    }

\xdef\showgraphics{\one}
\newcommand{\email}[1]{\href{mailto:#1}{\nolinkurl{#1}}} % for email

\DeclareSymbolFont{italics}{\encodingdefault}{\rmdefault}{m}{it}
\DeclareMathSymbol{f}{\mathalpha}{italics}{`f}

\let\oldFootnote\footnote
\newcommand\nextToken\relax

\renewcommand\footnote[1]{%
    \oldFootnote{#1}\futurelet\nextToken\isFootnote}

\newcommand\isFootnote{%
    \ifx\footnote\nextToken\textsuperscript{,}\fi}

\begin{document}

% using today is a bad idea.
 \author{
 Pontus Giselsson\thanks{Department of Automatic Control,
 Lund University,
 Lund, Sweden.
 E-mail:
 \texttt{pontus.giselsson@control.lth.se}.}
 ~and~ Walaa M.\ Moursi\thanks{
 Department of Combinatorics and Optimization,
 University of Waterloo,
 Waterloo, Ontario, N2L~3G1,
 Canada. E-mail:
 \texttt{walaa.moursi@uwaterloo.ca}.
 %and
 %Mansoura University, Faculty of Science,
 %Mathematics Department,
 %Mansoura 35516, Egypt.
 }}

\title{ \sffamily
On compositions of special cases\\ of Lipschitz continuous operators}
% \hl{Pontus: Please feel free to change/edit. Also please add your email, etc.}

 \date{December 30, 2019}
\maketitle
\begin{abstract}
\noindent
Many iterative optimization algorithms
involve  
compositions of special cases of Lipschitz continuous operators,
namely firmly nonexpansive, averaged and nonexpansive operators.
The structure and properties of the compositions are of particular importance 
in the proofs of convergence of such algorithms.
In this paper, we systematically study the compositions of
further special cases of
Lipschitz continuous operators.
Applications of our results include 
compositions of scaled conically nonexpansive mappings, as well
as the Douglas--Rachford and forward-backward operators, when applied
to solve certain structured monotone inclusion and optimization problems.
Several examples illustrate and tighten 
our conclusions. 
\end{abstract}
\maketitle
{\small

}

% use many carriage returns	(easy to edit)

%introduction

%introduction

\section{Introduction}
In this paper, we assume that
\begin{empheq}[box=\mybluebox]{equation*}
\label{T:assmp}
\text{$X$ is a real Hilbert space},
\end{empheq}
with inner product $\scal{\cdot}{\cdot}$ and
induced norm $\norm{\cdot}$. 
Let $L>0$ and let $T\colon X\to X$.
Then $T$ is
$L$-\emph{Lipschitz continuous} if
$(\forall (x,y)\in X\times X)$ 
$\norm{Tx-Ty}\le L\norm{x-y}$,
and $T$ is \emph{nonexpansive}
if $T$ is $1$-Lipschitz continuous, 
i.e., $(\forall (x,y)\in X\times X)$ 
$\norm{Tx-Ty}\le \norm{x-y}$.
In this paper, we study compositions of, what we call (see \cref{def:in:decomp}), Identity-Nonexpansive decompositions (\IN s for short) of Lipschitz continuous operators. 
Let $(\alpha,\beta)\in \RR^2$ and let $\Id\colon X\to X$
be the \emph{identity operator} on $X$.
A Lipschitz continuous operator $R$ admits an $(\idparam,\Nparam)$-\IN~if $R=\idparam\id+\Nparam N$ for some nonexpansive
 operator $N\colon X\to X$. For instance, averaged\footnote{Let
$T\colon X\to X$. Then 
$T$ is
$\alpha$-\emph{averaged} if $\alpha\in \left]0,1\right[$
and 
nonexpansive
$N\colon X\to X$ exists such that 
$
T=(1-\alpha)\Id+\alpha N$.}, conically nonexpansive\footnote{Let
$T\colon X\to X$. Then 
$T$ is
$\alpha$-\emph{conically nonexpansive} if $\alpha\in \left]0,\infty\right[$
and 
nonexpansive
$N\colon X\to X$ exists such that 
$
T=(1-\alpha)\Id+\alpha N$.}, and cocoercive\footnote{
Let $T\colon X\to X$, and let $\beta>0$.
Then 
$T$ is $\tfrac{1}{\beta}$-\emph{cocoercive} if  
nonexpansive $N:X\to X$ exists such that 
$T=\tfrac{\beta}{2}(\id+N)$.} 
operators are all Lipschitz continuous operators that admit
special \IN s. 

We consider compositions of the form
\begin{equation}
\label{eq:comp}
    R=R_m\ldots R_1,
\end{equation}
where $m\in \{2,3\ldots\}$,
$I=\{1,\ldots, m\}$,
and $(R_i)_{i\in I}$ is a family of Lipschitz continuous operators
such that, for each $i\in I$,
$R_i$ admits an $(\alpha_i,\beta_i)$-\IN. 
That is, $R_i=\idparam_i\Id+\Nparam_iN_i$ for all $i\in I$, where 
$\alpha_i$ and $\beta_i$ are real numbers, and $N_i\colon X\to X$ are 
nonexpansive for all $i\in I$.
A straightforward (and naive) conclusion is that 
the composition is Lipschitz continuous with a constant 
$\Pi_{i\in I}\big(\abs{\idparam_i}+\abs{\Nparam_i}\big)$.
However, such a conclusion can be further refined 
when, for instance, each $R_i$ is an averaged operator.
Indeed, in this case
it is known that the composition is an averaged 
(and not just Lipschitz continuous)
operator
(see, e.g., \cite[Proposition~4.46]{BC2017},
\cite[Lemma~2.2]{Comb04}
and \cite[Theorem~3]{OguraYamada2002}).
In this paper, we provide a systematic study of the
structure of $R$, under additional assumptions
on the decomposition parameters.

Our main result is stated in \cref{thm:comp:more:gen}. 
We show that for $m=2$, under a mild assumption on 
$(\idparam_1,\idparam_2,\Nparam_1,\Nparam_2)$ the composition \cref{eq:comp}
is a scalar multiple of a conically nonexpansive operator.
As a consequence of \cref{thm:comp:more:gen}, 
we show in \Cref{thm:conic:conic} that, 
under additional assumptions on the decomposition parameters, compositions of 
scaled conically nonexpansive mappings are scaled conically 
nonexpansive mappings, see also \cite{BDP19} for a relevant result\footnote{The paper
\cite{BDP19} appeared online while putting the finishing touches on this paper. Partial results of this work were presented by the second author at the \emph{Numerical Algorithms in Nonsmooth Optimization} workshop at Erwin Schr\"{o}dinger International Institute for Mathematics and Physics (ESI) in Vienna in February 2019 and at the 
\emph{Operator Splitting Methods in Data Analysis} workshop at the Flatiron Institute,
in New York in March 2019. Both workshops predate \cite{BDP19}.}. 
Special cases of \Cref{thm:conic:conic} include, e.g., compositions of 
averaged operators \cite[Proposition~4.46]{BC2017}, and 
compositions of averaged and negatively averaged operators 
\cite{Giselsson2017Tight}.

Of particular interest are compositions $R$ that are averaged, 
conically nonexpansive, or contractive. 
Let $x_0\in X$.
For an averaged  
(respectively contractive) operator $R$, 
the sequence $(R^kx_0)_{k\in \NN}$ converges weakly
(respectively strongly)
towards a fixed-point of $R$ (if one exists) 
\cite[Theorem~5.14]{BC2017}. 
For conically nonexpansive operators, 
a simple averaging trick gives an 
averaged 
operator with the same fixed-point set as the conically nonexpansive 
operator. 
Iterating the new averaged operator
yields a sequence that converges weakly to 
a fixed-point of the conically nonexpansive operator. 
These properties have been instrumental in proving convergence 
for the Douglas-Rachford algorithm and the forward-backward algorithm. 
In this paper, we apply our composition result 
\Cref{thm:conic:conic} to prove 
convergence of these splitting methods in new settings.

The Douglas--Rachford and forward-backward methods
traditionally solve monotone inclusion problems of the form
\begin{equation}
\label{eq:inclusion}
\text{Find $x\in X$ such that $0\in Ax+Bx$},
\end{equation}
where $A\colon X\rras X$ 
and $B\colon X\rras X$ are maximally monotone,
and, in the case of the forward-backward method,
$A$ is additionally assumed to be cocoercive.
The Douglas-Rachford method iterates the Douglas-Rachford map $T=\tfrac{1}{2}(\id+R_{\gamma B}R_{\gamma A})$, where\footnote{Let 
	$A\colon X\rras X$ be an operator.
	The \emph{resolvent} of $A$, denoted by $J_A$, is defined 
	by $J_A=(\Id+A)^{-1}$,
	and the \emph{reflected resolvent} of $A$, 
	denoted by $R_A$, is defined 
	by $R_A=2J_A-\Id$}
 $\gamma>0$ is a positive step-size. The Douglas--Rachford map is an averaged map of the composition of reflected resolvents. The forward-backward method iterates the forward-backward map $T=J_{\gamma B}(\id-\gamma A)$, where $\gamma>0$ is a positive step-size. The forward-backward map is a composition of a resolvent and a forward-step.

In this paper, we show that for Douglas-Rachford splitting, 
we need not impose monotonicity on the individual operators, 
but only on the sum, provided the sum is strongly monotone. 
The reflected resolvents, $R_{\gamma A}$ and $R_{\gamma B}$, are negatively conically nonexpansive, 
the composition is conically nonexpansive, 
and a sufficient averaging gives an averaged map that 
converges to a fixed-point when iterated. 
Relevant work appears in 
\cite{DP18}, \cite{GH17}
and \cite{GHY17}. 

More striking, for the forward-backward method we show that 
it is sufficient that the sum is monotone 
(not strongly monotone 
as for DR). 
More specifically, we show that identity can be shifted between 
the two operators, while still guaranteeing averagedness 
of the forward-backward map $T=J_{\gamma B}(\id-\gamma A)$. 
Indeed, the resolvent $J_{\gamma B}$
is cocoercive and the forward-step $(\id-\gamma A)$ is scaled averaged. 
This implies that the composition is averaged 
(given restrictions on the cocoercivity and averagedness parameters). 
Moreover, when the sum is strongly monotone, again with no assumptions 
on monotonicity of the individual operators, we show that the 
forward-backward map is contractive. We also prove tightness of our 
contraction factor. 

We also provide, in \Cref{thm:m:conic}, a generalization of \Cref{thm:conic:conic} to the setting in \cref{eq:comp} of compositions of more than two operators. We assume that all $R_i$ are scaled conically 
nonexpansive operators and provide conditions on the parameters that 
give a specific scaled conically nonexpansive representation of $R$. 
Our condition is symmetric in the individual operators and allows for 
one of them to be scaled conic, while the rest must be scaled averaged. 
This is in compliance with the $m=2$ case in \Cref{thm:conic:conic}. 
%that requires $\alpha_1\alpha_2<1$.

Finally, in \Cref{sec:graphical}, we provide graphical 
2D-representations of different operator classes that 
admit \IN s, 
such as Lipschitz continuous operators, averaged operators, 
and 
cocoercive operators. We also provide 2D-representations of 
compositions of two such operator classes. Illustrations of the 
firmly nonexpansive
\big($\tfrac{1}{2}$-averaged\big) and nonexpansive operator 
classes have previously 
appeared in \cite{Eckstein1989Splitting,Eckstein1992OnTheDouglas}, and 
illustrations of more operator classes that admit particular \IN s and 
their compositions have appeared in 
\cite{Giselsson2015Lecture,Ryu2019Scaled} and in early 
preprints of \cite{Giselsson2014Linear}.

\subsection*{Organization and notation}
The remainder of this paper is organized as
follows:
\cref{sec:aux}
presents useful facts and auxiliary results
that are used throughout the paper.
In \cref{sec:3}, we present the main abstract results 
of the paper.
\cref{sec:4} presents the main composition results  
of Lipschitz continuous operators that admit \IN s,
under mild assumptions on the decomposition parameters,
as well as illustrative and limiting examples.
In \cref{sec:main:1} and \cref{sec:main:2}, 
we present applications of our composition results to the 
Douglas--Rachford and forward-backward algorithms, respectively. 
In \cref{sec:main:3} 
we present applications of our results to 
optimization problems. Finally, in \cref{sec:graphical}, 
we provide graphical representations of many different \IN s and their 
compositions.

The notation we use is standard and
follows, e.g., \cite{BC2017}
or \cite{Rock98}.

% ====================================================================

\section{Facts and auxiliary results}
\label{sec:aux}

Let $\rho\in \RR$. Let $A\colon X\to X$.
Recall that $A$
is $\rho$-\emph{monotone} if $(\forall (x,u)\in \gra A)$
$(\forall (y,v)\in \gra A)$
\begin{equation}
\label{e:rm:def}
\scal{x-y}{u-v}\ge\rho \norm{x-y}^2,
\end{equation}
and is maximally $\rho$-monotone
if any proper extension of $\gra A$
will violate \cref{e:rm:def}. 
In passing we point out that
$A$ is (maximally) monotone 
(respectively $\rho$-hypomonotone,
$\rho$-strongly monotone) if $\rho=0$
(respectively $\rho<0$, $\rho>0$)
see, e.g., \cite[Chapter~20]{BC2017},
\cite[Definition~6.9.1]{BurIus}, 
\cite[Definition~2.2]{CombPenn04}
and \cite[Example~12.28]{Rock98}.

\begin{fact}
	\label{prop:static:zeros}
	Let $A\colon X\rras X$, let $B\colon X\rras X$,
	let $\lambda\in \RR$,
	and suppose that $\zer(A+B)= (A+B)^{-1}(0)\neq \fady$.
	Suppose that $J_A$ and $J_B$ are single-valued
	and that $\dom J_A=\dom J_B=X$.
	Set 
	\begin{equation}
	T=(1-\lambda )\Id +\lambda R_BR_A.
	\end{equation}
	Then $T$ is single-valued,
	$\dom T=X$ and
	\begin{equation}
	\zer(A+B)=J_A(\Fix R_BR_A)=J_A(\Fix T).
	\end{equation}
\end{fact}
\begin{proof}
	See	\cite[Lemma~4.1]{DP18}. 
\end{proof}

\begin{proposition}
	\label{prop:static:zeros:FB}
	Let $A\colon X\to X$, let $B\colon X\rras X$,
	and suppose that $\zer(A+B)= (A+B)^{-1}(0)\neq \fady$.
	Suppose that  $J_B$ is single-valued
	and that $ \dom J_B=X$.
	Set 
	\begin{equation}
	T=J_B(\Id-A).
	\end{equation}
	Then $T$ is single-valued,
	$\dom T=X$ and
	\begin{equation}
	\zer(A+B)= \Fix T.
	\end{equation}
\end{proposition}
\begin{proof}
	The proof is similar to the proof of 
	\cite[Proposition~26.1(iv)]{BC2017}\footnote{In 
		passing, we mention that 
		\cite[Proposition~26.1(iv)]{BC2017} assume that
		$A$ and $B$ are maximally monotone, which is not required 
		here. However, the proof is the same. }.
	Indeed, let $x\in X$. Then, $x\in \zer(A+B)$
	$\siff$ $-Ax\in Bx$ 
	$\siff$ $(\Id-A)x\in (\Id+B)x$
	$\siff$ $x=J_B(\Id-A)x=Tx$.
\end{proof}

\begin{lemma}
	\label{lem:T:gra:AB}
	Let $\lambda\in \RR$,
	let $R_1\colon X\to X$,
	let $R_2\colon X\to X$, and set
	\begin{equation}
	R_\lambda=(1-\lambda)\Id+\lambda R_2R_1.
	\end{equation}
	Let $(x,y)\in X\times X$.
	Then
	\begin{align}
	&\qquad\scal{R_\lambda x-R_\lambda y}{(\Id-R_\lambda)x-(\Id-R_\lambda)y}
	\nonumber
	\\
	&=(1-2\lambda)\scal{x-y}{(\Id-R_\lambda)x-(\Id-R_\lambda)y}
	\nonumber
	\\
	&\quad+\lambda^2\scal{(\Id+R_1)x-(\Id+R_1)y}{(\Id-R_1)x-(\Id-R_1)y}
	\nonumber
	\\
	&\quad+\lambda^2\scal{(\Id+R_2)R_1x-(\Id+R_2)R_1y}{
		(\Id-R_2)R_1x-(\Id-R_2)R_1y}.
	\end{align}
	
\end{lemma}

\begin{proof}
	See \cref{app:A}.
\end{proof}

\begin{proposition}
	\label{prop:180425:a}
	Let $\idparam\in \RR$, let 
	$\Nparam\in \RR$, let $N\colon \hilbert \to \hilbert $,
	and set $T=\idparam\id+\Nparam N$. 
	Let $(x,y)\in \hilbert\times \hilbert$.
	Then the following hold: 
	\begin{subequations}
		\begin{align}
		&\qquad\Nparam^2(\norm{x-y}^2-\norm{Nx-Ny}^2)
		\nonumber
		\\
		&= 
		(\Nparam^2-\idparam^2)\norm{x-y}^2-\norm{Tx-Ty}^2+2\alpha \scal{x-y}{Tx-Ty}
		\label{e:180425:a}	
		\\
		&= 
		(\Nparam^2-\idparam^2)\norm{x-y}^2-(1-2\alpha)\norm{Tx-Ty}^2+2\alpha 
		\scal{Tx-Ty}{(\Id-T)x-(\Id-T)y}
		\label{e:180425:ab}	
		\\
		&=
		(\Nparam^2-\idparam(\idparam-1))\norm{x-y}^2
		-\big((1-\idparam)\norm{Tx-Ty}^2+\idparam \norm{(\Id-T)x-(\Id-T)y}^2\big).
		\label{e:180425:b}
		\end{align}
	\end{subequations}
\end{proposition}
\begin{proof}
	Indeed, we have
	\begin{subequations}
		\begin{align}
		&\quad\Nparam^{2}(\|x-y\|^2-\|Nx-Ny\|^2 )
		\nonumber
		\\
		&=\Nparam^{2}\|x-y\|^2-\|(Tx-\idparam x)-(Ty-\idparam y)\|^2
		\\
		&=\Nparam^{2}\|x-y\|^2-(\|Tx-Ty\|^2+\idparam^2\|x-y\|^2-2\idparam
		\scal{Tx-Ty}{x-y})
		\\
		&=({\Nparam^2}-{\idparam^2})\|x-y\|^2-\big(\|Tx-Ty\|^2-2\idparam
		\scal{Tx-Ty}{x-y}\big)
		\label{e:180425:c}
		\\
		&=({\Nparam^2}-{\idparam^2}+\idparam)\|x-y\|^2-\big((1-\idparam)\|Tx-Ty\|^2
		\nonumber
		\\
		&\quad
		+\idparam\|Tx-Ty\|^2
		-2\idparam
		\scal{Tx-Ty}{x-y}+
		\idparam\|x-y\|^2\big)
		\\
		&=
		(\Nparam^2-\idparam(\idparam-1))\norm{x-y}^2
		-\big((1-\idparam)\norm{Tx-Ty}^2+\idparam \norm{(\Id-T)x-(\Id-T)y}^2\big).
		\label{e:180425:d}
		\end{align}
	\end{subequations}
	This proves \cref{e:180425:a} and \cref{e:180425:b} in view of 
	\cref{e:180425:c} and \cref{e:180425:d}.
	Finally, note that 
	$(\Nparam^2-\idparam^2)\norm{x-y}^2-\norm{Tx-Ty}^2+2\alpha \scal{x-y}{Tx-Ty}
	=(\Nparam^2-\idparam^2)\norm{x-y}^2-(1-2\idparam)\norm{Tx-Ty}^2
	-2\idparam\norm{Tx-Ty}^2+2\alpha \scal{x-y}{Tx-Ty}
	=	(\Nparam^2-\idparam^2)\norm{x-y}^2-(1-2\alpha)\norm{Tx-Ty}^2+2\alpha 
	\scal{Tx-Ty}{(\Id-T)x-(\Id-T)y}$.
	This proves \cref{e:180425:ab}.
\end{proof}

\begin{proposition}
	\label{prop:Tchar}
	Let $\idparam\in \RR$, let 
	$\Nparam\in \RR$, let $N\colon \hilbert \to \hilbert $,
	and set $T=\idparam\id+\Nparam N$. 
	Let $(x,y)\in \hilbert\times \hilbert$.
	Then the following are equivalent:
	\begin{enumerate}
		\item 
		\label{prop:Tchar:a}
		$N$ is nonexpansive.
		\item 
		\label{prop:Tchar:b}
		$  \|Tx-Ty\|^2-2\idparam\scal{x-y}{Tx-Ty} 
		\leq(\Nparam^2-\idparam^2)\|x-y\|^2$.
		\item 
		\label{prop:Tchar:c}
		$  (1-2\idparam)\|Tx-Ty\|^2-2\idparam\scal{Tx-Ty}{(\id-T)x-(\id-T)y} 
		\leq(\Nparam^2-\idparam^2)\|x-y\|^2$.
		\item 
		\label{prop:Tchar:d}
		$  (2\idparam-1)\|(\id-T)x-(\id-T)y\|^2-2(1-\idparam)\scal{Tx-Ty}{(\id-T)x-(\id-T)y} 
		\leq(\Nparam^2-(1-\idparam)^2)\|x-y\|^2$.
		\item 
		\label{prop:Tchar:e}
		$(1-\idparam)\|Tx-Ty\|^2 +\idparam\|(\id-T)x-(\id-T)y\|^2
		\leq(\Nparam^2-\idparam(\idparam-1))\|x-y\|^2
		$.	
	\end{enumerate}
	\label{prop:Tcharacterizations}
\end{proposition}
\begin{proof}
	\ref{prop:Tchar:a}$\siff$\ref{prop:Tchar:b}$\siff$\ref{prop:Tchar:c}$\siff$\ref{prop:Tchar:e}:
	This is a direct consequence of 
	\cref{prop:180425:a}.	
	\ref{prop:Tchar:a}$\siff$\ref{prop:Tchar:d}:
	Applying \cref{e:180425:ab}	 with $(T,\alpha,\beta)$ replaced by 
	$(\Id-T, 1-\alpha,-\beta)$
	yields
	$\Nparam^2(\norm{x-y}^2-\norm{Nx-Ny}^2)
	=
	(\Nparam^2-(1-\idparam)^2)\|x-y\|^2
	-(2\idparam-1)\|(\Id-T)x-(\Id-T)y\|^2+2(1-\idparam)\scal{Tx-Ty}{(\id-T)x-(\id-T)y} $.
	The proof is complete.
\end{proof}

\begin{proposition}
	\label{prop:Tchar:conic}
	Let $\idparam\in \RR$,  let $N\colon \hilbert \to \hilbert $,
	and set $T=(1-\idparam)\Id+\idparam N$. 
	Let $(x,y)\in \hilbert\times \hilbert$.
	Then the following are equivalent:
	\begin{enumerate}
		\item 
		\label{prop:Tchar:conic:a}
		$N$ is nonexpansive.
		\item 
		\label{prop:Tchar:conic:b}
		$  \|Tx-Ty\|^2-2(1-\idparam)\scal{x-y}{Tx-Ty} 
		\leq(2\idparam-1)\|x-y\|^2$.
		\item 
		\label{prop:Tchar:conic:c}
		$  (2\idparam-1)\|Tx-Ty\|^2-2(1-\idparam)\scal{Tx-Ty}{(\id-T)x-(\id-T)y} 
		\leq(2\idparam-1)\|x-y\|^2$.
		\item 
		\label{prop:Tchar:conic:d}
		$  (1-2\idparam)\|(\id-T)x-(\id-T)y\|^2\leq2\idparam\scal{Tx-Ty}{(\id-T)x-(\id-T)y} 
		$.
		\item 
		\label{prop:Tchar:conic:e}
		$ (1-\idparam)\|(\id-T)x-(\id-T)y\|^2
		\leq\idparam\|x-y\|^2
		-\idparam\|Tx-Ty\|^2$.	
	\end{enumerate}
\end{proposition}
\begin{proof}
	Apply \cref{prop:Tchar} with $(\idparam,\Nparam)$
	replaced by $(1-\idparam,\idparam)$.
\end{proof}
\begin{lemma}
	\label{lem:Young}
	Let $\lambda<1$.
	Then 
		\begin{equation}
	\|x\|^2-\lambda\|y\|^2\geq -\tfrac{\lambda}{1-\lambda}\|x+y\|^2.
	\end{equation}
\end{lemma}
\begin{proof}
	Let $\delta>0$.
	By Young's inequality, 
	$
	\|x+y\|^2=\|x\|^2+2\langle x,y\rangle+\|y\|^2\geq (1-\delta)\|x\|^2+(1-\delta^{-1})\|y\|^2
	$.
	Equivalently, 
	$\|x+y\|^2-(1-\delta)\|x\|^2\geq (1-\delta^{-1})\|y\|^2$.
	Now, replace $(x,y,\delta)$ by $(-y,x+y,1-\lambda)$. 
\end{proof}

\begin{proposition}
	\label{prop:why:conic}
	Let $\alpha\in \left]0,1\right[$,
	let $\beta >0$, and let $T\colon X\to X$. Then 
	$T$ is $\alpha$-averaged 
	if and only if $T=(1-\beta)\Id +\beta M$
	and
	$M$ is $\tfrac{\alpha}{\beta}$-conically nonexpansive. 
\end{proposition}
\begin{proof}
Indeed, $T$ is $\alpha$-averaged 
if and only if there exists a nonexpansive mapping $N\colon X\to X$
such that $T=(1-\alpha)\id+\alpha N$.
Equivalently,
	$$T=(1-\alpha)\id+\alpha N=(1-\beta )\Id
	+\beta \big((1-\tfrac{\alpha}{\beta})\Id+\tfrac{\alpha}{\beta}N\big),$$
	 and the conclusion follows by setting 
	 $M=\Big(1-\tfrac{\alpha}{\beta}\Big)\Id+\tfrac{\alpha}{\beta}N$.
\end{proof}

The following three lemmas can be directly verified,
hence we omit the proof.

\begin{lemma}
	\label{lem:av:coco}
	Let $\alpha >0$, and let $T\colon X\to X$.
	Then $T$ is $\alpha$-conically nonexpansive
	$\siff$ $\Id-T$ is $\tfrac{1}{2\alpha} $-cocoercive
	$\RA$ $\Id-T$ is maximally monotone.
\end{lemma}
\begin{lemma}
	\label{lem:beta:mu:coco}
	Let $\beta >0$, let $\mu\in \RR$,
	and let $A\colon X\to X$.
	Suppose that $A$ is maximally $\mu$-monotone
	and $\tfrac{1}{\beta} $-cocoercive. Then 
	$\mu\le \tfrac{1}{\beta} $.
\end{lemma}
\begin{lemma}
	\label{lem:beta:smaller}
	Let $\beta >0$, let $T\colon X\to X$,
	and let $\overline{\beta}\ge \beta$.
	Suppose that $T$ is $\tfrac{1}{\beta} $-cocoercive. Then 
	$T$ is $\tfrac{1}{\overline{\beta}} $-cocoercive. 
\end{lemma}
\begin{lemma}
	\label{lem:T:Lips:to:coco}
	Let $\beta>0$, 
	and let $A\colon X\to X$.
	Suppose that $A$ is ${\beta}$-Lipschitz continuous.
	Then 
	the following hold:
	\begin{enumerate}
		\item 
		\label{lem:T:Lips:to:coco:i}
		$A$ is maximally 
		$(-{\beta})$-monotone.
		\item 
		\label{lem:T:Lips:to:coco:ii}
		$A+{\beta}\Id$ is $\tfrac{1}{2\beta}$-cocoercive.
	\end{enumerate}
\end{lemma}
\begin{proof}
	See \cref{app:CD}.
\end{proof}
\begin{lemma}
	\label{lem:beta:delta:coco}
	Let $\beta>\delta >0$, let $T_1\colon X\to X$,
	and let  $T_2\colon X\to X$.
	Suppose that $T_1$ (respectively $T_2$) 
	is $\tfrac{1}{\beta} $-cocoercive 
	(respectively $\tfrac{1}{\delta}$-cocoercive). Then 
	$T_1-T_2$ is ${\beta}$-Lipschitz continuous. 
\end{lemma}
\begin{proof}
	See \cref{app:C}.
\end{proof}
As a corollary, we obtain the following result 
which was stated in \cite[page~4]{WCP17}.
\begin{corollary}
	\label{cor:f1:f2:lips}
	Let $f_1\colon X\to \RR$,
	$f_2\colon X\to \RR$
	be Frech\'et differentiable convex
	functions and let $\beta>\delta>0$.
	Suppose that $\grad f_1$ (respectively $\grad f_2$)
	is ${\beta}$-Lipschitz continuous  
	(respectively ${\delta}$-Lipschitz continuous).
	Then the following hold:
	\begin{enumerate}
		\item  
		\label{cor:f1:f2:lips:i}
		$\grad f_1-\grad f_2$
		is ${\beta}$-Lipschitz continuous.
		\item  
		\label{cor:f1:f2:lips:ii}
		Suppose that $f_1-f_2$ is convex. 
		Then
		$\grad f_1-\grad f_2$
		is $\tfrac{1}{\beta}$-cocoercive.
	\end{enumerate}
\end{corollary}
\begin{proof}
	See \cref{app:D}.
\end{proof}

\begin{lemma}
	\label{lem:T:av:factor}
	Let $\alpha\in \left]0,1\right[$, 
	let $\delta\in \left]0,1\right]$,
	and let $T\colon X\to X$.
	Suppose that $T$ is $\alpha$-averaged.
	Then the following hold:
	\begin{enumerate}
		\item 
		\label{lem:T:av:factor:i}
		$\delta T$ is 
		$(1-\delta(1-\alpha))$-averaged.
		\item 
		\label{lem:T:av:factor:ii}
		Suppose that $\delta\in \left]0,1\right[$.
		Then $\delta T$ is a Banach contraction with constant $\delta$.
	\end{enumerate}
\end{lemma}
\begin{proof}
	See \cref{app:E}.
\end{proof}

Let $A$ be maximally $\rho$-\emph{monotone}, where
$\rho>-1$. Then (see \cite[Proposition~3.4]{DP18}
and 
\cite[Corollary~2.11~and~Proposition~2.12]{BMW19})
we have 
\begin{equation}
\label{e:sv:fdom}
\text{$J_A$ is single-valued and $\dom J_A=X$.}
\end{equation}

The following result involves resolvents and reflected resolvents of
	$\rho$-monotone operators.

\begin{proposition}
	\label{prop:J:R:.5}
	Let $A$ be $\rho$-monotone, where 
	$\rho>-1$.
	Then the following hold:
	\begin{enumerate}
		\item 
		\label{prop:J:R:.5:i}
		$J_A$ is $(1+\rho)$- cocoercive,
		in which case $J_A$ is Lipschitz 
		continuous with constant $\tfrac{1}{1+\rho}$.
		\item 
		\label{prop:J:R:.5:ii}
		$-R_A$ is $\tfrac{1}{1+\rho}$-conically nonexpansive.
		\item 
		\label{prop:J:R:.5:iii}
		Suppose that $\rho\le 0$. Then
		$R_{A}$ is Lipschitz continuous with constant	
		$\tfrac{1- \rho}{1+\rho}  $.
	\end{enumerate}
\end{proposition}

\begin{proof}	
	\ref{prop:J:R:.5:i}:
	See \cite[Lemma~3.3(ii)]{DP18}.
	Alternatively, it follows from  
	\cite[Corollary~3.8(ii)]{BMW19}
	that $\Id-T$ is $\tfrac{1}{2(1+\rho)}$-averaged.
	Now apply  \cref{lem:av:coco} with
	$T$ replaced by $\Id-J_A$.
	\ref{prop:J:R:.5:ii}:
	It follows from \ref{prop:J:R:.5:i} that there exists a nonexpansive 
	operator $N\colon X\to X $ such that 
	$J_A=\tfrac{1}{2(1+\rho)}(\Id+N)$.
	Now,
	$-R_A=\Id-2J_A=\Id -\tfrac{1}{1+\rho}(\Id+N)
	=\big(1-\tfrac{1}{1+\rho}\big)\Id+\tfrac{1}{1+\rho}N$.
	\ref{prop:J:R:.5:iii}: Indeed, let 
	$(x,y)\in X\times X$ and let $N$ be as defined above. We have
	\begin{subequations}
		\begin{align}
\norm{R_Ax-R_Ay}&=\norm{-\tfrac{\rho}{1+\rho}(x-y)-\tfrac{1}{1+\rho}(Nx-Ny)}
\leq-\tfrac{\rho}{1+\rho}\norm{x-y}+\tfrac{1}{1+\rho}\norm{Nx-Ny}\\
&\leq\tfrac{1-\rho}{1+\rho}\norm{x-y}.
		\end{align}
	\end{subequations}
The proof is complete.
\end{proof}

\section{Compositions}
\label{sec:3}

\begin{definition}[$(\alpha,\beta)$-\IN]
\label{def:in:decomp}
    Let $R\colon X\to X$ be Lipschitz continuous,
    and let\footnote{Here and elsewhere
    we use $\RR_+$ to denote the interval
    $\left[0,+\infty\right[$.} $(\alpha,\beta)\in \RR\times \RR_+$.
    We say that $R$ admits an $(\alpha,\beta)$-Identity-Nonexpansive (I-N)
    decomposition if there exists a nonexpansive operator $N\colon X\to X$
    such that $R=\alpha \Id +\beta N$.
\end{definition}

Throughout the rest of this paper, we assume that
\begin{empheq}[box=\mybluebox]{equation*}
\label{T:assmp}
\text{$R_1\colon X\to X$ and $R_2\colon X\to X$ are Lipschitz continuous operators.}
\end{empheq}

\begin{proposition}
	\label{lem:two_comp}
	Let $\idparam_1\in \left]-\infty,1\right[$,
	let $\idparam_2\in \left]-\infty,1\right[$,
	let $\Nparam_1\in \RR_+$,
	let $\Nparam_2\in \RR_+$,
	and suppose that $\idparam_2(\idparam_2-1)\le \Nparam_2^2$.
	Set
	\begin{subequations}
		\label{eq:delta}
		\begin{align}
		\label{eq:delta_1}  
		\delta_1 &= \tfrac{\alpha_1}{1-\idparam_1}\Big(1-\tfrac{(1-\idparam_2)^2-\Nparam_2^2}{1-\idparam_2}\Big),
		\\
		\label{eq:delta_2} 
		\delta_2 &=\tfrac{\alpha_2}{1-\idparam_2},
		\\
		\label{eq:delta_3}  
		\delta_3 &=1-\Big(\tfrac{(1-\idparam_1)^2-\Nparam_1^2}{1-\idparam_1}
		\Big(1-\tfrac{(1-\idparam_2)^2-\Nparam_2^2}{1-\idparam_2}\Big)
		+\tfrac{(1-\idparam_2)^2-\Nparam_2^2}{(1-\idparam_2)}\Big).
		\end{align}
	\end{subequations}
	Suppose that $R_1$ admits an $(\idparam_1,\Nparam_1)$-\IN\ 
	and that $R_2$ admits an $(\idparam_2,\Nparam_2)$-\IN.
	Then  $(\forall (x,y)\in X \times X)$ we have
	\begin{multline}
	\norm{R_2R_1x-R_2R_1y}^2
	+\delta_1\norm{(\Id-R_1)x-(\Id-R_1)y}^2
	\\
	+\delta_2\norm{(\Id-R_2)R_1x-(\Id-R_2)R_1y}^2
	\leq\delta_3\|x-y\|^2.
	\end{multline}
\end{proposition}
\begin{proof}
	Set $T_i=\tfrac{1}{2}(\id+R_i)=\tfrac{1+\idparam_i}{2} \Id+\tfrac{\Nparam_i}{2} N_i$,
	and observe that by \cref{prop:Tchar} applied  with
	$(T, \alpha,\beta)$ replaced by $\big(T_i,\tfrac{1+\idparam_i}{2},\tfrac{\Nparam_i}{2}\big)$,
	$i\in \{1,2\}$, we have $(\forall (x,y)\in X\times X)$
	\begin{equation}
	\label{eq:nov:16:i*}
	\scal{T_ix-T_iy}{(\Id-T_i)x-(\Id-T_i)y}
	\ge 
	\tfrac{\alpha_i}{1-\idparam_i}\norm{(\Id-T_i)x-(\Id-T_i)y}^2
	+\tfrac{(1-\idparam_i)^2-\Nparam_i^2}{4(1-\idparam_i)}\norm{x-y}^2.	 
	\end{equation}
	Equivalently,
	\begin{equation}
	\label{eq:nov:16:i}
	\begin{split}
	\scal{(\Id+R_i)x-(\Id+R_i)y}{(\Id-R_i)x-(\Id-R_i)y}
	\qquad\qquad\qquad\qquad\qquad\qquad\qquad\qquad
	\\
	\ge 
	\tfrac{\alpha_i}{1-\idparam_i}\norm{(\Id-R_i)x-(\Id-R_i)y}^2
	+\tfrac{(1-\idparam_i)^2-\Nparam_i^2}{1-\idparam_i}\norm{x-y}^2.
	\end{split}	 
	\end{equation}
	Observe also that, because $\alpha_2<1$, we have  
	\begin{equation}
	\label{e:more:det:1}
	\idparam_2(\idparam_2-1)\le \Nparam_2^2
	\siff 1-\tfrac{(1-\idparam_2)^2-\Nparam_2^2}{1-\idparam_2}\ge 0.
	\end{equation}	
	It follows from 
	\cref{eq:nov:16:i}, applied with $i=2$ and $(x,y)$ 
	replaced by $(R_1x,R_1y)$ in \cref{eq:nov15:b}
	and  with $i=1$ in \cref{eq:nov15:e}
	below, in view of \cref{e:more:det:1} that
	\begin{subequations}
		\label{ineq:general}
		\begin{align}
		&\quad\norm{x-y}^2-\norm{R_2R_1x-R_2R_1y}^2
		\nonumber
		\\
		&=\norm{x-y}^2-\norm{R_1x-R_1y}^2+\norm{R_1x-R_1y}^2-\norm{R_2R_1x-R_2R_1y}^2
		\\
		&=\scal{(\Id+R_1)x-(\Id+R_1)y}{(\id-R_1)x-(\id-R_1)y}
		\nonumber
		\\
		&\quad+\scal{(\Id+R_2)R_1x-(\Id+R_2)R_1y}{(\id-R_2)R_1x-(\id-R_2)R_1y}
		\\
		%=================================== 1
		&\ge\scal{(\Id+R_1)x-(\Id+R_1)y}{(\id-R_1)x-(\id-R_1)y}
		+	 \tfrac{\alpha_2}{1-\idparam_2}\norm{(\Id-R_2)R_1x-(\Id-R_2)R_1y}^2
		\nonumber
		\\
		&\quad	
		+\tfrac{(1-\idparam_2)^2-\Nparam_2^2}{1-\idparam_2}\norm{R_1x-R_1y}^2
		\label{eq:nov15:b}	
		\\
		%=================================== 2
		&=\scal{(\Id+R_1)x-(\Id+R_1)y}{(\id-R_1)x-(\id-R_1)y}
		+	 \tfrac{\alpha_2}{1-\idparam_2}\norm{(\Id-R_2)R_1x-(\Id-R_2)R_1y}^2
		\nonumber
		\\
		&\quad	
		+\tfrac{(1-\idparam_2)^2-\Nparam_2^2}{1-\idparam_2}
		\big(\norm{x-y}^2-\scal{(\Id+R_1)x-(\Id+R_1)y}{(\id-R_1)x-(\id-R_1)y}\big)	
		\\
		%=================================== 2/2
		&=\Big(1-\tfrac{(1-\idparam_2)^2-\Nparam_2^2}{1-\idparam_2}\Big)
		\scal{(\Id+R_1)x-(\Id+R_1)y}{(\id-R_1)x-(\id-R_1)y}
		\nonumber
		\\
		&\quad	
		+\tfrac{\alpha_2}{1-\idparam_2}\norm{(\Id-R_2)R_1x-(\Id-R_2)R_1y}^2
		+\tfrac{(1-\idparam_2)^2-\Nparam_2^2}{1-\idparam_2}
		\norm{x-y}^2
		\\
		%=================================== 2/3
		&\ge\Big(1-\tfrac{(1-\idparam_2)^2-\Nparam_2^2}{1-\idparam_2}\Big)
		\Big(	 \tfrac{\alpha_1}{1-\idparam_1}\norm{(\Id-R_1)x-(\Id-R_1)y}^2
		+\tfrac{(1-\idparam_1)^2-\Nparam_1^2}{1-\idparam_1}\norm{x-y}^2	 \Big)
		\nonumber
		\\
		&\quad	
		+	 \tfrac{\alpha_2}{1-\idparam_2}\norm{(\Id-R_2)R_1x-(\Id-R_2)R_1y}^2
		+\tfrac{(1-\idparam_2)^2-\Nparam_2^2}{1-\idparam_2}
		\norm{x-y}^2
		\label{eq:nov15:e}
		\\
		%=================================== 2/4
		&=
		\tfrac{\alpha_1}{1-\idparam_1}\Big(1-\tfrac{(1-\idparam_2)^2-\Nparam_2^2}{1-\idparam_2}\Big)
		\norm{(\Id-R_1)x-(\Id-R_1)y}^2
		+\tfrac{\alpha_2}{1-\idparam_2}\norm{(\Id-R_2)R_1x-(\Id-R_2)R_1y}^2
		\nonumber
		\\
		&\quad+\Big(\tfrac{(1-\idparam_1)^2-\Nparam_1^2}{1-\idparam_1}
		\Big(1-\tfrac{(1-\idparam_2)^2-\Nparam_2^2}{1-\idparam_2}\Big)
		+\tfrac{(1-\idparam_2)^2-\Nparam_2^2}{1-\idparam_2}\Big)	\norm{x-y}^2.
		\end{align}
	\end{subequations}
	Rearranging yields the desired result.
\end{proof}

\begin{theorem}
	\label{theorem:gen:alpha:beta}
	Let $\idparam_1\in \left]-\infty,1\right[$,
	let $\idparam_2\in \left]-\infty,1\right[$,
	let $\Nparam_1\in \RR_+$,
	let $\Nparam_2\in \RR_+$,
	and suppose that $\idparam_2(\idparam_2-1)\le \Nparam_2^2$.
	Let $\delta_1$, $\delta_2$, and $\delta_3$ be 
	defined as in \cref{eq:delta}.
	Set
	\begin{align}
	\delta_4 = \tfrac{\delta_1\delta_2}{\delta_1+\delta_2},
	\label{eq:delta_4}
	\end{align}
	and suppose that $\delta_1+\delta_2>0$, 
	that $\delta_3-\delta_4+\delta_3\delta_4\ge 0$,
	and that $\delta_4>-1$.
	Suppose that $R_1$ admits an $(\idparam_1,\Nparam_1)$-\IN, 
	and that $R_2$ admits an $(\idparam_2,\Nparam_2)$-\IN.
	Then, $R_2R_1$ admits an $(\idparam,\Nparam)$-\IN,
	where
	\begin{equation}
	\idparam=\tfrac{\delta_4}{1+\delta_4},
	\qquad 
	\Nparam=\tfrac{\sqrt{\delta_3-\delta_4+\delta_3\delta_4}}{1+\delta_4}.
	\end{equation}
\end{theorem}
\begin{proof}
	Let $\underline{\delta}:=\min(\delta_1,\delta_2)$, 
	let $\bar{\delta}:=\max(\delta_1,\delta_2)$, 
	and let $\lambda:=-\underline{\delta}/\bar{\delta}$ 
	(i.e., $\lambda=-\delta_1/\delta_2$ if $\delta_1\leq\delta_2$, 
	and $\lambda=-\delta_2/\delta_1$ if $\delta_1\geq\delta_2$). 
	\cref{lem:two_comp} then \cref{lem:Young} imply that
	\begin{subequations}
		\label{eq:comp_ineq}  
		\begin{align}
		&\quad\delta_3\norm{x-y}^2-\norm{R_2R_1x-R_2R_1y}^2
		\nonumber\\
		&\geq\delta_1\norm{(\Id-R_1)x-(\Id-R_1)y}^2+\delta_2\norm{(\Id-R_2)R_1x-(\Id-R_2)R_1y}^2
		\\
		&=\bar{\delta}(\tfrac{\delta_1}{\bar{\delta}}\norm{(\Id-R_1)x-(\Id-R_1)y}^2
		+\tfrac{\delta_2}{\bar{\delta}}\norm{(\Id-R_2)R_1x-(\Id-R_2)R_1y}^2
		\\
		&\geq\bar{\delta}(-\tfrac{\lambda}{1-\lambda}\norm{(\Id-R_1)x-(\Id-R_1)y+(\Id-R_2)R_1x-(\Id-R_2)R_1y}^2)
		\\
		&=-\tfrac{\lambda\bar{\delta}}{1-\lambda}\norm{(\id-R_2R_1)x-(\id-R_2R_1)y}^2
		\\
		&=\tfrac{\underline{\delta}\bar{\delta}}{\bar{\delta}+\underline{\delta}}
		\norm{(\id-R_2R_1)x-(\id-R_2R_1)y}^2
		\\
		&=\delta_4\norm{(\id-R_2R_1)x-(\id-R_2R_1)y}^2.                                                       
		\end{align}
	\end{subequations}
	Comparing \cref{eq:comp_ineq} to \cref{prop:Tchar}
	applied with $T$ replaced by $R_2R_1$, we learn that
	there exists a nonexpansive operator $N\colon X\to X$
	and $(\alpha,\beta)\in \RR^2$ 
	such that 
	$R_2R_1=\idparam \Id+\Nparam N$, 
	where 
	$\delta_3=\tfrac{\Nparam^2+\idparam(1-\idparam)}{1-\idparam}$
	and
	$\delta_4=\tfrac{\idparam}{1-\idparam}$.
	Equivalently, 
	$ \idparam=\tfrac{\delta_4}{1+\delta_4}$,
	hence,
	$\Nparam=\tfrac{\sqrt{\delta_3-\delta_4+\delta_3\delta_4}}{1+\delta_4}$, as claimed.
\end{proof}

\begin{theorem}
	\label{thm:comp:more:gen}
	Let $\idparam_1\in \RR$,
	let $\idparam_2\in \RR$,
	let $\Nparam_1>0$,
	let $\Nparam_2>0$, 
	suppose 
	that ${\idparam_1}+\Nparam_1>0$,
	that ${\idparam_2}+\Nparam_2>0$,
	and that either
	$\tfrac{\beta_1\beta_2}{(\alpha_1+\beta_1)(\alpha_2+\beta_2)}<1$
	or 
	$\max\big\{\tfrac{\beta_1}{\alpha_1+\beta_1},
	\tfrac{\beta_2}{\alpha_2+\beta_2}\big\}=1$.
	Set
	\begin{subequations}
		\label{def:kappa:theta}
		\begin{align}
		\kappa
		&=(\alpha_1+\beta_1)(\alpha_2+\beta_2)
		\\
		\theta
		&=
		\begin{cases}
		\frac{\beta_1{\alpha_2}+\beta_2{\alpha_1}}{
			{\alpha_1}{\alpha_2}+{\alpha_1}\beta_2+{\alpha_2}\beta_1},
		&
		\frac{\beta_1\beta_2}{({\alpha_1}+\beta_1)({\alpha_2}+\beta_2)}<1;
		\\
		1,
		&
		\max\big\{\frac{\beta_1}{{\alpha_1}+\beta_1},
		\frac{\beta_2}{{\alpha_2}+\beta_2}\big\}=1.
		\end{cases}
		\end{align}
	\end{subequations}
	Suppose that $R_1$ admits an $(\idparam_1,\Nparam_1)$-\IN, 
	and that $R_2$ admits an $(\idparam_2,\Nparam_2)$-\IN.
	Then $\theta\in \left]0,+\infty\right[$ and 
	$R_2R_1$ admits a $(\kappa(1-\theta),\kappa\theta)$-\IN, i.e.,
	$R_2R_1$ is $\kappa$-scaled $\theta$-conically nonexpansive.
	That is, there exists a nonexpansive operator $N\colon X\to X$
	such that
	\begin{equation}
    R=\kappa(1-\theta)\Id+\kappa\theta N.
	\end{equation}
\end{theorem}
\begin{proof}
	Let $\theta_i=\tfrac{\beta_i}{{\alpha_i}+\beta_i}>0$,
	and observe that 
	\begin{equation}
	\label{eq:191125a}
	R_i=(\alpha_i+\beta_i)\big((1-\theta_i)\Id+\theta_iN_i\big),
	\quad i\in \{1,2\}.
	\end{equation}
	Next, let 
	$\widetilde{N}_2=\tfrac{1}{\alpha_1+\beta_1}N_2\circ (\alpha_1+\beta_1)\Id$,
	and note that $\widetilde{N}_2$ is nonexpansive.
	Now, set 
	\begin{equation}
	\label{eq:191125b}
	\widetilde{R}_1=(1-\theta_1)\Id+\theta_1N_1,
	\qquad
	\widetilde{R}_2=(1-\theta_2)\Id+\theta_2\widetilde{N}_2.
	\end{equation}	
	Then \cref{eq:191125a} and \cref{eq:191125b}
	yield
	\begin{subequations}
		\label{e:191125d}
		\begin{align}
		R_2R_1&=\Big(({\alpha_2}+\beta_2)((1-\theta_2)\id+\theta_2 N_2)\Big)
		\Big(\big({\alpha_1}+\beta_1)((1-\theta_1)\id+\theta_1 {N}_1\big)\Big)\\
		&=({\alpha_1}+\beta_1)({\alpha_2}+\beta_2)
		\Big(\tfrac{1}{{\alpha_1}+\beta_1}((1-\theta_2)\id+\theta_2 {N}_2)\Big)
		\Big(({\alpha_1}+\beta_1)
		 \widetilde{R}_1\Big)\\
		&=({\alpha_1}+\beta_1)({\alpha_2}+\beta_2)\widetilde{R}_2 \widetilde{R}_1.
		\end{align}
	\end{subequations}
	We proceed by cases.	
		\textsc{Case~I}: $\alpha_1\alpha_2=0$.
	Observe that $0\in \{\alpha_1,\alpha_2\}$ $\siff$
	$\max\big\{\tfrac{\beta_1}{{\alpha_1}+\beta_1}
	,\tfrac{\beta_2}{{\alpha_2}+\beta_2}\big\}=\max\{\theta_1,\theta_2\}=1$.
	The conclusion follows by observing that
	$\widetilde{R}_i$ is nonexpansive, $i\in \{1,2\}$.
	
	\textsc{Case~II}: $\alpha_1\alpha_2\neq 0$. 
    By assumption we must have $\tfrac{\beta_1}{{\alpha_1}+\beta_1}
	\tfrac{\beta_2}{{\alpha_2}+\beta_2}=\theta_1\theta_2<1$.
	We claim that $\widetilde{R}_i$, $i\in \{1,2\}$, satisfy the conditions of  
	\cref{theorem:gen:alpha:beta} 
	with $(\idparam_i,\Nparam_i)$ replaced by $(1-\theta_i,\theta_i)$.
	Indeed, observe that 
	$(1-\theta_2)(1-\theta_2-1)\le \theta_2^2
	\siff
	\theta_2(\theta_2-1)\le \theta_2^2
	\siff \theta_2-1\le\theta_2$, which is always true.
	Moreover, replacing $(\idparam_i,\Nparam_i)$ by $(1-\theta_i,\theta_i)$
	yields
	$\delta_1=\tfrac{1-\theta_1}{\theta_1}$,
	$\delta_2=\tfrac{1-\theta_2}{\theta_2}$,
	$\delta_3=1$,
	and, consequently, 
	$\delta_4=\tfrac{\theta_2(1-\theta_1)+\theta_1
		(1-\theta_2)}{(1-\theta_1)(1-\theta_2)}$.
	We claim that
	\begin{equation}
	\label{eq:aux:help}
	\theta_1+\theta_2-2\theta_1\theta_2>0.
	\end{equation}
	Indeed, recall that 
	$\theta_1+\theta_2-2\theta_1\theta_2
	=\theta_1\theta_2(\tfrac{1}{\theta_1}+\tfrac{1}{\theta_2}-2)
	>\theta_1\theta_2(\tfrac{1}{\theta_1}+\theta_1-2)
	=\theta_1\theta_2\Big(\sqrt{\theta_1}-\tfrac{1}{\sqrt{\theta_1}}\Big)^2>0$.
	This implies that $\delta_1+\delta_2=\tfrac{\theta_1
		+\theta_2-2\theta_1\theta_2}{\theta_1\theta_2}>0$.
	Moreover,
	\begin{equation}
	\delta_4=\tfrac{(1-\theta_1)(1-\theta_2)}{\theta_2(1-
		\theta_1)+\theta_1(1-\theta_2)}
	=\tfrac{1-\theta_1-\theta_2+\theta_1\theta_2}{\theta_1
		+\theta_2-2\theta_1\theta_2}
	=-1+\tfrac{1-\theta_1\theta_2}{\theta_1+\theta_2-2\theta_1\theta_2}>-1.
	\end{equation} 
	Therefore, by \cref{theorem:gen:alpha:beta}, we conclude that
	there exists a nonexpansive operator $N\colon X\to X$,
	such that 
	$\widetilde{R}_2\widetilde{R}_1=\alpha\Id+\beta N$, 
	$\alpha=\tfrac{\delta_4}{1+\delta_4}
	=\tfrac{1-\theta_1-\theta_2+\theta_1\theta_2}{1-\theta_1\theta_2}
	=\frac{{\alpha_1}{\alpha_2}}{
		{\alpha_1}{\alpha_2}+{\alpha_1}\beta_2+{\alpha_2}\beta_1}$,
	and 
	$\beta=\tfrac{1}{1+\delta_4}
	=\tfrac{\theta_1+\theta_2-2\theta_1\theta_2}{1-\theta_1\theta_2}
	=\frac{\beta_1{\alpha_2}+\beta_2{\alpha_1}}{
		{\alpha_1}{\alpha_2}+{\alpha_1}\beta_2+{\alpha_2}\beta_1}$.
	Now combine with \cref{e:191125d}.
\end{proof}

\section{Applications to special cases}
\label{sec:4}

We start this section by recording the following simple lemma which 
can be easily verified, hence we omit the proof.
\begin{lemma}
	\label{lem:comp:neg}
	Set $(\widetilde{R}_1,\widetilde{R}_2)=(-R_1,R_2\circ (-\Id))$.
	Then the following hold:
	\begin{enumerate}
		\item
		\label{lem:comp:neg:i}
		$R_2R_1=\widetilde{R}_2\widetilde{R}_1$.	
		\item 
		\label{lem:comp:neg:ii}
		Let $\alpha_i>0$, let $\delta_i\in \RR\smallsetminus\{0\}$ and suppose that
		$-\tfrac{1}{\delta_i}R_i$ is $\alpha_i$-conically nonexpansive.
		Then $\tfrac{1}{\delta_i}\widetilde{R}_i$
		is $\alpha_i$-conically nonexpansive.
		
	\end{enumerate}
\end{lemma}

\begin{theorem}
	\label{thm:conic:conic}
	Let $i\in \{1,2\}$, let 
	$\alpha_i>0$, let $\delta_i\in \RR\smallsetminus\{0\}$, 
	let $R_i\colon X\to X$ be 
	such that $\tfrac{1}{\delta_i}R_i$ is 
	$\alpha_i$-conically nonexpansive.
	Suppose that either $\alpha_1\alpha_2< 1$
	or $\max\{\alpha_1,\alpha_2\}=1$.
	Set 
	\begin{equation}
	R=R_2R_1,\quad \alpha
	=
	\begin{cases}
	\tfrac{\alpha_1+\alpha_2-2\alpha_1\alpha_2}{1-\alpha_1\alpha_2},
	&\alpha_1\alpha_2<1;
	\\
	1, 
	&\max\{\alpha_1,\alpha_2\}=1.
	\end{cases}
	\end{equation}
	Then there exists a nonexpansive
	operator $N\colon X\to X$ such that
	\begin{equation} 
	R=\delta_1\delta_2\big((1-\alpha)\Id+\alpha N\big).
	\end{equation} 
	Furthermore, $\alpha<1$
	$\siff$ {\rm[}$\alpha_1<1$ and $\alpha_2<1${\rm]}.
\end{theorem}
\begin{proof}
	Set $(\widetilde{R}_1,\widetilde{R}_2)=(-R_1,R_2\circ (-\Id))$.
	The proof proceeds by cases.

	\textsc{Case~I:}
	$\delta_i>0$, $i\in \{1,2\}$.
	By assumption,
	there exist
	nonexpansive operators
	$N_i\colon X\to X$
	such that
	$R_i=\delta_i(1-\alpha_i)\Id+\delta_i\alpha_i N_i$.
	Moreover, one can easily check that
	$R_i$ satisfy the assumptions of \cref{thm:comp:more:gen}
	with $(\alpha_i,\beta_i)$ replaced by
	$(\delta_i(1-\alpha_i),\delta_i\alpha_i)$.
	Applying \cref{thm:comp:more:gen},
	with $(\alpha_i,\beta_i)$ replaced by 
	$(\delta_i(1-\alpha_i),\delta_i\alpha_i)$,
	we learn that 
	there exists a nonexpansive operator ${N}\colon X\to X$
	such that 
	$R_2R_1=(\delta_1(1-\alpha_1)+\delta_1\alpha_1)(\delta_2(1-\alpha_2)+\delta_2\alpha_2)
	((1-\alpha)\Id+\alpha {N})
	=\delta_1\delta_2((1-\alpha)\Id+\alpha {N})$,
	where 
	\begin{equation}
	\alpha =\frac{\delta_1(1-\alpha_1)\delta_2\alpha_2+\delta_2(1-\alpha_2)\delta_1\alpha_1}{
		\delta_1(1-\alpha_1)\delta_2\alpha_2
		+\delta_2(1-\alpha_2)\delta_1\alpha_1+\delta_1(1-\alpha_1)\delta_2(1-\alpha_2)}
	=\frac{\alpha_1+\alpha_2-2\alpha_1\alpha_2}{1-\alpha_1\alpha_2}.
	\end{equation}
	Finally,
	observe that 
	$\alpha<1$ $\siff$ [$\alpha_1\alpha_2<1$
	and $\tfrac{\alpha_1+\alpha_2-2\alpha_1\alpha_2}{1-\alpha_1\alpha_2}<1$]
	$\siff$ [$\alpha_1\alpha_2<1$
	and $1-\alpha_1\alpha_2>\alpha_1+\alpha_2-2\alpha_1\alpha_2$]
	$\siff$ [$\alpha_1\alpha_2<1$
	and
	$(1-\alpha_1)(1-\alpha_2)>0$]
	$\siff$ [$\alpha_1<1$ and $\alpha_2<1$].
	
	\textsc{Case~II:}
	$\delta_i<0$, $i\in \{1,2\}$.
	Observe that $\tfrac{1}{\delta_i}R_i=-\tfrac{1}{\abs{\delta_i}}R_i$ 
	is $\alpha_i$-conically nonexpansive.
	Therefore,  \cref{lem:comp:neg}\ref{lem:comp:neg:ii},
	applied with $\delta_i$ replaced by $\abs{\delta_i}$,
	implies that
	$\tfrac{1}{\abs{\delta_i}}\widetilde{R}_i$ are $\alpha_i$-conically nonexpansive. 
	Now 
	combine \cref{lem:comp:neg}\ref{lem:comp:neg:i}
	and
	\textsc{Case~I} applied with 
	$(R_i,\delta_i)$ replaced by $(\widetilde{R}_i,\abs{\delta_i})$.

	\textsc{Case~III:}
	$\delta_1<0$ and $\delta_2>0$:
	Observe that 
	$\tfrac{1}{\delta_1}R_1=-\tfrac{1}{\abs{\delta_1}}R_1$ 
	is $\alpha_1$-conically nonexpansive.
	Now, using \cref{lem:comp:neg}\ref{lem:comp:neg:i}\&\ref{lem:comp:neg:ii}
	we have  $-R=-R_2R_1=-\widetilde{R}_2\widetilde{R}_1$, and 
	$-\tfrac{1}{\delta_2}\widetilde{R}_2$ is $\alpha_2$-conically nonexpansive.
	Now combine with \textsc{Case~II}, applied with 
	$(R_1,R_2,\delta_1)$ replaced by $(\widetilde{R}_1,-\widetilde{R}_2,\abs{\delta_1})$,
	to learn that there exists a nonexpansive mapping
	$N\colon X\to X$ such that 
	$-R=\abs{\delta_1}\delta_2((1-\alpha)\Id+\alpha N)$,
	 and the conclusion follows.
	
	\textsc{Case~IV:}
	$\delta_1>0$ and $\delta_2<0$:
	Indeed, $-R=-R_2R_1$.
	Now combine with \textsc{Case~I} applied with
	$R_2$ replaced by $-R_2$, in view of 
	\cref{lem:comp:neg}\ref{lem:comp:neg:ii}.
	\end{proof}

\begin{corollary}
	\label{cor:scav:coco}
	Let $\alpha\in \left]0,1\right[$,
	let $\beta>0$,
	let $\delta\in \RR\smallsetminus\{0\}$,
	let $\{i,j\}=\{1,2\}$,
	and suppose that 
	$\tfrac{1}{\delta}R_i$ is 
	$\alpha$-averaged,
	and that $R_j$ is $\tfrac{1}{\beta}$-cocoercive.
	Set  $\overline{\alpha}=\tfrac{1}{2-\alpha}$.
	Then $\overline{\alpha}\in \left]0,1\right[$,
	and there exists a nonexpansive operator $N\colon X\to X$ such that 
	\begin{equation}
	R_2R_1=\beta\delta\big((1-\overline{\alpha})\Id+\overline{\alpha}N\big).
	\end{equation}
\end{corollary}
\begin{proof}
	Suppose first that $(i,j)=(1,2)$, and observe that 
	there exists a nonexpansive operator $\overline{N}$
	such that 
	$R_2=\tfrac{\beta}{2}(\Id+\overline{N})$.
	Applying \cref{thm:m:conic}
	with $m=2$,
	$(\alpha_1,\alpha_2,\delta_1,\delta_2)$
	replaced by 
	$(\alpha,1/2,\delta,\beta)$
	yields that there exists a nonexpansive operator $N$ such that 
	$R_2R_1=\beta\delta\big((1-\overline{\alpha})\Id+\overline{\alpha} N\big)$,
	where 
	\begin{equation}
	\overline{\alpha}
	=\frac{\alpha+\tfrac{1}{2}-2\tfrac{\alpha}{2}}{1-\tfrac{\alpha}{2}}
	=\frac{1}{2-\alpha}\in \left]0,1\right[.
	\end{equation} 
	The case $(i,j)=(2,1)$ follows similarly.
\end{proof}

The assumption $\alpha_1\alpha_2<1$ is critical in the conclusion of 
\cref{thm:conic:conic} as we illustrate below.

\begin{example}[$\alpha_1=\alpha_2>1$]
	Let $\alpha>1$,
	and set $R_1=R_2=(1-\alpha)\id-\alpha\id=(1-2\alpha)\id$. Then
	\begin{equation}
	R_2R_1 = (1-2\alpha)^2\id=(1-4\alpha+4\alpha^2)\id.
	\end{equation}
	Hence, $\Id-R_2R_1=4\alpha(1-\alpha)\Id$.
	That is,  $\Id-R_2R_1$ is not monotone;
	hence, $R_2R_1$ is \emph{not} conically nonexpansive by 
	\cref{lem:av:coco} applied with $T$ replaced by $R_2R_1$.
\end{example}

The following proposition provides an abstract framework to
construct a family of operators $R_1$ and $R_2$
such that
$R_1$ is $\alpha_1$-conically nonexpansive,
$R_2$ is $\alpha_2$-conically nonexpansive,
 $\alpha_1\alpha_2>1$,
and the composition $R_2R_1$ fails to be conically nonexpansive.
\begin{proposition}
	\label{prop:counter:ab>1}
	Let $\theta\in \RR$,
	let $\alpha_1>0$, let $\alpha_2>0$,
	let 
	\begin{equation}
		\label{eq:def:Rtheta}
	R_\theta
	=\begin{bmatrix}
	\cos \theta &-\sin \theta\\
	\sin \theta &\cos \theta
	\end{bmatrix},
	\end{equation}
	set 
	\begin{equation}
	R_1=(1-\alpha_1)\Id+\alpha_1 R_\theta,
	\quad 
	R_2=(1-\alpha_2)\Id-\alpha_2 R_\theta,
	\quad 
	R=R_2R_1,
	\end{equation}
	and set 
		\begin{equation}
			\label{eq:def:kappa}
		\kappa=\alpha_1+\alpha_2-2\alpha_1\alpha_2\sin^2 \theta-(\alpha_1-\alpha_2)\cos \theta.
		\end{equation}
	Then 
	$R_1$ is $\alpha_1$-conically nonexpansive,
	and
	$R_2$ is $\alpha_2$-conically nonexpansive. Moreover, we have the implication
	$\kappa<0$ $\RA$ $R$ is \emph{not} conically nonexpansive.
	%	Consequen, if 	\item
	%	$\kappa_1>0$
	%	and $\kappa_2\le0$,  
	%	then 	$R$ is \emph{not}
	%	conically nonexpansive.
\end{proposition}
\begin{proof}
	Set $S=R_{\pi/2}$, and observe that $S^2=-\Id$, and that 
	$R_\theta=(\cos \theta) \Id+(\sin\theta) S$.
	Now,
	\begin{subequations}
		\begin{align}
		R
		&=R_2R_1=((1-\alpha_1)\Id+\alpha_1 R_\theta )
		((1-\alpha_2)\Id-\alpha_2 R_\theta)
		\\
		&=(1-\alpha_1-\alpha_2+\alpha_1\alpha_2)\Id
		+(\alpha_1-\alpha_2)R_\theta -\alpha_1 \alpha_2 R_{2\theta}
		\\
		&=(1-\alpha_1-\alpha_2+\alpha_1\alpha_2+(\alpha_1-\alpha_2)\cos \theta
		-\alpha_1\alpha_2\cos(2\theta) )\Id
		\nonumber
		\\
		&\quad+((\alpha_1-\alpha_2)\sin \theta -\alpha_1\alpha_2\sin(2\theta))S
		\\
		&=(1-\alpha_1-\alpha_2+\alpha_1\alpha_2+(\alpha_1-\alpha_2)\cos \theta
		-\alpha_1\alpha_2 (2\cos^2\theta-1) )\Id
		\nonumber
		\\
		&\quad+((\alpha_1-\alpha_2)\sin \theta -\alpha_1\alpha_2\sin(2\theta))S
		\\
		&=(1-\alpha_1-\alpha_2+2\alpha_1\alpha_2\sin^2 \theta+(\alpha_1-\alpha_2)\cos \theta
		)\Id
		+((\alpha_1-\alpha_2)\sin \theta -\alpha_1\alpha_2\sin(2\theta))S.
		\end{align}
	\end{subequations}
	Consequently,
	\begin{equation}
	\Id-R =(\alpha_1+\alpha_2-2\alpha_1\alpha_2\sin^2 \theta-(\alpha_1-\alpha_2)\cos \theta)\Id
	-((\alpha_1-\alpha_2)\sin \theta -\alpha_1\alpha_2\sin(2\theta))S.
	\end{equation}
	Hence, $(\forall x\in \RR^2)$
	\begin{equation}
	\label{eq:not:mono}
	\scal{(\Id-R)x}{x}=(\alpha_1+\alpha_2-2\alpha_1\alpha_2\sin^2 \theta-(\alpha_1-\alpha_2)\cos \theta)\norm{x}^2=\kappa\norm{x}^2.
	\end{equation}
	Now,
	$R$ is conically nonexpansive
	$\RA$ $\Id-R$ is monotone by \cref{lem:av:coco},
	and the conclusion follows in view of \cref{eq:not:mono}. 
\end{proof}

The following example provides two concrete instances where:
\begin{enumerate*}
	\item 
	$\alpha_1>1$, $\alpha_2>1$, hence $\alpha_1\alpha_2>1$,
	\item 
	$\alpha_1>1$, $\alpha_2<1$, $\alpha_1\alpha_2>1$.
\end{enumerate*}	
In both cases, $R_2R_1 $ is \emph{not} conically nonexpansive. 
\begin{example}
	\label{ex:cases}
	Suppose that one of the following holds:
	\begin{enumerate}
		\item
		\label{ex:cases:i}
		$\theta \in \left ]0,\pi/2\right[$,
		$ \epsilon\ge 0$, $\delta\ge 0$, 
		$\alpha_1=\tfrac{1+\epsilon}{\sin^2\theta}$
		and 
		$\alpha_2=\tfrac{1+\delta}{\sin^2\theta}$.
		\item
		\label{ex:cases:ii}
		$\theta \in \left ]\pi/4,\pi/2\right[$,
		$\epsilon> \tfrac{\cos^2\theta(2-\cos^2\theta)}{(1-2\cos^2\theta)(1+\cos\theta)+\cos\theta}$,
		$\alpha_1=\tfrac{1+\epsilon}{\sin^2\theta}$
		and 
		$\alpha_2={\sin^2\theta}$.	
	\end{enumerate}
	Let
	$R_\theta$ be defined as in \cref{eq:def:Rtheta},
	let
	$R_1=(1-\alpha_1)\Id+\alpha_1 R_\theta$,
	let 
	$R_2=(1-\alpha_2)\Id-\alpha_2 R_\theta$,
	and let $R=R_2R_1$. Then
	$\alpha_1\alpha_2>1$, and
	$R$ is \emph{not}
	conically nonexpansive.
\end{example}	
\begin{proof}
	Let $\kappa$ be defined as in
	\cref{eq:def:kappa}. In view of 
	\cref{prop:counter:ab>1}, it is sufficient to show that
	$\kappa<0$.
	\ref{ex:cases:i}:	
	Note that $\kappa<0$ $\siff$
	$\kappa\sin^2 \theta<0$.
	Now,
	\begin{subequations}
		\begin{align}
		\kappa\sin^2 \theta
		&=2+\epsilon+\delta-(\epsilon-\delta)\cos\theta
		-2-2 \epsilon-2\delta-2 \epsilon\delta
		\\
		&=-(\epsilon(1+\cos\theta)+\delta(1-\cos\theta)+2 \epsilon\delta)<0.
		\end{align}
	\end{subequations}
	\ref{ex:cases:ii}:	
	We have 
	\begin{subequations}
		\begin{align}
		\kappa
		&=\tfrac{1+\epsilon+\sin^4\theta}{\sin^2\theta}-2(1+\epsilon)\sin^2\theta
		-\tfrac{1+\epsilon-\sin^4\theta}{\sin^2\theta}\cos \theta
		\\
		&=-\tfrac{1}{\sin^2\theta}\big(2(1+\epsilon)\sin^4\theta
		-(1+\epsilon+\sin^4\theta)+(1+\epsilon-\sin^4\theta)\cos\theta
		\big)
		\\
		&=-\tfrac{1}{1-\cos^2\theta}\big((2\sin^4\theta+\cos\theta-1)\epsilon
		+\sin^4\theta(1-\cos \theta)
		-(1-\cos\theta)
		\big)
		\\
		&=-\tfrac{1-\cos\theta}{1-\cos^2\theta}\big((2(1+\cos \theta)(1-\cos^2\theta)-1)\epsilon
		+1-2\cos^2\theta+\cos^4\theta
		-1
		\big)
		\\
		&=-\tfrac{1}{1+\cos\theta}\big((1+2\cos\theta-2\cos^2\theta-2\cos^3 \theta)\epsilon
		-\cos^2\theta(2-\cos^2\theta)
		\big)
		\\
		&=-\tfrac{1}{1+\cos\theta}\big((1-2\cos^2\theta)(1+\cos\theta)+\cos\theta)\epsilon
		-\cos^2\theta(2-\cos^2\theta)
		\big).
		\end{align}	
	\end{subequations}
	Now, observe that 
	$\big(\forall \theta \in \left]\tfrac{\pi}{4},\tfrac{\pi}{2}\right[\big)$
	$1-2\cos^2\theta=-\cos(2\theta) >0$. 
    Consequently,
	$(1-2\cos^2\theta)(1+\cos\theta)+\cos\theta>\cos\theta>0$.
	Now use the assumption 
	$\epsilon> \tfrac{\cos^2\theta(2-\cos^2\theta)}{(1-2\cos^2\theta)(1+\cos\theta)+\cos\theta}$
	to learn that 
	$(1-2\cos^2\theta)(1+\cos\theta)+\cos\theta)\epsilon
		-\cos^2\theta(2-\cos^2\theta)>0$,
	hence $\kappa<0$, and the conclusion follows.
\end{proof}

\begin{theorem}[composition of $m$ scaled conically nonexpansive operators]
	\label{thm:m:conic}
	Let $m\ge 2$ be an integer,
	set $I=\{1,\ldots,m\}$,
	let $(R_i)_{i\in I}$ be a family of operators from $\hilbert$
	to $\hilbert$,
	let $r\in I$,
	let $\alpha_i$ be real numbers 
	such that $\menge{\alpha_i}{i\in 
		I\smallsetminus\{r\} }\subseteq 
	\left]0,1\right[$, and $\alpha_r>0$,
	let $\delta_i$ be real numbers 
	in $\RR\smallsetminus\{0\}$,
	and suppose that for every $i\in I$, $\tfrac{1}{\delta_i}R_i$
	is $\alpha_i$-conically nonexpansive.
	Set 
	\begin{equation}
	R=R_m\ldots R_1,
	\qquad
	\text{and}
	\qquad
	\overline{\alpha}=\frac{\sum_{\substack{i=1 \\ i\neq \overline{r}}}^{m}
		\tfrac{\alpha_i}{1-\alpha_i}	
	}{1+\sum_{\substack{i=1 \\ i\neq \overline{r}}}^{m}\tfrac{\alpha_i}{1-\alpha_i}}.
	\end{equation}
	Suppose that $\alpha_r\overline{\alpha}<1$, and set
	\begin{equation}
	\alpha=	
	\begin{cases}
	\frac{\sum_{i=1 
		}^{m}
		\tfrac{\alpha_i}{1-\alpha_i}	
	}{1+\sum_{i=1 }^{m}\tfrac{\alpha_i}{1-\alpha_i}},
	& \alpha_r\neq 1;
	\\
	1,
	&
	\alpha_r=1.
	\end{cases}
	\end{equation}
	Then there exists a nonexpansive operator $N\colon X\to X$ such that
	\begin{equation}
	R=\delta_m\ldots\delta_1((1-\alpha)\Id+\alpha N).
	\end{equation} 
\end{theorem}	
\begin{proof}
First, observe that $(\forall i\in I\smallsetminus\{r\})$,
$\tfrac{1}{\delta_i}R_i$ is nonexpansive.
	If $\alpha_r=1$ then $(\forall i\in \{1,\ldots,m\})$
	$R_i$ is ${\abs{\delta_i}}$-Lipschitz continuous and the conclusion 
	readily follows.
	Now, suppose that $\alpha_r\neq 1$.
	We proceed by induction on $k\in \{2,\ldots,m\}$.
	At $k=2$, the claim 
	holds by \cref{thm:conic:conic}. 
	Now, suppose that the claim holds for some $k\in \{2,\ldots,m-1\}$.
	Let $(R_i)_{1\le i\le k+1}$ 
	be a family of operators from $\hilbert$
	to $\hilbert$,
	let $\overline{r}\in \{1,\ldots,k,k+1\}$,
	let $\alpha_i$ be real numbers 
	such that $\menge{\alpha_i}{i\in 
		\{1,\ldots,k,k+1\}\smallsetminus\{\overline{r}\} }\subseteq 
	\left]0,1\right[$, and $\alpha_{\overline{r}}
	\in \left]0,+\infty\right[\smallsetminus \{1\}$,
	let $\delta_i$ be real numbers 
	in $\RR\smallsetminus\{0\}$,
	and suppose that for every $i\in \{1,\ldots,k+1\}$,
	$\tfrac{1}{\delta_i}R_i$
	is $\alpha_i$-conically nonexpansive.
	Set 
	$\overline{\beta}=\frac{\sum_{\substack{i=1 \\ i\neq \overline{r}}}^{k+1}
		\tfrac{\alpha_i}{1-\alpha_i}	
	}{1+\sum_{\substack{i=1 \\ i\neq \overline{r}}}^{k+1}\tfrac{\alpha_i}{1-\alpha_i}}$,
	 and suppose that 
	 $\alpha_{\overline{r}}\overline{\beta}<1$.
	We examine two cases.
	
	\textsc{Case~I:} $\alpha_{k+1}=\alpha_{\overline{r}}$.
	In this case the conclusion follows by 
	applying \cref{thm:conic:conic}, in view of the inductive hypothesis,
	with $(R_1,R_2)$ replaced by $(R_k\ldots R_1,R_{k+1})$
	and $(\delta_1,\delta_2,\alpha_1,\alpha_2)$ replaced by
	$(\delta_1\ldots \delta_k,\delta_{k+1},
	\big({\sum_{i=1}^k\tfrac{\alpha_i}{1-\alpha_i}}\big)/
	\big({1+\sum_{i=1}^k\tfrac{\alpha_i}{1-\alpha_i}}\big),
	\alpha_{k+1})$. 
	
	\textsc{Case~II:} $\alpha_{k+1}\neq\alpha_{\overline{r}}$.
	We claim that 
	\begin{equation}
	\label{e:191126a}
	\alpha_{k+1}\frac{\sum_{i=1}^k\tfrac{\alpha_i}{1-\alpha_i}}{
		1+\sum_{i=1}^k\tfrac{\alpha_i}{1-\alpha_i}}<1.
	\end{equation}	
	To this end, set $\hat{\alpha}
	=\tfrac{\sum_{\substack{i=1 \\ i\neq \overline{r}}}^{k}\tfrac{\alpha_i}{1-\alpha_i}}{
		1+\sum_{\substack{i=1 \\ i\neq \overline{r}}}^{k}\tfrac{\alpha_i}{1-\alpha_i}}$,
	and observe that 
	$\hat{\alpha}<\overline{\beta}$.
	 By assumption we have $\alpha_{\overline{r}}\overline{\beta}<1$.
	Altogether, we conclude that $\alpha_{\overline{r}}\hat{\alpha}<1$.
	It follows from the inductive hypothesis that
	\begin{equation}
	\label{e:191220:d}
	  \text{$\tfrac{1}{\delta_1\ldots \delta_k}(R_k\ldots R_1)$ is 
	  $\tfrac{\sum_{i=1}^{k}\tfrac{\alpha_i}{1-\alpha_i}}{
		1+\sum_{i=1}^{k}\tfrac{\alpha_i}{1-\alpha_i}}$-conically nonexpansive.} 
	\end{equation}
	Next note that 
	\begin{subequations}
	\label{eq:label:today}
		\begin{align}
	       \frac{\sum_{i=1}^k\tfrac{\alpha_i}{1-\alpha_i}}{
	       1+\sum_{i=1}^k\tfrac{\alpha_i}{1-\alpha_i}}
            &
            =\frac{\frac{\sum_{\substack{i=1 \\ i\neq \overline{r}}}^{k}\tfrac{\alpha_i}{1-\alpha_i}+\tfrac{\alpha_{\overline{r}}}{1-\alpha_{\overline{r}}}}{1+\sum_{\substack{i=1 \\ i\neq \overline{r}}}^{k}\tfrac{\alpha_i}{1-\alpha_i}}}
            {\frac{1+\sum_{\substack{i=1 \\ i\neq \overline{r}}}^{k}\tfrac{\alpha_i}{1-\alpha_i}+\tfrac{\alpha_{\overline{r}}}{1-\alpha_{\overline{r}}}}{1+\sum_{\substack{i=1 \\ i\neq \overline{r}}}^{k}\tfrac{\alpha_i}{1-\alpha_i}}}
            =\frac{\hat{\alpha}+\frac{\alpha_{\overline{r}}}{
            (1-\alpha_{\overline{r}})\Big(1+\sum_{\substack{i=1 \\ i\neq \overline{r}}}^{k}\tfrac{\alpha_i}{1-\alpha_i}\Big)}}{1+\frac{\alpha_{\overline{r}}}{(1-\alpha_{\overline{r}})\Big(1+\sum_{\substack{i=1 \\ i\neq \overline{r}}}^{k}\tfrac{\alpha_i}{1-\alpha_i}\Big)}}
            \\
            &=\frac{\hat{\alpha}(1-\alpha_{\overline{r}})\bigg(1+\sum_{\substack{i=1 \\ i\neq \overline{r}}}^{k}\tfrac{\alpha_i}{1-\alpha_i}\bigg)+\alpha_{\overline{r}}}{(1-\alpha_{\overline{r}})
            \bigg(1+\sum_{\substack{i=1 \\ i\neq \overline{r}}}^{k}\tfrac{\alpha_i}{1-\alpha_i}\bigg)+\alpha_{\overline{r}}}
            \\
            &=\frac{\alpha_{\overline{r}}\bigg(1-\hat{\alpha}
            \bigg(1+\sum_{\substack{i=1 \\ i\neq \overline{r}}}^{k}\tfrac{\alpha_i}{1-\alpha_i}\bigg)\bigg)+\hat{\alpha}
            \bigg(1+\sum_{\substack{i=1 \\ i\neq \overline{r}}}^{k}\tfrac{\alpha_i}{1-\alpha_i}\bigg)}{1+(1-\alpha_{\overline{r}})\sum_{\substack{i=1 \\ i\neq \overline{r}}}^{k}\tfrac{\alpha_i}{1-\alpha_i}}.
		\end{align}
		\end{subequations}
        Because $\alpha_{\overline{r}}\overline{\beta}<1$,
        we learn that $1+(1-\alpha_{\overline{r}})\sum_{\substack{i=1 \\ i\neq \overline{r}}}^{k}\tfrac{\alpha_i}{1-\alpha_i}>0$.
        Moreover, because $\hat{\alpha}<1$, we have $\alpha_{k+1}\hat{\alpha}<1$.
        Therefore, \cref{eq:label:today} implies
    	\begin{subequations}
    	\label{e:191220:a}
		\begin{align}
            &\quad\alpha_{k+1}\frac{\sum_{i=1}^k\tfrac{\alpha_i}{1-\alpha_i}}{1
            +\sum_{i=1}^k\tfrac{\alpha_i}{1-\alpha_i}}<1
            \\
            &\siff
            \alpha_{k+1}\Bigg(\alpha_{\overline{r}}\bigg(1-\hat{\alpha}\Bigg(1
            +\sum_{\substack{i=1 \\ i\neq \overline{r}}}^{k}\tfrac{\alpha_i}{1-\alpha_i}\Bigg)\Bigg)
            +\hat{\alpha}\Bigg(1
            +\sum_{\substack{i=1 \\ i\neq \overline{r}}}^{k}\tfrac{\alpha_i}{1-\alpha_i}\Bigg)\Bigg)
            <1+(1-\alpha_{\overline{r}})
            \sum_{\substack{i=1 \\ i\neq \overline{r}}}^{k}\tfrac{\alpha_i}{1-\alpha_i}
                \\
            &\siff
            \alpha_{\overline{r}}\Bigg(\alpha_{k+1}\Bigg(1-\hat{\alpha}
            \Bigg(1+\sum_{\substack{i=1 \\ i\neq \overline{r}}}^{k}\tfrac{\alpha_i}{1-\alpha_i}\Bigg)\Bigg)
            +\sum_{\substack{i=1 \\ i\neq \overline{r}}}^{k}\tfrac{\alpha_i}{1-\alpha_i}\Bigg)
            <\Bigg(1+\sum_{\substack{i=1 \\ i\neq \overline{r}}}^{k}\tfrac{\alpha_i}{1-\alpha_i}\Bigg)(1-\alpha_{k+1}\hat{\alpha})
            \\
            &\siff
            \alpha_{\overline{r}}\Bigg(\alpha_{k+1}\Bigg(1-\sum_{\substack{i=1 \\ i\neq \overline{r}}}^{k}\tfrac{\alpha_i}{1-\alpha_i}\Bigg)
            +\sum_{\substack{i=1 \\ i\neq \overline{r}}}^{k}\tfrac{\alpha_i}{1-\alpha_i}\Bigg)
            <\Bigg(1+\sum_{\substack{i=1 \\ i\neq \overline{r}}}^{k}\tfrac{\alpha_i}{1-\alpha_i}\Bigg)(1-\alpha_{k+1}\hat{\alpha})\\
            &\siff
            \alpha_{\overline{r}}\frac{\alpha_{k+1}\bigg(1-\sum_{\substack{i=1 \\ i\neq \overline{r}}}^{k}
            \tfrac{\alpha_i}{1-\alpha_i}\bigg)+\sum_{\substack{i=1 \\ i\neq \overline{r}}}^{k}\tfrac{\alpha_i}{1-\alpha_i}}{
            \bigg(1+\sum_{\substack{i=1 \\ i\neq \overline{r}}}^{k}\tfrac{\alpha_i}{1-\alpha_i}\bigg)(1-\alpha_{k+1}\hat{\alpha})}<1.
       	\end{align}
		\end{subequations}
Now, observe that 
        \begin{equation}
        \label{e:191220:b}
            \alpha_{k+1}\Bigg(1-\sum_{\substack{i=1 \\ i\neq \overline{r}}}^{k}\tfrac{\alpha_i}{1-\alpha_i}\Bigg)
            +\sum_{\substack{i=1 \\ i\neq \overline{r}}}^{k}\tfrac{\alpha_i}{1-\alpha_i}
            =\Bigg(\sum_{\substack{i=1 \\ i\neq \overline{r}}}^{k}\tfrac{\alpha_i}{1-\alpha_i}+\tfrac{\alpha_{k+1}}{1-\alpha_{k+1}}\Bigg)
            (1-\alpha_{k+1})
            =\sum_{\substack{i=1 \\ i\neq \overline{r}}}^{k+1}\tfrac{\alpha_i}{1-\alpha_i}(1-\alpha_{k+1}),
        \end{equation}
        and
        \begin{subequations}
               \label{e:191220:c}
        \begin{align}
            \Bigg(1+\sum_{\substack{i=1 \\ i\neq \overline{r}}}^{k}\tfrac{\alpha_i}{1-\alpha_i}\Bigg)\Bigg(1-\alpha_{k+1}\hat{\alpha}\Bigg)
            &=1+\sum_{\substack{i=1 \\ i\neq \overline{r}}}^{k}\tfrac{\alpha_i}{1-\alpha_i}
            -\alpha_{k+1}\sum_{\substack{i=1 \\ i\neq \overline{r}}}^{k}\tfrac{\alpha_i}{1-\alpha_i}
            \\
            &=\Bigg(1+\sum_{\substack{i=1 \\ i\neq \overline{r}}}^{k}\tfrac{\alpha_i}{1-\alpha_i}+\tfrac{\alpha_{k+1}}{1-\alpha_{k+1}}\Bigg)
            (1-\alpha_{k+1})
            \\
            &=\Bigg(1+\sum_{\substack{i=1 \\ i\neq\overline{r} }}^{k+1}\tfrac{\alpha_i}{1-\alpha_i}\Bigg)
            (1-\alpha_{k+1}).
            \end{align}
        \end{subequations}
        In view of \cref{e:191220:b} and \cref{e:191220:c},
        \cref{e:191220:a} becomes
        \begin{equation}
            \alpha_{k+1}\frac{\sum_{i=1}^k\tfrac{\alpha_i}{1-\alpha_i}}{1+\sum_{i=1}^k\tfrac{\alpha_i}{1-\alpha_i}}<1
            \siff
            \alpha_{\overline{r}}\frac{\sum_{i=1}^{k+1}\tfrac{\alpha_1}{1-\alpha_i}}{1+\sum_{i=1}^{k+1}\tfrac{\alpha_1}{1-\alpha_i}}=\alpha_{\overline{r}}\overline{\beta}<1.
        \end{equation}
        This proves \cref{e:191126a}.
    	Now proceed similar to \textsc{Case~I}, in view of 
		\cref{e:191126a} and \cref{e:191220:d}.
\end{proof}
The assumption $\alpha_r\overline{\alpha}<1$ is critical in the 
conclusion of the 
above theorem as we illustrate in the following example.

\begin{example}
	\label{example:190502}
	Let $\epsilon>0$, let $\delta>1$, 
	let $\alpha_1\in\ ]0,\tfrac{1}{2}(\sqrt{(\epsilon+\delta)^2+4}-(\epsilon+\delta))[$,
	let $\alpha_2=\alpha_1+\delta+\epsilon$ and let 
	\begin{equation}
	S=\begin{bmatrix}
	0&-1
	\\
	1&0
	\end{bmatrix}.
	\end{equation}
	Set $R_1=(1-\alpha_1)\Id-\alpha_1 S$,
	$R_2=(1-\alpha_2)\Id+\alpha_2 S$,
	$R_3=-\tfrac{1}{\delta}S$, and 
	\begin{equation}
	R=R_3R_2R_1.
	\end{equation}
	Then $R=R_3R_1R_2=R_1R_2R_3=R_1R_3R_2=R_2R_3R_1=R_2R_1R_3$.
	Moreover, the following hold:
	\begin{enumerate}
		\item
		\label{example:190502:i} 
		$\alpha_1\in\ ]0,1[$, $\alpha_2>1$,
		and $\alpha_1\alpha_2<1$.
		\item 
		\label{example:190502:ii}
		$R_3$ is $\alpha_3$-conically nonexpansive where 
		$\alpha_3=\tfrac{1+\delta}{2\delta}\in\ ]1/2,1]$.  
		\item 
		\label{example:190502:iii}
		$\tfrac{\alpha_1+\alpha_2-2\alpha_1\alpha_2}{1-\alpha_1\alpha_2}\alpha_3>1$.
		\item 
		\label{example:190502:iv}
		$R=\big(\tfrac{\epsilon+\delta}{\delta}\big)\Id
		+\big(\tfrac{\alpha_1+\alpha_2-2\alpha_1\alpha_2-1}{\delta}\big)S$. 
		\item 
		\label{example:190502:v}
		$\Id-R=-\tfrac{\epsilon}{\delta}\Id
		-\big(\tfrac{\alpha_1+\alpha_2-2\alpha\alpha_2-1}{\delta}\big)S$. 
		Hence, $\Id-R$ is not monotone.
		\item 
		\label{example:190502:vi}
		$R$ is \emph{not} conically nonexpansive.
	\end{enumerate}		
\end{example}	
\begin{proof}
	It is straightforward to verify that
	$R=R_3R_1R_2=R_1R_2R_3=R_1R_3R_2=R_2R_3R_1=R_2R_1R_3$.
	\ref{example:190502:i}:
	It is clear that $\alpha_1\in\ ]0,1[$, and that $\alpha_2>1$. 	
	Note that 
	$\alpha_1\alpha_2<1 \siff \alpha_1^2+(\epsilon+\delta)\alpha_1-1<0$
	$\siff$ $\alpha_1$ lies between the roots of the quadratic
	$x^2+(\epsilon+\delta)x-1$, and the conclusion follows 
	from the quadratic formula.
	\ref{example:190502:ii}:
	This  follows from 
	\cite[Proposition~4.38]{BC2017}.	
	\ref{example:190502:iii}:	
	Indeed, in view of \ref{example:190502:i} we have 
	\begin{subequations}
		\begin{align}
		&\quad\tfrac{\alpha_1+\alpha_2
			-2\alpha_1\alpha_2}{1-\alpha_1\alpha_2}\alpha_3>1
		\nonumber
		\\
		&\siff (\alpha_1+\alpha_2
		-2\alpha_1\alpha_2)\alpha_3>1-\alpha_1\alpha_2
		\\
		&
		\siff
		(\alpha_1+\alpha_2-2\alpha_1\alpha_2)(1+\delta)
		>2(1-\alpha_1\alpha_2)\delta
		\\
		&
		\siff
		(\alpha_1+\alpha_2)(1+\delta)
		-2\alpha_1\alpha_2-2\alpha_1\alpha_2\delta
		>2\delta-2\alpha_1\alpha_2\delta
		\\
		&
		\siff
		(\alpha_1+\alpha_2)(1+\delta)
		-2\alpha_1\alpha_2
		>2\delta
		\\
		&
		\siff
		(2\alpha_1+\epsilon+\delta)(1+\delta)
		-2\alpha_1(\alpha_1+\epsilon+\delta)
		>2\delta
		\\
 		&
 		\siff 
 		2\alpha_1(1+\delta-\alpha_1-\epsilon-\delta)
 		+\delta^2+\delta(1+\epsilon)
 		+\epsilon>2\delta
 		\\
		&
		\siff
		2\alpha_1(\alpha_1-1+\epsilon)<\delta^2-\delta +\epsilon\delta+\epsilon
		=\delta^2-\delta +(1+\delta )\epsilon.
		\end{align}
	\end{subequations}	
	Now, because $\alpha_1<1$, $\delta\ge 1$
	we learn that 
	$2\alpha_1(\alpha_1-1+\epsilon)<2\alpha_1\epsilon
	<(1+\delta)\epsilon <(1+\delta)\epsilon+\delta^2-\delta$,
	and the conclusion follows.
	\ref{example:190502:iv}:
	It is straightforward, by noting that 
	$S^2=-\Id$, to verify that
	$R_2R_1=R_1R_2=(1-\alpha_1-\alpha_2+\alpha_1\alpha_2)\Id+
	(\alpha_2(1-\alpha_1)-\alpha_1(1-\alpha_2))S-\alpha_1\alpha_2S^2
	=(1-\alpha_1-\alpha_2+2\alpha_1\alpha_2)\Id
	+(\alpha_2-\alpha_1)S
	$.
	Consequently,
	$R_3R_2R_1=\tfrac{1}{\delta}(-(1-\alpha_1-\alpha_2+2\alpha_1\alpha_2)S
	-(\alpha_2-\alpha_1)S^2)
	=\tfrac{1}{\delta}((\alpha_2-\alpha_1)\Id
	-(1-\alpha_1-\alpha_2+2\alpha_1\alpha_2)S)
	=\tfrac{\epsilon+\delta}{\delta}\Id
	+\tfrac{\alpha_1+\alpha_2-2\alpha_1\alpha_2-1}{\delta}S$.
	\ref{example:190502:v}:
	This is a direct consequence
	of \ref{example:190502:iv}.
	\ref{example:190502:vi}:
	Combine \ref{example:190502:v} and \cref{lem:av:coco}.
\end{proof}

\begin{theorem}[composition of cocoercive operators]
	Let $m\ge 1$ be an integer,
	set $I=\{1,\ldots,m\}$,
	let $(R_i)_{i\in I}$ be a family of operators from $\hilbert$
	to $\hilbert$,
	let $\beta_i$ be real numbers 
	in $\left]0,+\infty\right[$
	and suppose that for every $i\in I$, $R_i$
	is $\tfrac{1}{\beta_i}$-cocoercive.
	Set 
	$	R=R_m\ldots R_1
	$.
	Then there exists a nonexpansive operator $N\colon X\to X$ such that 
	\begin{equation}
	R=\beta_m\ldots \beta_1\big(\tfrac{1}{1+m}\Id+\tfrac{m}{1+m} N\big).
	\end{equation} 
\end{theorem}	
\begin{proof}
	Apply
	\cref{thm:m:conic} with $(\alpha_i,\delta_i)$
	replaced by $(1/2,\beta_i)$,
	$i\in \{1,\ldots,m\}$.
\end{proof}	

\section{Application to the Douglas--Rachford algorithm}
\label{sec:main:1}

\begin{theorem}[averagedness of the
	Douglas--Rachford operator]
	\label{thm:DR:relax}
	Let $\mu>\omega\ge 0$,
	%	let $\lambda\in \left]0,1\right[$
	and let $\gamma\in \left]0, 
	{(\mu-\omega)}/{(2\mu\omega})\right[$.
	Suppose that one of the following holds:
	\begin{enumerate}
		\item 
		\label{asm:order:sq:1}
		$A$ is maximally $(-\omega) $-monotone and 
		$B$ is maximally $\mu $-monotone. 
		\item 
		\label{asm:order:sq:2}
		$A$ is maximally $\mu $-monotone and 
		$B$ is maximally $(-\omega) $-monotone. 
	\end{enumerate}
	Set 
	\begin{equation}
	T=\tfrac{1}{2}(\Id+ R_{\gamma B}R_{\gamma A}),
	\quad
	\text{and}
	\quad 
	\alpha=
	\tfrac{\mu-\omega}{2(\mu-\omega-\gamma\mu\omega)}.
	\end{equation}
	Then $\alpha\in \left]0,1\right[$ 
	and $T$ 
	is $\alpha$-averaged.
\end{theorem}
\begin{proof}
	Suppose that \cref{asm:order:sq:1} holds.
	Note that $\gamma A$ is $-\gamma\omega$-monotone, and 
	\begin{equation}
	-\gamma\omega >-\tfrac{\mu-\omega}{2\mu}
	\ge- \tfrac{\mu}{2\mu}>- 1.
	\end{equation} 
	Using \cref{e:sv:fdom}
	and \cref{prop:static:zeros}
	we learn that  $J_{\gamma A}$ and, in turn, 
	$T$ are single valued and 
	$\dom J_{\gamma A}=\dom T=X$.
	It follows from \cite[Proposition~4.3~and~Table~1]{BMW19}
	that $-R_{\gamma A}$ is $\tfrac{1}{1+\gamma \mu}$-conically nonexpansive
	and
	$-R_{\gamma B}$ is $\tfrac{1}{1-\gamma \omega}$-conically nonexpansive.
	It follows from \cref{thm:conic:conic}, applied with 
	$(\alpha_1,\beta_1,\delta_1,\alpha_2,\beta_2,\delta_2)$ 
	replaced by $(1-\tfrac{1}{1+\gamma\mu},\tfrac{1}{1+\gamma\mu},-1,
	1-\tfrac{1}{1-\gamma\omega},\tfrac{1}{1-\gamma\omega},-1)$,
	that
	$R_{\gamma B}
	R_{\gamma A}$ is $\tfrac{\mu-\omega}{\mu-\omega-\gamma\mu\omega}$-conically nonexpansive.
	Therefore, there exists a 
	nonexpansive mapping $N\colon X\to X$ such that
	\begin{equation}
	R_{\gamma B}R_{\gamma A}=(1-\delta)\Id +\delta N, 
	\quad \delta=\tfrac{\mu-\omega}{
		\mu-\omega-\gamma\mu\omega}.
	\end{equation}
	The conclusion now follows 
	by applying \cref{prop:why:conic} with
	$(\beta,N) $
	replaced by $(\tfrac{\alpha}{\delta}, R_{\gamma B}R_{\gamma A})$.
	Finally notice that 
	$\gamma <\tfrac{\mu-\omega}{2\mu\omega}$,
	which implies  that $0<\mu-\omega
	< 2(\mu-\omega-\gamma\mu\omega)$.
	Therefore, 
	\begin{equation}
	\alpha
	% =2\lambda \delta 
	=\tfrac{\mu-\omega}{2(\mu
		-\omega-\gamma \mu\omega) }
	\in \left]0,1\right[.
	\end{equation} 
	The proof of \cref{asm:order:sq:2}
	follows similarly. 
\end{proof}	
\begin{corollary}\cite[Theorem~4.5(ii)]{DP18}
	\label{cor:overn:con}
	Let $\mu>\omega\ge 0$,
	%	let $\lambda\in \left]0,1\right [$
	and let $\gamma\in \left]0, 
	{(\mu-\omega)}/{(2\mu\omega)}\right [$.
	Suppose that one of the following holds:
	\begin{enumerate}
		\item 
		\label{asm:order:sq:1:p}
		$A$ is maximally $(-\omega) $-monotone and 
		$B$ is maximally $\mu $-monotone.
		\item 
		\label{asm:order:sq:2:p}
		$A$ is maximally $\mu $-monotone and 
		$B$ is maximally $(-\omega) $-monotone. 
	\end{enumerate}	
	Set 
	$
	T=\tfrac{1}{2}(\Id+ R_{\gamma B}R_{\gamma A})
	$ and let $x_0\in X$.
	Then $(\exists\ \overline{x}\in\fix T=\fix R_{\gamma B}R_{\gamma A})$ such that
	$T^n x_0\weakly \overline{x}$.
\end{corollary}
\begin{proof}
	Combine \cref{thm:DR:relax}
	and \cite[Theorem~5.15]{BC2017}.
\end{proof}

\begin{remark}
	In view of \cref{e:sv:fdom}, one might 
	think that the scaling factor $\gamma$ 
	is required \emph{only} to guarantee the single-valuedness
	and full domain of $T$. However, it is actually critical 
	to guarantee convergence as well, as we illustrate in 
	\cref{ex:scaling:factor} below.
\end{remark}
\begin{example}
	\label{ex:scaling:factor}
	Let $\mu>\omega\ge 0$,
	let $U$ be a closed linear subspace of $X$,
	suppose that\footnote{Let $C$ be a nonempty, 
		closed convex subset of $X$. Here and elsewhere, we shall
		use $N_C$ to denote the \emph{normal cone operator} associated
		with $C$, defined by 
		$N_C(x)=\menge{u\in X}{\sup\scal{C-x}{u}\le 0}$, if $x\in C$;
		and $N_C(x)=\fady$, otherwise.} 
	\begin{equation}
	A=N_U+\mu\Id,\quad
	B=-\omega \Id.
	\end{equation}
	Then $A$ is $\mu $-monotone, 
	$B$ is $-\omega $-monotone and 
	$(\forall \gamma\in [{1}/{(2\omega)},{1}/{\omega}[)$
	$J_{\gamma B}$ is single-valued.
	Furthermore, we have 
	\begin{equation}
	\label{eq:ex:scaling}
	T=\tfrac{1}{2}
	(\Id+R_{\gamma B}R_{\gamma A})
	=\tfrac{1+\gamma \omega}{(1-\gamma \omega)(1+\gamma\mu)}P_U
	-\tfrac{\gamma \omega}{1-\gamma \omega}\Id,
	\end{equation}
	and $(\forall x_0\in U^\perp)$ $(T^n x_0)_\nnn$ 
	does not converge.
\end{example}
\begin{proof}
	Indeed, one can verify that 
	\begin{equation}
	J_{\gamma A}=\tfrac{1}{1-\gamma\omega}\Id,
	\quad
	J_{\gamma B}=\tfrac{1}{1+\gamma\mu}P_U.
	\end{equation}
	Consequently,
	\begin{equation}
	R_{\gamma A}=\tfrac{1+\gamma\omega}{1-\gamma\omega}\Id,
	\quad
	R_{\gamma B}=\tfrac{2}{1+\gamma\mu}P_U-\Id,
	\end{equation}
	and \cref{eq:ex:scaling} follows.
	Therefore,
	\begin{equation}
	T_{ |U^\perp}=
	-\tfrac{\gamma \omega}{1-\gamma \omega}\Id,\quad\text{and}\quad 
	-\tfrac{\gamma \omega}{1-\gamma \omega}\in \left]-\infty,-1\right].
	\end{equation}
	Hence, $(\forall x_0\in U^\perp)$ $(T^n x_0)_\nnn$ 
	does not converge.
\end{proof}

Before we proceed to the convergence analysis
we recall that if $T$ is averaged and $\Fix T\neq \fady$
then $(\forall x\in X)$ we have 
(see, e.g., \cite[Theorem~3.7]{Reich81})
\begin{equation}
\label{e:T:asym:reg}
T^n x-T^{n+1}x\to 0.
\end{equation}
We conclude  this section by 
proving the strong convergence of the shadow
sequence of the Douglas--Rachford algorithm.

\begin{theorem}[convergence analysis 
	of the Douglas--Rachford algorithm].
	\label{thm:convergence:DR}
	Let $\mu>\omega\ge 0$,
	%let $\lambda\in \left]0,1\right [$
	and let $\gamma\in \left]0, 
	{(\mu-\omega)}/{(2\mu\omega)}\right [$.
	Suppose that one of the following holds:
	\begin{enumerate}
		\item 
		\label{thm:convergence:DR:i}
		$A$ is maximally $\mu $-monotone and 
		$B$ is maximally $(-\omega) $-monotone. 
		\item 
		\label{thm:convergence:DR:ii}
		$A$ is maximally $(-\omega )$-monotone and 
		$B$ is maximally $\mu $-monotone. 
	\end{enumerate}	
	Set 
	\begin{equation}
	T=\tfrac{1}{2}(\Id+ R_{\gamma B}R_{\gamma A}),
	\end{equation} 
	and let $x_0\in X$.
	Then  $\zer(A+B)\neq \fady$. 
	Moreover, 
	there exists $\overline{x}\in \fix T
	=\Fix R_{\gamma B}R_{\gamma A}$,
	$\zer(A+B)=\{J_{\gamma A}\overline{x}\}
	=\{J_{\gamma B} R_{\gamma A} \overline{x}\}$,
	$T^nx_0\weakly \overline{x}$,
	$J_{\gamma A}T^nx_0\to J_{\gamma A}\overline{x}$,
	and 
	$J_{\gamma B} R_{\gamma A}T^nx_0\to 
	J_{\gamma B} R_{\gamma A}\overline{x}$.
\end{theorem}
\begin{proof}
	Suppose that \cref{asm:order:sq:1} holds.
	Since $A+B$ is $(\mu-\omega)$-monotone,
	and $\mu-\omega>0$, we conclude from
	\cite[Proposition~23.35]{BC2017}
	that $\zer(A+B)$ is a singleton.
	Combining with \cref{prop:static:zeros}
	with $(A,B)$ replaced by $(\gamma A,\gamma B)$
	yields $\zer(A+B)=\zer(\gamma A+\gamma B)
	=\{J_{\gamma A}\overline{x}\}
	=\{J_{\gamma B} R_{\gamma A} \overline{x}\}$.
	The claim that $T^nx_0\weakly \overline{x}$
	follows from \cref{cor:overn:con}.
	It remains to show that 
	$J_{\gamma A}T^nx_0\to J_{\gamma A}\overline{x}$
	and 
	$J_{\gamma B} R_{\gamma A}T^nx_0\to 
	J_{\gamma B} R_{\gamma A}\overline{x}$.
	To this end, note that $(T^nx_0)_\nnn$ is bounded,
	consequently, since $J_{\gamma A}$ 
	and $J_{\gamma B}R_{\gamma A}$ are Lipschitz 
	continuous (see 
	\cref{prop:J:R:.5}\ref{prop:J:R:.5:i}\&\ref{prop:J:R:.5:ii}), 
	we learn that 
	\begin{equation}
	\label{eq:bdd:sq}
	\text{$(J_{\gamma A}T^nx_0)_\nnn$
		and $(J_{\gamma B}R_{\gamma A}T^nx_0)_\nnn$ are bounded.}
	\end{equation}
	On the one hand, in  view of \cref{e:T:asym:reg} we have
	\begin{equation}
	\label{eq:aasym:reg}
	(\Id-T)T^n x_0=T^n x_0-T^{n+1}x_0
	=J_{\gamma A}T^nx_0-J_{\gamma B}R_{\gamma A}T^nx_0\to 0.
	\end{equation}
	Combining \cref{eq:bdd:sq} and \cref{eq:aasym:reg} yields
	\begin{subequations}
		\label{e:s:con:i}
		\begin{align}
		&\qquad\norm{J_{\gamma A}T^nx_0-J_{\gamma A}\overline{x}}^2
		-\norm{J_{\gamma B}R_{\gamma A}T^nx_0
			-J_{\gamma B}R_{\gamma A}\overline{x}}^2
		\\
		&=\scal{J_{\gamma A}T^nx_0
			-J_{\gamma B}R_{\gamma A}T^nx_0}{J_{\gamma A}T^nx_0
			+J_{\gamma B}R_{\gamma A}T^nx_0-J_{\gamma A}\overline{x}
			-J_{\gamma B}R_{\gamma A}\overline{x}}
		\\
		&=\scal{T^nx_0-T^{n+1}x_0}{J_{\gamma A}T^nx_0
			+J_{\gamma B}R_{\gamma A}T^nx_0-J_{\gamma A}\overline{x}
			-J_{\gamma B}R_{\gamma A}\overline{x}}\to 0.
		\end{align}
	\end{subequations}
	On the other hand, combining 
	\cref{lem:T:gra:AB},
	applied  with $(R_1,R_2,R_\lambda,\lambda)$
	replaced by $(R_{\gamma A},R_{\gamma B},T,1/2)$
	and $(x,y) $ replaced by $(T^nx_0, \overline{x})$,
	in view of \cref{e:T:asym:reg}
	yields
	\begin{subequations}
		\begin{align}
		0\leftarrow &\scal{T^{n+1}x_0-\overline x}{T^n x_0-T^{n+1}x_0}
		% -(1-2\lambda)\scal{T^{n}x_0-\overline x}{T^n x_0-T^{n+1}x_0} 
		\\
		&\ge
		\gamma \mu
		\big(\norm{J_{\gamma A}T^n x_0-J_{\gamma A}\overline{x}}^2
		-\tfrac{\omega}{\mu}
		\norm{J_{\gamma B}R_{\gamma A}T^n x_0-
			J_{\gamma B}R_{\gamma A}\overline{x}}^2\big)
		\\
		&\ge
		%(1-2\lambda)\scal{T^{n}x_0-\overline x}{T^n x_0-T^{n+1}x_0} 
		- \tfrac{\gamma\mu\omega}{\mu-\omega}
		\norm{T^n x_0-T^{n+1}x_0}^2\to 0.
		\end{align}
	\end{subequations}
	Therefore,
	\begin{equation}
	\label{e:s:con:ii}
	\norm{J_{\gamma A}T^n x_0-J_{\gamma A}\overline{x}}^2
	-\tfrac{\omega}{\mu}\norm{J_{\gamma B}
		R_{\gamma A}T^n x_0-J_{\gamma B}R_{\gamma A}
		\overline{x}}^2\to 0.
	\end{equation}
	Combining \cref{e:s:con:i} and \cref{e:s:con:ii} 
	and noting that 
	$\tfrac{\omega}{\mu}<1$ yields
	$ \norm{J_{\gamma A}T^n x_0-J_{\gamma A}\overline{x}}^2\to 0$
	and 
	$\norm{J_{\gamma B}R_{\gamma A}T^n x_0
		-J_{\gamma B}R_{\gamma A}\overline{x}}^2\to 0$,
	which proves \cref{thm:convergence:DR:i}.
	The proof of \cref{thm:convergence:DR:ii} proceeds similarly.
\end{proof}

\begin{remark}(relaxed Douglas--Rachford algorithm)
	A careful look at the proofs of \cref{thm:DR:relax},
	and \cref{thm:convergence:DR} reveals that 
	analogous conclusions can be drawn for the relaxed 
	Douglas--Rachford operator defined by 
	$T_\lambda=(1-\lambda)\Id+\lambda R_{\gamma B}R_{\gamma A} $,
	$\lambda\in \left]0,1\right[$.
	In this case, we choose  $\gamma\in \left]0, 
	{((1-\lambda)(\mu-\omega))}/{(\mu\omega)}\right [$.
	One can verify that the corresponding averagedness constant 
	is   
	$\alpha
	% =2\lambda \delta 
	=\tfrac{\lambda(\mu-\omega)}{\mu
		-\omega-\gamma \mu\omega }
	\in \left]0,1\right[$. 
\end{remark}	

\section{Application to the forward-backward algorithm}
\label{sec:main:2}
Throughout this section we assume that 
\begin{empheq}[box=\mybluebox]{equation*}
\text{$A\colon X\to X $,\ 
	$B\colon X\rras X$,\
	$\mu\ge 0$, $\omega\ge 0$,
	and $\beta>0$}.
\end{empheq} 
In the rest of this section, we prove that 
the forward-backward operator 
is averaged, hence its iterates 
form a weakly convergent sequence,
in each of the following situations:
\begin{itemize}
	\item
	$A$ is maximally $\mu$-monotone,
	$A-\mu \Id $ is $\tfrac{1}{\beta}$-cocoercive,
	$B$ is maximally $(-\omega)$-monotone,
	and $\mu\ge \omega$. 
	\item
	$A$ is maximally $(-\omega)$-monotone,
	$A+\omega\Id$ is $\tfrac{1}{\beta}$-cocoercive, 
	$B$ is maximally $\mu$-monotone,
	and $\mu\ge \omega$. 
	\item
	$A$ is $\beta$-Lipschitz continuous,
	$B$ is maximally $\mu$-monotone,
	and $\mu \ge \beta$.
\end{itemize}
That is, we do not require $A$ and $B$ to be monotone. Instead, it is enough that the sum $A+B$ is monotone, to have an averaged forward-backward map. In addition, we show that the forward-backward map is contractive if the sum $A+B$ is strongly monotone and we prove tightness of our contraction factor.

\begin{theorem}[case I: $A$ is $\mu$-monotone]
	\label{thm:FB:av:1}
	Let $\mu\ge \omega\ge 0$, and
	let $\beta>0$.
	Suppose that 
	$A$ is maximally $\mu$-monotone,
	$A-\mu \Id$ is $\tfrac{1}{\beta}$-cocoercive,
	and 
	$B$ is maximally $(-\omega)$-monotone.
	Let
	$\gamma\in \left ]0, 2/(\beta+2\mu)\right[$.
	Set $T=J_{\gamma B}(\Id-\gamma A)$, 
	set $\nu={\gamma \beta }/{(2(1-\gamma \mu))}$,
	set $\delta=(1-\gamma\mu)/(1-\gamma \omega)$,
	and let $x_0\in X$.
	Then $ \delta \in \left ]0,1\right]$ and $\nu \in \left ]0,1\right[ $. 
	Moreover, the following hold:
	\begin{enumerate}
		\item 
		\label{thm:FB:av:1:0}
		$T=\delta ((1-\nu)\Id+\nu N)$, $N$ is nonexpansive.
		\item 
		\label{thm:FB:av:1:i}
		$T$ is $(1-(\delta(1-\nu))/(2-\nu))$-averaged.
		\item 
		\label{thm:FB:av:1:ii}
		$T$ is $\delta$-Lipschitz continuous. 
		\item
		\label{thm:convergence:FB:A:i}
		There exists $\overline{x}\in \fix T
		=\zer(A+B) $ such that 
		$T^nx_0\weakly \overline{x}$.
	\end{enumerate}
	Suppose that  $\mu>\omega$. Then we additionally have:
	\begin{enumerate} 
		\setcounter{enumi}{4}
		\item
		\label{thm:convergence:FB:A:v}
		$T $ is a Banach contraction with a constant $\delta<1$.
		\item
		\label{thm:convergence:FB:A:ii}
		$\zer(A+B)=\{\overline{x}\}$
		and $T^nx_0\to\overline{x}$ with a linear 
		rate $\delta<1$. 
	\end{enumerate}
	
\end{theorem}
\begin{proof}
	Clearly, $\delta\in \left ]0,1\right] $ and 
	$\nu>0$. Moreover,
	we have $\nu<1$ $\siff $
	$\gamma\beta<2(1-\gamma \mu)$
	$\siff$
	$\gamma<2/(\beta+2\mu\beta)$.
	Hence,  $\nu \in \left ]0,1\right[ $
	as claimed. 
	Next note that $\mu<(\beta+2\mu)/2$, hence 
	$\gamma\omega<\gamma \mu <(2\gamma)/(\beta+2\mu)<1$.
	It follows from \cref{prop:static:zeros:FB}
	that $J_{\gamma B}$ and, in turn, $T$
	are single-valued and $\dom J_{\gamma B}=\dom T=X$.
	The assumption on $A$ 
	implies that there exists 
	$\overline{N}\colon X\to X$, $\overline{N}$ is nonexpansive,
	such that $A-\mu\Id=\tfrac{\beta}{2 }\Id+\tfrac{\beta}{2 }\overline{N}$.
	Therefore,
	\begin{subequations}
		\label{eq:191122c}
		\begin{align}
		\Id-\gamma A
		&=\Id-\gamma( A-\mu \Id)-\gamma\mu \Id
		=(1-\gamma\mu)\Id-\tfrac{\gamma\beta}{2}(\Id+\overline{N})
		\\
		&=(1-\gamma \mu) \big((1-\nu)\Id+\nu (-\overline{N})\big).
		\end{align}
	\end{subequations}
	
	Moreover,
	\cref{prop:J:R:.5}\ref{prop:J:R:.5:i}
	implies that 
	\begin{equation}
	\label{eq:191122d}
	\text{$J_{\gamma B}$ is $(1-\gamma \omega)$-cocoercive.}
	\end{equation}
	
	\ref{thm:FB:av:1:0}:
	It follows from \cref{cor:scav:coco}
	applied with $(R_1,R_2)$
	replaced by $(\Id-\gamma A,J_{\gamma B})$
	and $(\alpha,\beta,\delta )$
	replaced by 
	$(\nu,1/(1-\gamma \omega),1-\gamma \mu)$,
	in view of \cref{eq:191122c} and \cref{eq:191122d},
	that there exists a 
	nonexpansive operator $N$ such that 
	$T=J_{\gamma B}(\Id-\gamma A)=\delta ((1-\nu)\Id+\nu N)$.
	\ref{thm:FB:av:1:i}:
	Combine \ref{thm:FB:av:1:0} and  \cref{lem:T:av:factor}\ref{lem:T:av:factor:i}. 
	\ref{thm:FB:av:1:ii}:
	Combine \ref{thm:FB:av:1:0} and
	\ref{thm:FB:av:1:i}.
	\ref{thm:convergence:FB:A:i}:
	Applying
	\cref{prop:static:zeros:FB}
	with $(A,B)$ replaced by $(\gamma A,\gamma B)$
	yields $\zer(A+B)=\zer(\gamma A+\gamma B)
	=\fix T$.
	The claim that $T^nx_0\weakly \overline{x}$
	follows from combining
	\ref{thm:FB:av:1:i}
	and \cite[Theorem~5.15]{BC2017}.
	\ref{thm:convergence:FB:A:v}:
	Observe that $\delta<1\siff \mu>\omega$.
	Now, combine with  \ref{thm:FB:av:1:ii}.
	\ref{thm:convergence:FB:A:ii}:
	Note that $A+B$ is maximally $(\mu-\omega)$-monotone
	and $\mu-\omega>0$, we conclude 
	from \cite[Proposition~23.35]{BC2017}
	that $\zer(A+B)$ is a singleton.
	Alternatively,  use \ref{thm:FB:av:1:ii} to
	learn that $T$ is a Banach contraction 
	with a constant $\delta<1$,
	hence $\zer(A+B)=\fix T$ is a singleton,
	and the conclusion follows.
\end{proof}

\begin{theorem}
	\label{thm:FB:av:1b}
	Let $\mu> \omega\ge 0$, and
	let $\beta>0$.
	Suppose that 
	$A$ is maximally $\mu$-monotone,
	$A-\mu \Id$ is $\tfrac{1}{\beta}$-cocoercive,
	and 
	$B$ is maximally $(-\omega)$-monotone.
	Let
	$\gamma\in \left [2/(\beta+2\mu), 2/(\beta+\mu)\right[$.
	Set $T=J_{\gamma B}(\Id-\gamma A)$, 
	set $\nu={\gamma \beta }/{(2(\gamma (\mu+\beta)-1))}$,
	set $\delta=(1-\gamma(\mu+\beta))/(1-\gamma \omega)$,
	and let $x_0\in X$.
	Then $ \delta \in \left ]-1,0\right]$ and $\nu \in \left ]0,1\right[ $. 
	Moreover, the following hold:
	\begin{enumerate}
		\item 
		\label{thm:FB:av:1b:0}
		$T=\delta ((1-\nu)\Id+\nu N)$, $N$ is nonexpansive.
		\item
		\label{thm:convergence:FB:Ab:v}
		$T $ is a Banach contraction with a constant $|\delta|<1$.
		\item
		\label{thm:convergence:FB:Ab:ii}
		 There exists $\overline{x}\in X$ such that 
		$\fix T=\zer(A+B)=\{\overline{x}\}$
		and $T^nx_0\to\overline{x}$ with a linear 
		rate $\abs{\delta}<1$. 
	\end{enumerate}
	
\end{theorem}
\begin{proof}
We proceed similar to the proof of
\cref{thm:FB:av:1} to verify that $T$ is single-valued,
$\dom T=X$, $\nu\in \left ]0,1\right [$,
and $\delta\in \left ]-1,0\right]$.
	The assumption on $A$ 
	implies that there exists 
	$\overline{N}\colon X\to X$, $\overline{N}$ is nonexpansive,
	such that $A-\mu\Id=\tfrac{\beta}{2 }\Id+\tfrac{\beta}{2 }\overline{N}$.
	Therefore,
	\begin{subequations}
		\label{eq:191205a}
		\begin{align}
		\Id-\gamma A
		&=\Id-\gamma( A-\mu \Id)-\gamma\mu \Id
		=(1-\gamma\mu)\Id-\tfrac{\gamma\beta}{2}(\Id+\overline{N})
		\\
		&=(1-\gamma (\mu+\beta)) \big((1-\nu)\Id+\nu (\overline{N})\big).
		\end{align}
	\end{subequations}
	Now, proceed similar to the proof of
	\cref{thm:FB:av:1}\ref{thm:FB:av:1:0}, \ref{thm:convergence:FB:A:v}
	and \ref{thm:convergence:FB:A:ii},
	in view of \cref{eq:191122d}.
\end{proof}
\begin{corollary}
    Let $\mu> \omega\ge 0$, and
	let $\beta>0$.
	Suppose that 
	$A$ is maximally $\mu$-monotone,
	$A-\mu \Id$ is $\tfrac{1}{\beta}$-cocoercive,
	and 
	$B$ is maximally $(-\omega)$-monotone.
	Let
	$\gamma\in \left ]0, 2/(\beta+\mu)\right[$.
	Set $T=J_{\gamma B}(\Id-\gamma A)$, 
	set $\delta=\max(1-\gamma\mu,\gamma(\mu+\beta)-1)/(1-\gamma \omega)$,
	and let $x_0\in X$.
	Then $ \delta \in \left [0,1\right[$, $T$ is 
	a Banach contraction with a constant $\delta$,
	 and there exists $\overline{x}\in X$ such that 
	 $\fix T=\zer(A+B)=\{\overline{x}\}$ and $T^n x_0\to \overline{x}$.
\end{corollary}

\begin{proof}
    Combine \cref{thm:FB:av:1} and \cref{thm:FB:av:1b}.
\end{proof}
\begin{remark}[tightness of the Lipschitz constant]\
\begin{enumerate}
    \item 
    Suppose that the setting of \cref{thm:FB:av:1}
    holds. Set $(A,B)=(\mu\Id,-\omega \Id)$.
    Then $T=\tfrac{1-\gamma \mu}{1-\gamma \omega}\Id$.
    Hence, the claimed Lipschitz constant is tight.
    \item     
    Suppose that the setting of \cref{thm:FB:av:1b}
    holds. Set $(A,B)=((\mu+\beta)\Id,-\omega \Id)$.
    Then $T=\tfrac{\gamma (\mu+\beta)-1}{1-\gamma \omega}\Id$.
    Hence, the claimed contraction factor is tight.
\end{enumerate}
Note in particular that the worst cases are subgradients of convex functions. Hence, the worst cases are attained by the proximal gradient method.
\end{remark}

\begin{theorem}[case II: $A+\omega\Id $ is cocoercive]
	\label{thm:FB:av:2}
	Let $\mu\ge \omega\ge 0$, 
	let $\beta>0$ and let 
	$\overline{\beta}\in \left]\max\{\beta, \mu+\omega\},+\infty\right[$.
	Suppose that 
	$A$ is maximally $(-\omega)$-monotone,
	$A+\omega\Id$ is $\beta$-cocoercive,
	and 
	$B$ is maximally $\mu$-monotone.
	Let
	$\gamma\in \left ]0,
	2/(\overline{\beta}-2\omega)\right[$.
	Set $T=J_{\gamma B}(\Id-\gamma A)$, 
	set $\nu={\gamma\overline{\beta}  }/{(2(1+\gamma \omega))}$,
	set $\delta=(1+\gamma\omega)/(1+\gamma \mu)$,
	and let $x_0\in X$.
	Then $ \delta \in \left ]0,1\right]$ and $\nu \in \left ]0,1\right[ $.  
	Moreover, the following hold:
	\begin{enumerate}
		\item 
		\label{thm:FB:av:2:0}
		$T=\delta ((1-\nu)\Id+\nu N)$, $N$ is nonexpansive.
		\item 
		\label{thm:FB:av:2:i}
		$T$ is $(1-(\delta(1-\nu))/(2-\nu))$-averaged.
		\item 
		\label{thm:FB:av:2:ii} 
		$T$ is  
		$\delta$-Lipschitz continuous. 
		\item
		\label{thm:FB:av:2:iii}
		There exists $\overline{x}\in \fix T
		=\zer(A+B) $, and 
		$T^nx_0\weakly \overline{x}$.
	\end{enumerate}
	Suppose that  $\mu>\omega$. Then we additionally have:
	\begin{enumerate} 
		\setcounter{enumi}{4}
		\item
		\label{thm:FB:av:2:iv}
		$T$ is  a Banach contraction with a constant 
		$\delta<1$.
		\item
		\label{thm:FB:av:2:v}
		$\zer(A+B)=\{\overline{x}\}$
		and $T^nx_0\to\overline{x}$ with a linear 
		rate $\delta<1$. 
	\end{enumerate}
\end{theorem}
\begin{proof}
	Observe that, the assumption on $A$
	and \cref{lem:beta:smaller} applied with
	$T$ replaced by $A+\omega\Id$
	imply that there exists 
	$\overline{N}\colon X\to X$, $\overline{N}$ is nonexpansive,
	such that $A+\omega\Id=\tfrac{\overline{\beta}}{2 }\Id+\tfrac{\overline{\beta}}{2}\overline{N}$. 
	\begin{subequations}
		\label{eq:191122e}
		\begin{align}
		\Id-\gamma A
		&=\Id-\gamma( A+\omega \Id)+\gamma\omega \Id
		=(1+\gamma\omega)\Id-\tfrac{\gamma\overline{\beta}}{2}(\Id+\overline{N})
		\\
		&=(1+\gamma \omega) \big((1-\nu)\Id+\nu (-\overline{N})\big).
		\end{align}
	\end{subequations}
	Moreover,
	\cref{prop:J:R:.5}\ref{prop:J:R:.5:i}
	implies that 
	\begin{equation}
	\label{eq:191122f}
	\text{$J_{\gamma B}$ is $(1+\gamma \mu)$-cocoercive.}
	\end{equation}
	Now proceed similar to the proof of \cref{thm:FB:av:1}
	but use \cref{eq:191122e} and \cref{eq:191122f}.
\end{proof}

\begin{theorem}
	\label{thm:FB:av:2:b}
	Let $\mu> \omega\ge 0$, 
	let $\beta>0$ and let 
	$\overline{\beta}\in \left]\max\{\beta, \mu+\omega\},+\infty\right[$.
	Suppose that 
	$A$ is maximally $(-\omega)$-monotone,
	$A+\omega\Id$ is $\beta$-cocoercive,
	and 
	$B$ is maximally $\mu$-monotone.
	Let
	$\gamma\in \left [
	2/(\overline{\beta}-2\omega),2/(\overline{\beta}-\mu-\omega)\right[$.
	Set $T=J_{\gamma B}(\Id-\gamma A)$, 
	set $\nu={\gamma\overline{\beta}  }/{(2(\gamma \overline{\beta}-\gamma \omega-1))}$,
	set $\delta=(1+\gamma\omega-\gamma \overline{\beta})/(1+\gamma \mu)$,
	and let $x_0\in X$.
	Then $ \delta \in \left ]-1,0\right]$ and $\nu \in \left ]0,1\right[ $.  
	Moreover, the following hold:
	\begin{enumerate}
		\item 
		\label{thm:FB:av:b:2:0}
		$T=\delta ((1-\nu)\Id+\nu N)$, $N$ is nonexpansive.
		\item
		\label{thm:FB:av:b:2:iv}
		$T$ is  a Banach contraction with a constant 
		$\abs{\delta}<1$.
		\item
		\label{thm:FB:av:b:2:v}
		There exists $\overline{x}\in X$ such that 
		$\fix T=\zer(A+B)=\{\overline{x}\}$
		and $T^nx_0\to\overline{x}$ with a linear 
		rate $\abs{\delta}<1$. 
	\end{enumerate}
\end{theorem}
\begin{proof}
	Observe that, the assumption on $A$
	and \cref{lem:beta:smaller} applied with
	$T$ replaced by $A+\omega\Id$
	imply that there exists 
	$\overline{N}\colon X\to X$, $\overline{N}$ is nonexpansive,
	such that $A+\omega\Id=\tfrac{\overline{\beta}}{2 }\Id+\tfrac{\overline{\beta}}{2}\overline{N}$. 
	\begin{subequations}
		\label{eq:191207a}
		\begin{align}
		\Id-\gamma A
		&=\Id-\gamma( A+\omega \Id)+\gamma\omega \Id
		=(1+\gamma\omega)\Id-\tfrac{\gamma\overline{\beta}}{2}(\Id+\overline{N})
		\\
		&=(1+\gamma \omega-\gamma\overline{\beta}) \big((1-\nu)\Id+\nu \overline{N}\big).
		\end{align}
	\end{subequations}
	Now proceed similar to the proof of \cref{thm:FB:av:2}, in 
	view of \cref{eq:191122f}.
\end{proof}

\begin{corollary}
	Let $\mu> \omega\ge 0$, 
	let $\beta>0$ and let 
	$\overline{\beta}\in \left]\max\{\beta, \mu+\omega\},+\infty\right[$.
	Suppose that 
	$A$ is maximally $(-\omega)$-monotone,
	$A+\omega\Id$ is $\beta$-cocoercive,
	and 
	$B$ is maximally $\mu$-monotone.
	Let
	$\gamma\in \left [
	0,2/(\overline{\beta}-\mu-\omega)\right[$.
	Set $T=J_{\gamma B}(\Id-\gamma A)$, 
	set $\delta=\max\{1+\gamma\mu,\gamma \overline{\beta}-\gamma\omega-1\}/(1+\gamma \mu)$,
	and let $x_0\in X$.
	Then $ \delta \in \left ]-1,0\right]$ and $\nu \in \left ]0,1\right[ $.  
	Then $ \delta \in \left [0,1\right[$, $T$ is 
	a Banach contraction with a constant $\delta$,
	 and there exists $\overline{x}\in X$ such that 
	 $\fix T=\zer(A+B)=\{\overline{x}\}$ and $T^n x_0\to \overline{x}$.
\end{corollary}

\begin{proof}
    Combine \cref{thm:FB:av:2} and \cref{thm:FB:av:2:b}
    .
\end{proof}

\begin{theorem}[case III: $A$ is ${\beta}$-Lipschitz continuous]
	\label{thm:FB:av:3}
	Let $ \mu\ge \beta>0$.
	Suppose that 
	$A$ is $\beta$-Lipschitz continuous
	and that 
	$B$ is maximally $\mu$-monotone.
	Let 
	$\overline{\beta}\in \left]2\beta,+\infty \right[$,
 and let 
	$\gamma\in \left ]0,2
	/( \overline{\beta}-2\beta)\}\right[$.
	Set $T=J_{\gamma B}(\Id-\gamma A)$, 
	set $\nu={\gamma\overline{\beta} }/{(2(1+\gamma\beta))}$,
	set $\delta=(1+\gamma\beta)/(1+\gamma \mu)$,
	and let $x_0\in X$.
	Then $ \delta \in \left ]0,1\right]$ and $\nu \in \left ]0,1\right[ $. 
	Moreover, the following hold:
	\begin{enumerate}
		\item 
		\label{thm:FB:av:3:0}
		$T=\delta ((1-\nu)\Id+\nu N)$, $N$ is nonexpansive.
		\item 
		\label{thm:FB:av:3:i}
		$T$ is $(1-(\delta(1-\nu))/(2-\nu))$-averaged.
		\item 
		\label{thm:FB:av:3:ii} 
		$T$ is $\delta$-Lipschitz continuous. 
		\item
		\label{thm:FB:av:3:iv}
		There exists $\overline{x}\in \fix T
		=\zer(A+B) $, and 
		$T^nx_0\weakly \overline{x}$.
	\end{enumerate}
	Suppose that  $\mu>1/\beta$. Then we additionally have:
	\begin{enumerate} 
		\setcounter{enumi}{4}
		\item
		\label{thm:FB:av:3:v}
		$T$ is  a Banach contraction with a constant $\delta<1$.
		\item
		\label{thm:FB:av:3:vi}
		$\zer(A+B)=\{\overline{x}\}$
		and $T^nx_0\to\overline{x}$ with a linear 
		rate $\delta<1$. 
	\end{enumerate}
\end{theorem}
\begin{proof}
	Combine \cref{lem:T:Lips:to:coco}
	and 
	\cref{thm:FB:av:2} applied with 
	$(\omega,\beta)$ replaced by 
	$(\beta,2\beta)$.	
\end{proof}
\begin{theorem}
	\label{thm:FB:av:3:b}
	Let $\mu> \beta> 0$. 
	Suppose that 
	$A$ is $\beta$-Lipschitz continuous 
	and that
	$B$ is maximally $\mu$-monotone.
	Let 
	$\overline{\beta}\in \left]\mu+\beta,+\infty\right[$,
	and let
	$\gamma\in \left [
	2/(\overline{\beta}-2\beta),2/(\overline{\beta}-\mu-\beta)\right[$.
	Set $T=J_{\gamma B}(\Id-\gamma A)$, 
	set $\nu={\gamma\overline{\beta}  }/{(2(\gamma \overline{\beta}-\gamma \beta-1))}$,
	set $\delta=(1+\gamma\beta-\gamma \overline{\beta})/(1+\gamma \mu)$,
	and let $x_0\in X$.
	Then $ \delta \in \left ]-1,0\right]$ and $\nu \in \left ]0,1\right[ $.  
	Moreover, the following hold:
	\begin{enumerate}
		\item 
		\label{thm:FB:av:b:2:0}
		$T=\delta ((1-\nu)\Id+\nu N)$, $N$ is nonexpansive.
		\item
		\label{thm:FB:av:b:2:iv}
		$T$ is  a Banach contraction with a constant 
		$\abs{\delta}<1$.
		\item
		\label{thm:FB:av:b:2:v}
		There exists $\overline{x}\in X$ such that 
		$\fix T=\zer(A+B)=\{\overline{x}\}$
		and $T^nx_0\to\overline{x}$ with a linear 
		rate $\abs{\delta}<1$. 
	\end{enumerate}
\end{theorem}
\begin{proof}
	Combine \cref{lem:T:Lips:to:coco}
	and 
	\cref{thm:FB:av:2:b} applied with 
	$(\omega,\beta)$ replaced by 
	$(\beta,2\beta)$.	
\end{proof}
\section{Applications to optimization problems}
\label{sec:main:3} 
Let $f\colon X\to \left]-\infty,+\infty\right]$, and let 
$g\colon X\to \left]-\infty,+\infty\right]$.
Throughout this section, we shall assume that
\begin{empheq}[box=\mybluebox]{equation*}
\text{$f$ and $g$
	are 
	lower semicontinuous proper functions.}
\end{empheq} 
We shall use $\partial f$ to denote 
the subdifferential mapping from convex analysis.
\begin{definition}
	\label{def:abst}(see \cite[Definition~6.1]{BMW19})
	An abstract subdifferential $\pars$ associates a subset $\pars
	f(x)$ of $X$ to
	$f$ at $x\in X$, and it 
	satisfies the following properties:
	\begin{enumerate}
		\item\label{i:p1} $\pars f=\partial f$ if $f$ is a proper lower semicontinuous convex function;
		\item\label{i:p1.5} $\pars f=\nabla f$ if $f$ is continuously differentiable;
		\item\label{i:p2} $0\in\pars f(x)$ if $f$ attains a local minimum at $x\in\dom f$;
		\item\label{i:p3} for every $\beta\in\RR$, 
		$$\pars\Big(f+\beta\tfrac{\|\cdot-x\|^2}{2}\Big)=\pars f +\beta (\Id-x).$$
	\end{enumerate}
\end{definition}
The Clarke--Rockafellar subdifferential,
Mordukhovich subdifferential, and Frech\'et subdifferential all
satisfy Definition~\ref{def:abst}\ref{i:p1}--\ref{i:p3},
see, e.g., \cite{Clarke90}, \cite{Mord06}, \cite{Mord18},
so they are $\pars$. 

Let $\lambda>0$. Recall that $f$
is ${\lambda}$-hypoconvex (see \cite{Rock98, Wang2010}) if
\begin{equation}
\label{eq:def:hypocon}
f((1-\tau)x+\tau y)\le (1-\tau) f(x)
+\tau f(y) +\tfrac{\lambda}{2}\ \tau(1-\tau)
\norm{x-y}^2,
\end{equation}
for all $(x,y)\in X\times X$ and $\tau\in \left]0,1\right[$, 
or equivalently; 
\begin{equation}
\text{$f+\tfrac{\lambda}{2}\norm{\cdot}^2$ is convex.}
\end{equation}
For $\gamma>0$,  
the \emph{proximal mapping} $\prox{\gamma f}$ is defined 
at $x\in X$ by
\begin{equation}
\prox{\gamma f}(x)
=\underset{y\in X}{\argmin}
\Big(f(y)+\tfrac{\gamma}{2}\|x-y\|^2\Big).
\end{equation}
\begin{fact}
	\label{f:sub:for}
	Suppose that  
	$f\colon X\to \left]-\infty,+\infty\right]$ is a proper
	lower semicontinuous ${\lambda}$-hypoconvex 
	function. Then
	\begin{equation}\label{e:hypoc:sub}
	\pars f=\partial\Big(f+\tfrac{\lambda}{2}\|\cdot\|^2\Big)-{\lambda}\Id.
	\end{equation}
	Moreover, we have:
	\begin{enumerate}
		\item 
		\label{f:sub:for:i}
		The Clarke--Rockafellar,
		Mordukhovich,
		and Frech\'et subdifferential operators of $f$ 
		all coincide.
		\item 
		\label{f:sub:for:ii} 
		$\pars f$ is maximally $-\lambda$-monotone.
		
		\item 
		\label{f:sub:for:iii}
		$(\forall \gamma\in \left]0,{1}/{\lambda}\right[)$
		$\prox{\gamma f}$ 
		is single-valued and $\dom\prox{\gamma f}=X$.
	\end{enumerate}
\end{fact}
\begin{proof}
	See \cite[Proposition~6.2~and~Proposition~6.3]{BMW19}.	
\end{proof}
\begin{proposition}
	\label{prop:sum}
	Let $\mu\ge \omega\ge 0$.
	Suppose that $\argmin(f+g)\neq \fady$ and that 
	one of the following conditions is satisfied:
	\begin{enumerate}
		\item 
		\label{c:sum:1}
		$f$ is $\mu $- strongly convex,
		$g$ is $\omega $- hypoconvex.
		
		\item 
		\label{c:sum:2}
		$f$ is $\omega $- hypoconvex, and 
		$g$ is $\mu $- strongly convex. 
	\end{enumerate}
	Then $f+g$ is convex and 
	$\pars(f+g)=\partial (f+g)$.
	
	If, in addition, one of the following
	conditions is satisfied:
	\begin{enumerate}[(a)]
		\item 
		\label{sc:1}
		$0\in \sri (\dom f-\dom g)$.
		\item 
		\label{sc:2}
		$X$ is finite dimensional and 
		$0\in \ri (\dom f-\dom g)$.
		\item 
		\label{sc:3}
		$X$ is finite dimensional, 
		$f$ and $g$ are polyhedral, and  
		$\dom f\cap \dom g\neq \fady$.
	\end{enumerate}
	Then 
	\begin{equation}
	\label{eq:argmin:sum}
	\pars(f+g)=\partial(f+g)=\pars f+\pars g,
	\end{equation}
	and 
	\begin{equation}
	\label{eq:argmin:sum:zeros}
	\zer\pars(f+g)=\zer(\pars f+\pars g)=\argmin(f+g).
	\end{equation}
\end{proposition}
\begin{proof}
	It is clear that either \cref{c:sum:1} or \cref{c:sum:2}
	implies that $f+g$ is convex and the identity 
	follows in view  of \cref{def:abst}\ref{i:p1}.
	Now,
	suppose that \cref{c:sum:1} holds along with 
	one of the assumptions \cref{sc:1}--\cref{sc:3}.
	Rewrite $f$ and $g$ 
	as $(f,g)=(\overline{f}+\tfrac{\mu}{2}\norm{\cdot}^2,
	\overline{g}-\tfrac{\omega}{2}\norm{\cdot}^2)$
	and observe that both $\overline{f}$ and $\overline{g}$
	are convex, as is $\overline{f}+\overline{g}$.
	Moreover, we have $\dom f=\dom \overline{f}$
	and $\dom g=\dom \overline{g}$.
	Now,
	\begin{subequations}
		\begin{align}
		\pars(f+g)
		&=\pars(\overline{f}+\overline{g}
		+\tfrac{\mu-\omega}{2}\norm{\cdot}^2)
		\label{se:pars:1}
		\\
		&=\pars(\overline{f}+\overline{g})
		+({\mu-\omega})\Id=\partial(\overline{f}+\overline{g})
		+({\mu-\omega})\Id
		\label{se:pars:2}
		\\
		&=\partial\overline{f}+\partial\overline{g}
		+({\mu-\omega})\Id
		=\partial\overline{f}+\mu\Id+\partial\overline{g}-\omega \Id
		\label{se:pars:3}
		\\
		&=\partial f+\pars g=\pars f+\pars g.
		\label{se:pars:4}
		\end{align}
	\end{subequations}
	Here, \cref{se:pars:2} follows from
	applying \cref{def:abst}\ref{i:p3}
	to $\overline{f}+\overline{g}$,
	\cref{se:pars:3} 
	follows from \cite[Theorem~16.47]{BC2017}
	applied to $\overline{f}$ and $\overline{g}$,
	and 
	\cref{se:pars:3} 
	follows from applying \cref{f:sub:for} to 
	$f$  and $g$
	and using \cref{def:abst}\ref{i:p1}, 
	which verify \cref{eq:argmin:sum}.
	Finally, \cref{eq:argmin:sum:zeros}
	follows from combining \cref{eq:argmin:sum}
	and \cite[Theorem~16.3]{BC2017}.
\end{proof}

The following theorem 
provides an alternative proof to 
\cite[Theorem~4.4]{GHY17}
and \cite[Theorem~5.4(ii)]{DP18}.

\begin{theorem}
	\label{thm:DR:relax:functions}
	Let $\mu>\omega\ge 0$,
%	let $\lambda\in \left]0,1\right[$
	and let $\gamma\in \left]0, 
	{(\mu-\omega)}/(2\mu\omega)\right[$.
	Suppose that one of the following holds:
	\begin{enumerate}
		\item 
		\label{asm:order:sq:1:f}
		$f$ is $\mu $- strongly convex,
		$g$ is $\omega $- hypoconvex.
		
		\item 
		\label{asm:order:sq:2:f}
		$f$ is $\omega $-hypoconvex, and 
		$g$ is $\mu $-strongly convex. 
	\end{enumerate}
	and that $0\in \pars f+\pars g$ (see
	\cref{prop:sum} for sufficient conditions).
	Set 
	\begin{equation}
	T=\tfrac{1}{2}\big(\Id+
	(2\prox{\gamma g}-\Id)(2\prox{\gamma f}-\Id)\big),
	\quad
	\text{and}
	\quad 
	\alpha=
	\tfrac{\mu-\omega}{2(\mu-\omega-\gamma\mu\omega)},
	\end{equation}
	and let $x_0\in X$.
	Then  $\alpha\in \left]0,1\right[$, and
	$T$ 
	is $\alpha$-averaged.
	Moreover, 
	$(\exists\ \overline{x}\in \fix T)$
	such that $T^n x_0\weakly \overline{x}$,
	$\argmin(f+g)=\{ \prox{f} \overline{x}\}$,
	and  $\prox{f}T^n x_0 \to \prox{f} \overline{x}$.
\end{theorem}
\begin{proof}
	Suppose that \ref{asm:order:sq:1} holds.
	Then \cite[Example~22.4]{BC2017}
	(respectively  \cref{f:sub:for}\ref{f:sub:for:ii}) 
	implies that
	$\pars f=\partial f$ (respectively $\pars g$) 
	is maximally $\mu$-monotone
	(respectively maximally $(-\omega)$-monotone). 
	The conclusion follows from applying 
	\cref{thm:convergence:DR}\ref{thm:convergence:DR:i}
	with $(A,B)$ replaced by $(\pars f,\pars g)$.
	The proof for \ref{asm:order:sq:2} follows similarly, 
	by using \cref{thm:convergence:DR}\ref{thm:convergence:DR:ii}.
\end{proof}
Before we proceed further we recall the following 
useful fact.
\begin{fact}[Baillon--Haddad] 
	\label{fact:B:Haddad}
	Let $f\colon X\to \RR$ be a Frech\'et differentiable convex
	function and let $\beta>0$.
	Then $\grad f$ is $\beta$-Lipschitz continuous
	if and only if $\grad f$ is $\tfrac{1}{\beta}$-cocoercive.
\end{fact}
\begin{proof}
	See, e.g., \cite[Corollary~18.17]{BC2017}. 
\end{proof}
\begin{lemma}
	\label{lem:grad:st:con}
	Let $\mu \ge 0$, let $\beta>0$ and let 
	$f\colon X\to \RR$ be a Frech\'et differentiable 
	function. 
	Suppose that $f$ is $\mu$-strongly convex
	with a  ${\beta}$-Lipschitz continuous gradient.
	Then the following hold:
	\begin{enumerate}
		\item 	
		\label{lem:grad:st:con:0}
		$f-\tfrac{\mu}{2}\norm{\cdot}^2$ is convex. 
		\item 	
		\label{lem:grad:st:con:i}
		$\grad f$ is maximally $\mu$-monotone.
		% \item 	
		% \label{lem:grad:st:con:ii} 
		% $\grad f-\mu \Id $ is $\tfrac{1}{\beta}$-Lipschitz 
		% continuous. 
		\item 	
		\label{lem:grad:st:con:iii} 
		$\grad f-\mu \Id $ is 
		$\tfrac{1}{\beta}$-cocoercive.
	\end{enumerate} 	
\end{lemma}
\begin{proof}
	\ref{lem:grad:st:con:0}:
	See, e.g., \cite[Proposition~10.8]{BC2017}.
	\ref{lem:grad:st:con:i}:
	See, e.g., \cite[Example~22.4(iv)]{BC2017}.
	\ref{lem:grad:st:con:iii}:
	Combine \ref{lem:grad:st:con:0},
	\cref{lem:beta:mu:coco}
	and \cref{cor:f1:f2:lips}\ref{cor:f1:f2:lips:ii}
	applied with $(f_1,f_2)$ 
	replaced by $(f, \tfrac{\mu}{2}\norm{\cdot}^2)$.
\end{proof}

\begin{theorem}[the forward-backward algorithm 
	when $f$ is $\mu $-strongly convex]
	\label{thm:FB:A:functions}
	Let $\mu\ge \omega\ge 0$, and  let $\beta>0$.
	Let
	$f$ be $\mu $-strongly convex and Frech\'et differentiable with
	a ${\beta}$-Lipschitz continuous gradient,
	and let
	$g$ be  $\omega $-hypoconvex.
	Suppose that  $\argmin(f+g)\neq \fady$.
	Let $\gamma\in \left ]0,2/(\beta+2\mu)\right[$,
	and set $\delta=(1-\gamma\mu)/(1-\gamma \omega)$.
	Set $T=Prox_{\gamma g}(\Id-\gamma \grad f)$, 
	and let $x_0\in X$.
	Then  the following hold:
	\begin{enumerate} 
		\item
		\label{thm:FB:A:functions:i}
		There exists $\overline{x}\in \fix T=\zer (A+B)=\argmin(f+g)$
		such that $T^n x_0\weakly \overline{x}$.
	\end{enumerate}
	Suppose that  $\mu>\omega$. Then, we additionally have:
	\begin{enumerate} 
		\setcounter{enumi}{1}
		\item
		\label{thm:FB:A:functions:ii}
		$\fix T=\argmin(f+g)=\{ \overline{x}\}$ and 
		$T^n x_0\to \overline{x}$ with a linear rate $\delta<1$.
	\end{enumerate}
\end{theorem}

\begin{proof}
	Note that \cref{def:abst}\ref{i:p1.5} implies that
	$\pars f=\grad f$.
	Set $(A,B)=(\grad f,\pars g)$ and observe that
	\cref{prop:sum} and 
	\cref{prop:static:zeros:FB} imply that 
	$\fix T=\zer(A+B)=\argmin(f+g)$.	
	It follows from
	\cite[Example~22.4]{BC2017} 
	(respectively  \cref{f:sub:for}\ref{f:sub:for:ii}) 
	that
	$A$ (respectively $B$) 
	is maximally $\mu$-monotone
	(respectively maximally $(-\omega)$-monotone). 
	Moreover, \cref{lem:grad:st:con}\ref{lem:grad:st:con:iii}
	implies that $A-\mu \Id$ is $\tfrac{1}{\beta}$-cocoercive.
	\ref{thm:FB:A:functions:i}--\ref{thm:FB:A:functions:ii}:
	Apply \cref{thm:FB:av:1}\ref{thm:convergence:FB:A:i}\&   \ref{thm:convergence:FB:A:ii}.
\end{proof}

To proceed to the next result, we need the following lemma.

\begin{lemma}
	Let $\omega \ge 0$, let $\beta>0$ and let 
	$f\colon X\to \RR$ be a Frech\'et differentiable 
	function. 
	Suppose that $g$ is $\omega$-hypoconvex
	with a  $\tfrac{1}{\beta}$-Lipschitz continuous gradient.
	Then $\grad f+\omega\Id $ 
	is $\beta/(1+\omega \beta)$-cocoercive.
\end{lemma}

\begin{theorem}[the forward-backward algorithm
	when $f$ is $\omega$-hypoconvex]
	\label{thm:FB:B:functions}
	Let $\mu\ge\omega\ge 0$, let $\beta>0$,
	and let 
	$\overline{\beta}\in \left]
	\max\{\beta,2\omega \},+\infty \right[$.
	Let
	$f$ be $\omega$-hypoconvex,
	and let
	$g$ be  $\mu $-strongly convex and 
	Frech\'et differentiable with
	a ${\beta}$-Lipschitz continuous gradient.
	Suppose that  $\argmin(f+g)\neq \fady$.
	Let 
	$\gamma\in \left ]0,
	2/(\overline{\beta}-2\omega)\right[$,
	and set $\delta=(1+\gamma\omega)/(1+\gamma \mu)$.
	Set $T=Prox_{\gamma g}(\Id-\gamma \grad f)$, 
	and let $x_0\in X$.
	Then  the following hold:
	\begin{enumerate} 
		\item
		There exists $\overline{x}\in \fix T=\argmin(f+g)$
		such that $T^n x_0\weakly \overline{x}$.
	\end{enumerate}
	Suppose that  $\mu>\omega$. Then, we additionally have:
	\begin{enumerate} 
		\setcounter{enumi}{1}
		\item
		$\fix T=\argmin(f+g)=\{ \overline{x}\}$ and 
		$T^n x_0\to \overline{x}$ with a linear rate $\delta<1$.
	\end{enumerate}
\end{theorem}
\begin{proof}
	Proceed similar to the proof of
	\cref{thm:FB:A:functions} but use 
	\cref{thm:FB:av:2}\ref{thm:FB:av:2:iii}\&\ref{thm:FB:av:2:v}.
\end{proof}

\begin{theorem}[the forward-backward algorithm
	when $f$ is $1/\beta$-hypoconvex]
	\label{thm:FB:C:functions}
	Let 
	$\mu\ge \beta>0$,
	and let 
	$\overline{\beta}\in \left]2\beta,+\infty \right[$.
	Let
	$f$ be $\mu $-strongly convex,
	and let
	$g$ be  
	Frech\'et differentiable with
	a ${\beta}$-Lipschitz continuous gradient.
	Suppose that  $\argmin(f+g)\neq \fady$.
	Let
	$\gamma\in \left ]0,2
	/( \overline{\beta}-2\beta)\}\right[$,
	and set $\delta=(1+\gamma\beta)/(1+\gamma \mu)$.
	Set $T=Prox_{\gamma g}(\Id-\gamma \grad f)$, 
	and let $x_0\in X$.
	Then  the following hold:
	\begin{enumerate} 
		\item
		There exists $\overline{x}\in \fix T=\argmin(f+g)$
		such that $T^n x_0\weakly \overline{x}$.
	\end{enumerate}
	Suppose that  $\mu>1/\beta$. Then, we additionally have:
	\begin{enumerate} 
		\setcounter{enumi}{1}
		\item
		$\fix T=\argmin(f+g)=\{ \overline{x}\}$ and 
		$T^n x_0\to \overline{x}$ with a linear rate $\delta<1$.
	\end{enumerate}
\end{theorem}
\begin{proof}
	Combine \cref{lem:T:Lips:to:coco}
	applied with $A$ replaced by $\grad f$
	and 
	\cref{thm:FB:B:functions} applied with 
	$(\omega,\beta)$ replaced by 
	$(\beta,2\beta)$.
\end{proof}

\begin{remark}
The results of \cref{thm:FB:av:1b},
\cref{thm:FB:av:2:b},
 and \cref{thm:FB:av:3:b}
 can be directly applied to optimization settings
 in a similar fashion \`{a} la 
 \cref{thm:FB:A:functions},
 \cref{thm:FB:B:functions},
 and \cref{thm:FB:C:functions}.
\end{remark}

\section{Graphical characterizations}
\label{sec:graphical}

% ==============================================================

This section contains 2D-graphical representations
of different Lipschitz 
continuous operator classes that admit \IN s and of their composition classes. 
We illustrate exact shapes of the composition classes in 2D and conservative estimates from 
\cref{thm:comp:more:gen} and \cref{thm:conic:conic}. Similar graphical representations have appeared before in the 
literature. In \cite{Eckstein1989Splitting,Eckstein1992OnTheDouglas}, nonexpansiveness and firm nonexpansiveness ($\tfrac{1}{2}$-averagedness) are 
characterized. Early preprints of \cite{Giselsson2014Linear} have more 2D graphical 
representations, and the lecture notes \cite{Giselsson2015Lecture} contain many such 
characterizations with the purposes of illustrating how different properties relate to each 
other and to provide intuition on why different algorithms converge. This has been further extended and formalized in \cite{Ryu2019Scaled}. Not only do these illustrations provide intuition. Indeed, it is a straightforward consequence of, e.g., \cite{Ryu2019Scaled,Ryu2018Operator} that for 
compositions of two operator classes that admit \IN s, there always exists a
2D-worst case. Hence, if the 2D illustration implies that the composition class admits a specific $(\alpha,\beta)$-\IN, so does the full operator class.

In Section~\ref{sec:graph:sing_op}, we characterize many 
well-known special cases of operator classes that admit \IN s. 
In Section~\ref{sec:graph:comp}, we characterize classes 
obtained by compositions of such operator classes and highlight
differences between the true composition classes and their
characterizations using \cref{thm:comp:more:gen}.

\subsection{Single operators}

\label{sec:graph:sing_op}

We consider classes of $(\idparam,\Nparam)$-\IN\ of Lipschitz continuous operators. We graphically illustrate properties of
some special cases. The illustrations should be read as follows.
Assume that $x-y$ is represented by the marker in the figure. The
diagram then shows where $Rx-Ry$ can end up in relation to $x-y$.
If the point $x-y$ is rotated in the picture, the rest of the
picture rotates with it. The characterization is, by construction
of $(\idparam,\Nparam)$-\IN s, always a circle of radius
$\beta\|x-y\|$ shifted $\alpha\|x-y\|$ along the line defined by
the origin and the point $x-y$. 

\xdef\singleopscaling{3}
\paragraph{Lipschitz continuous operators.}

\xdef\betaone{0.8}

Let $\beta>0$ and let  $R\colon\hilbert\to\hilbert$.
Then $R$ is $\beta$-Lipschitz continuous if and only
if $R$ admits  an $(\idparam,\Nparam)$-\IN,
with $\idparam$ chosen as $0$. 
The following illustration shows the case $\beta=\betaone$. 
The radius of the Lipschitz circle is $\beta\|x-y\|$.

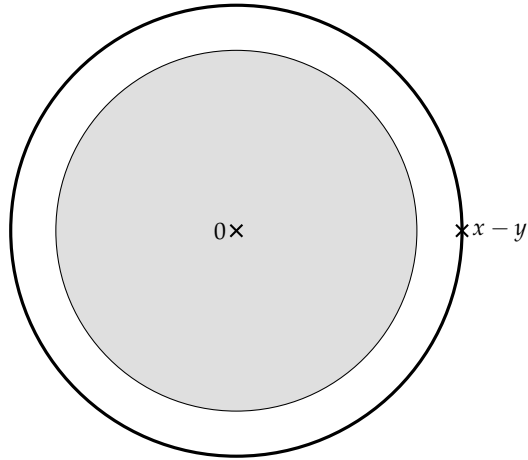
\ifx\showgraphics\one
\begin{center}
  \begin{tikzpicture}[scale=\singleopscaling]
  
  \lipcirc{\betaone}{25}
  
  \unitcirc{1}
\end{tikzpicture}
\end{center}
\fi

\paragraph{Cocoercive operators.}
\xdef\betaone{1.4}
\xdef\betatwo{0.7}

Let $\beta>0$, and 
let  $R\colon\hilbert\to\hilbert$.
Then $R$ is 
$\tfrac{1}{\beta}$-cocoercive
if and only if  $R$ admits an $(\idparam,\Nparam)$-\IN,
with $(\idparam,\Nparam)$ chosen as $\big(\tfrac{\beta}{2},\tfrac{\beta}{2}\big)$.
The following illustration shows the cases $\beta=\betaone$ and $\beta=\betatwo$. 
The diameter is $\beta\|x-y\|$. 
The figure clearly illustrates that $\tfrac{1}{\beta}$-cocoercive 
operators are also $\beta$-Lipschitz (but not necessarily the other way around).

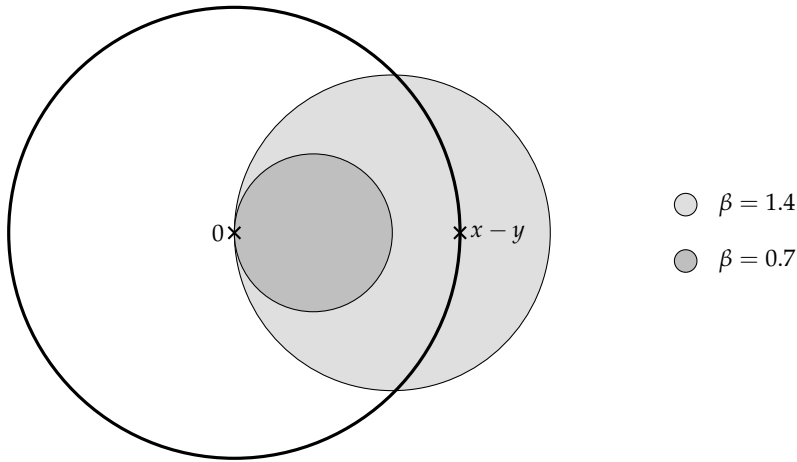
\ifx\showgraphics\one
\begin{center}
\begin{tikzpicture}[scale=\singleopscaling]
  \cococirc{1/\betaone}{25}
  \cococirc{1/\betatwo}{50}
  \unitcirc{1}
  \xdef\legendarray{{\betaone,\betatwo}}
  \xdef\colormixarray{{25,50}}
  \legend{2}{2}{0.1}{0.25}{\beta}{\legendarray}{\colormixarray}
  
\end{tikzpicture}
\end{center}
\fi

\paragraph{Averaged operators.}

\xdef\alphaone{0.25}
\xdef\alphatwo{0.5}
\xdef\alphathree{0.75}

Let $\alpha\in \left]0,1\right[$, and 
let  $R\colon\hilbert\to\hilbert$.
Then $R$ is 
$\alpha$-averaged if and only if 
$R$ admits an $(\idparam,\Nparam)$-\IN,
with $(\idparam,\Nparam)$ chosen as $\big(1-\alpha,\alpha\big)$. 
The following illustration shows the cases $\alpha=\alphaone$ 
and $\alpha=\alphatwo$, and $\alpha=\alphathree$. All averaged operators are nonexpansive.

\ifx\showgraphics\one
\begin{center}
\begin{tikzpicture}[scale=\singleopscaling]

  \coniccirc{\alphathree}{25}
  \coniccirc{\alphatwo}{50}
  \coniccirc{\alphaone}{75}
  \unitcirc{1}
  \xdef\legendarray{{\alphaone,\alphatwo,\alphathree}}
  \xdef\colormixarray{{75,50,25}}
  \legend{3}{2}{0.1}{0.25}{\alpha}{\legendarray}{\colormixarray}
\end{tikzpicture}
\end{center}
\fi

\paragraph{Conic operators.}

\xdef\alphaone{1.2}
\xdef\alphatwo{1.5}

Let $\alpha>0$, and 
let  $R\colon\hilbert\to\hilbert$.
Then $R$ is
$\alpha$-conically nonexpansive if and only if
$R$ admits an $(\idparam,\Nparam)$-\IN,
with $(\idparam,\Nparam)$ chosen as $\big(1-\alpha,\alpha\big)$.  The following illustration shows the cases $\alpha=\alphaone$ and $\alpha=\alphatwo$. Conically 
nonexpansive operators fail to be nonexpansive for $\alpha>1$. 

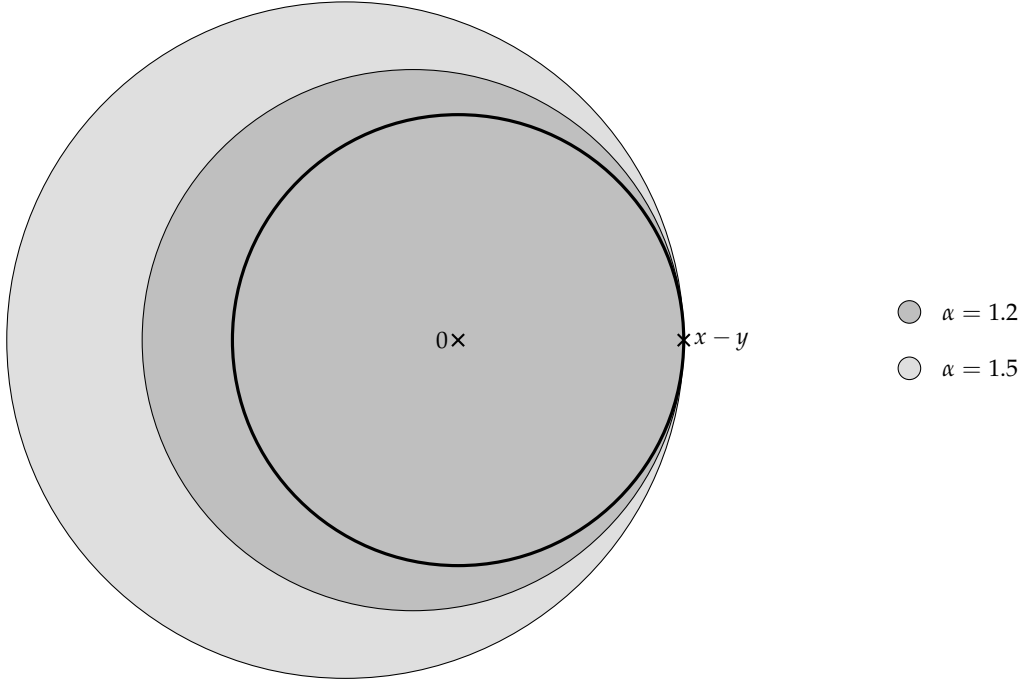
\ifx\showgraphics\one
\begin{center}
\begin{tikzpicture}[scale=\singleopscaling]

  %\coniccirc{\alphathree}{25}
  \coniccirc{\alphatwo}{25}
  \coniccirc{\alphaone}{50}
  \unitcirc{1}
  \xdef\legendarray{{\alphaone,\alphatwo}}
  \xdef\colormixarray{{50,25}}
  \legend{2}{2}{0.1}{0.25}{\alpha}{\legendarray}{\colormixarray}
\end{tikzpicture}
\end{center}
\fi

\paragraph{$\mu$-monotone operators.} 
    \xdef\muvar{0.2}

Let $\mu\in \RR$, and suppose that $A\colon X\rras X$ is
$\mu$-monotone.
The shortest distance between the vertical line and the origin 
in the illustration is $|\mu|\|x-y\|$. The following illustration show the case $\mu=\muvar$.

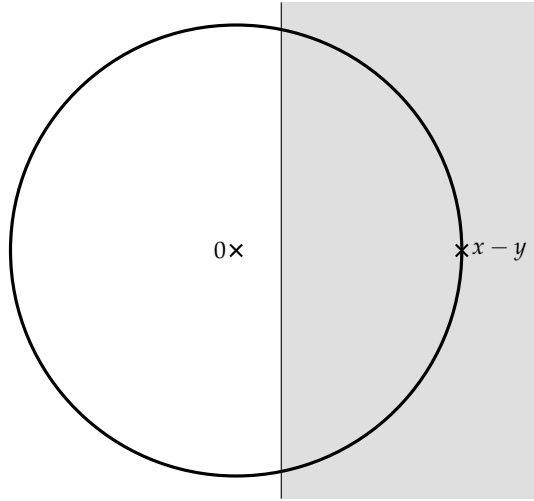
\ifx\showgraphics\one
\begin{center}
  \begin{tikzpicture}[scale=\singleopscaling]
  \mumono{\muvar}{25}
  \unitcirc{1}
  \xdef\legendarray{{\alphaone,\alphatwo}}
  \xdef\colormixarray{{50,25}}

\end{tikzpicture}
\end{center}
\fi

\subsection{Compositions of two operators}

\label{sec:graph:comp}

In this section, we provide illustrations of compositions of 
different classes of Lipschitz continuous operators. We consider compositions
of the form
\begin{equation*}
\text{$R=R_2R_1$ where $R_i$ admits an $(\idparam_i,\Nparam_i)$-\IN,}
\end{equation*}
$\forall i\in\{1,2\}$. 
Let $(x,y)\in X\times X$.
We illustrate the regions within which 
$R_2R_1x-R_2R_1y$ can end up. 
For most considered composition classes, we provide two illustrations. 
The left illustration explicitly shows how the the
composition is constructed. It shows the region within which $R_1x-R_1y$ must end up. 
The second operator $R_2$ is applied at a subset, marked by crosses, 
of boundary points of that region. Given these as starting points for $R_2$ application, 
the dashed circles show where $R_2R_1x-R_2R_1y$ can end up for this subset. 
The right illustration shows, in gray, the resulting exact shape of the composition. It also contains the estimate from 
\cref{thm:comp:more:gen} that provides an \IN~of the composition. 
From these illustrations, it is obvious that many different \IN~are valid. The illustrations also reveal that the specific
\IN s provided in \cref{thm:comp:more:gen} indeed are suitable for our purpose of characterizing the composition as 
averaged, conic, or contractive.

\paragraph{Averaged-averaged composition.}

\xdef\thetaone{0.5}
\xdef\thetatwo{0.5}
\xdef\thetavar{1}
\pgfmathsetmacro{\alphaone}{1-\thetaone}
\pgfmathsetmacro{\betaone}{\thetaone}
\pgfmathsetmacro{\alphatwo}{1-\thetatwo}
\pgfmathsetmacro{\betatwo}{\thetatwo}
\computealphabetacomp{\thetavar}{\alphaone}{\betaone}{\alphatwo}{\betatwo}
We first consider $\alpha_i$-averaged $R_i$ with $\alpha_i\in]0,1[$.
A special case is the 
forward-backward splitting operator 
$T=J_{\gamma B}(\id-\gamma A)$ with $\tfrac{1}{\beta}$-cocoercive $A$  
and maximally monotone $B$. This implies that $(\id-\gamma A)$ is $\tfrac{\gamma\beta}{2}$-averaged for 
$\gamma\in\left]0,\tfrac{2}{\beta}\right[$ and that $J_{\gamma B}$ is 
$\tfrac{1}{2}$-averaged. The example below has individual averagedness 
parameters $\alpha_1=\thetaone$ and $\alpha_2=\thetatwo$, i.e., $R=R_2R_1$ with 
$R_1=\thetaone\id+\thetaone N_1$ and $R_2=\thetatwo\id+\thetatwo N_2$. 
\cref{thm:comp:more:gen} shows that the composition 
is of the form $\pgfmathprintnumber{\alphacomptwo}\id+\pgfmathprintnumber{\betacomptwo} N$, 
where $N$ is nonexpansive, i.e., it is $\pgfmathprintnumber{\betacomptwo}$-averaged. The the composition is averaged is already known, see
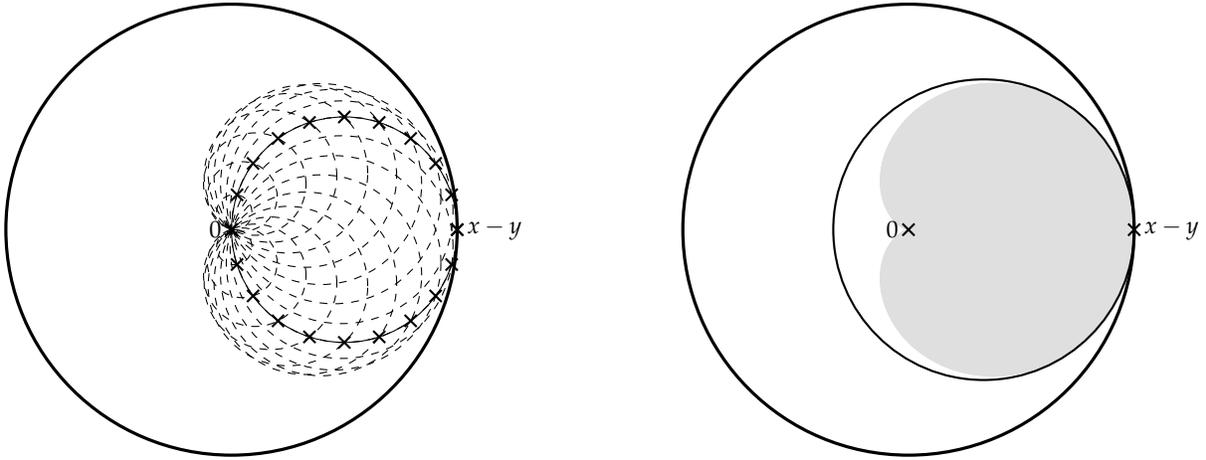
\cite{Combettes2015Compositions,Giselsson2017Tight}. 

\xdef\compopscaling{3}
\ifx\showgraphics\one
\begin{center}
  \begin{tikzpicture}[scale=\compopscaling]
    \begin{scope}
      %\clip (-2,-1.5) rectangle (1.5,1.5);
    \composition{\thetavar}{\alphaone}{\betaone}{\alphatwo}{\betatwo}{0}{0}{0}
  \end{scope}
\begin{scope}[xshift=3cm]
  %\clip (-2,-2) rectangle (1.5,2);
\composition{\thetavar}{\alphaone}{\betaone}{\alphatwo}{\betatwo}{1}{0}{1}
\end{scope}
\end{tikzpicture}
\end{center}
\fi

\xdef\thetaone{0.7}
\xdef\thetatwo{0.6}
\xdef\thetavar{1}
\pgfmathsetmacro{\alphaone}{1-\thetaone}
\pgfmathsetmacro{\betaone}{\thetaone}
\pgfmathsetmacro{\alphatwo}{1-\thetatwo}
\pgfmathsetmacro{\betatwo}{\thetatwo}
\computealphabetacomp{\thetavar}{\alphaone}{\betaone}{\alphatwo}{\betatwo}

The following example shows $\alpha_1=\thetaone$ and $\alpha_2=\thetatwo$. 
\cref{thm:comp:more:gen} shows that the composition is of the form 
$\pgfmathprintnumber{\alphacomptwo}\id+\pgfmathprintnumber{\betacomptwo} N$, where $N$ is 
nonexpansive, i.e., it is $\pgfmathprintnumber{\betacomptwo}$-averaged.

\ifx\showgraphics\one
\begin{center}
  \begin{tikzpicture}[scale=\compopscaling]
    \begin{scope}
      %\clip (-2,-1.5) rectangle (1.5,1.5);
    \composition{\thetavar}{\alphaone}{\betaone}{\alphatwo}{\betatwo}{0}{0}{0}
  \end{scope}
\begin{scope}[xshift=3cm]
  %\clip (-2,-2) rectangle (1.5,2);
\composition{\thetavar}{\alphaone}{\betaone}{\alphatwo}{\betatwo}{1}{0}{1}
  \end{scope}
\end{tikzpicture}
\end{center}
\fi

\paragraph{Conic-conic composition.}

\xdef\compopscaling{1.5}
\xdef\thetaone{1.7}
\xdef\thetatwo{0.45}
\xdef\thetavar{1}
\pgfmathsetmacro{\alphaone}{1-\thetaone}
\pgfmathsetmacro{\betaone}{\thetaone}
\pgfmathsetmacro{\alphatwo}{1-\thetatwo}
\pgfmathsetmacro{\betatwo}{\thetatwo}
\computealphabetacomp{\thetavar}{\alphaone}{\betaone}{\alphatwo}{\betatwo}

We consider $\alpha_i$-averaged $R_i$ with $\alpha_i>0$. Several examples with this setting are considered 
in for Douglas-Rachford splitting and forward-backward splitting in \cref{sec:main:1} and 
\cref{sec:main:2}. We know from \cref{thm:conic:conic} that the composition is conic if 
$\alpha_1\alpha_2<1$. The example below has $\alpha_1=\thetaone$ and $\alpha_2=\thetatwo$, 
that satisfies 
$\alpha_1\alpha_2=\pgfmathparse{\thetaone*\thetatwo}\pgfmathprintnumber{\pgfmathresult}<1$. 
\cref{thm:conic:conic} shows that the composition is of the form 
$\pgfmathprintnumber{\alphacomptwo}\id+\pgfmathprintnumber{\betacomptwo} N$, 
where $N$ is nonexpansive, i.e., it is 
$\pgfmathprintnumber{\betacomptwo}$-conic.

\ifx\showgraphics\one
\begin{center}
  \begin{tikzpicture}[scale=\compopscaling]
    \begin{scope}
      \clip (-2.5,-2) rectangle (2,2);
    \composition{\thetavar}{\alphaone}{\betaone}{\alphatwo}{\betatwo}{0}{0}{0}
  \end{scope}
\begin{scope}[xshift=5cm]
  \clip (-2.5,-2) rectangle (2,2);
\composition{\thetavar}{\alphaone}{\betaone}{\alphatwo}{\betatwo}{1}{0}{1}
  \end{scope}
\end{tikzpicture}
\end{center}
\fi

\xdef\thetaone{1.7}
\xdef\thetatwo{0.7}
\xdef\thetavar{1}

In \cref{ex:cases}, we have shown that the assumption $\alpha_1\alpha_2<1$ is 
critical for the composition to be conic. The following example illustrates the case 
$\alpha_1=\thetaone$ and $\alpha_2=\thetatwo$, which satisfies 
$\alpha_1\alpha_2=\pgfmathparse{\thetaone*\thetatwo}\pgfmathprintnumber{\pgfmathresult}>1$, 
hence \cref{thm:conic:conic} cannot be used to deduce that the composition is conic. Indeed, 
we see from the figure that the composition is not conic. It is impossible to draw a circle 
that touches the marker at $x-y$ and extends only to the left.

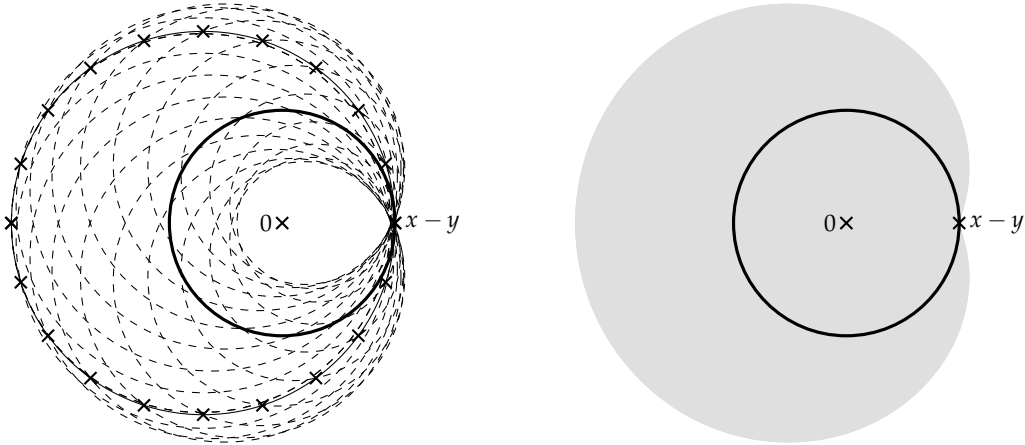
\ifx\showgraphics\one
\begin{center}
  \begin{tikzpicture}[scale=\compopscaling]
    \begin{scope}
      \clip (-2.5,-2) rectangle (2,2);
    \pgfmathsetmacro{\alphaone}{1-\thetaone}
    \pgfmathsetmacro{\betaone}{\thetaone}
    \pgfmathsetmacro{\alphatwo}{1-\thetatwo}
    \pgfmathsetmacro{\betatwo}{\thetatwo}
    \composition{\thetavar}{\alphaone}{\betaone}{\alphatwo}{\betatwo}{0}{0}{0}
  \end{scope}
\begin{scope}[xshift=5cm]
  \clip (-2.5,-2) rectangle (2,2);
  \composition{\thetavar}{\alphaone}{\betaone}{\alphatwo}{\betatwo}{0}{0}{1}
  \end{scope}
\end{tikzpicture}
\end{center}
\fi 

\xdef\gammavar{3.9}
\xdef\alphavar{0.05}
\xdef\thetavar{1}
\pgfmathsetmacro{\alphaone}{1-\gammavar/2}
\pgfmathsetmacro{\betaone}{\gammavar/2}
\pgfmathsetmacro{\alphatwo}{1/2}
\pgfmathsetmacro{\betatwo}{1/2}
\computealphabetacomp{\thetavar}{\alphaone}{\betaone}{\alphatwo}{\betatwo}

We conclude the conic composed with conic examples with a forward-backward example. 
The forward-backward splitting operator $J_{\gamma B}(\id-\gamma A)$ with $A$ 
$\tfrac{1}{\beta}$-cocoercive and $B$ (maximally) monotone is composed of $\tfrac{1}{2}$-averaged 
resolvent $J_{\gamma B}$ and $\tfrac{\gamma\beta}{2}$-conic forward step$(\id-\gamma A)$. 
The composition $R=R_2R_1$ with $R_i$ $\alpha_i$-conic is conic if $\alpha_1\alpha_2<1$, 
\cref{thm:conic:conic}. In the forward-backward setting, this corresponds to 
$\gamma\in(0,\tfrac{4}{\beta})$, which doubles the allowed range compared to guaranteeing an averaged 
composition. This extended range has been shown before, 
e.g., in 
\cite{Giselsson2019Nonlinear,Latafat2017Asymmetric}.

Below, we illustrate the forward-backward setting with $\gamma=\tfrac{\gammavar}{\beta}$. 
This corresponds to conic parameters $\alpha_1=\betaone$ and $\alpha_2=\betatwo$, i.e., 
$R=R_2R_1$ with $R_1=\alphaone\id+\betaone N_1$ and $R_2=\alphatwo\id+\betatwo N_2$. 
The composition is of the form 
$\pgfmathprintnumber{\alphacomptwo}\id+\pgfmathprintnumber{\betacomptwo} N$, where $N$ is 
nonexpansive, i.e., it is $\pgfmathprintnumber{\betacomptwo}$-conic, \cref{thm:conic:conic}. 
The left figure shows 
the resulting composition and (parts of) the conic approximation. The conic approximation is 
very large compared to the actual region. This is due to the local behavior around the point 
$x-y$, where it is almost vertical. As $\gamma\nearrow 4\beta$, the exact shape approaches 
being vertical around $x-y$ and the conic circle approaches to have an infinite radius. For 
$\gamma>4\beta$, the exact shape extends to the right of $x-y$ (as in the figure above), and 
the composition will not be conic.

\xdef\thetavartext{0.04}
In the right figure, we consider the relaxed forward-backward map $(1-\theta)\id+\theta J_{\gamma B}(\id-\gamma A)$ with $\theta>0$. If the composition $J_{\gamma B}(\id-\gamma A)$ is $\alpha$-conic, it is straightforward to verify that the relaxed map is $\theta\alpha$-conic. Therefore, any $\theta\in(0,\alpha^{-1})$ gives an $\theta\alpha$-averaged relaxed forward-backward map. An averaged map is needed to guarantee convergence to a fixed-point when iterated. In the figure, we let $\theta=\thetavartext$, which satisfies $\theta<\alpha^{-1}\approx\pgfmathparse{1/\betacomptwo}\pgfmathprintnumber[fixed]{\pgfmathresult}$. The approximation is indeed averaged, but the region within which the composition can end up is very small compared to the conic approximation.

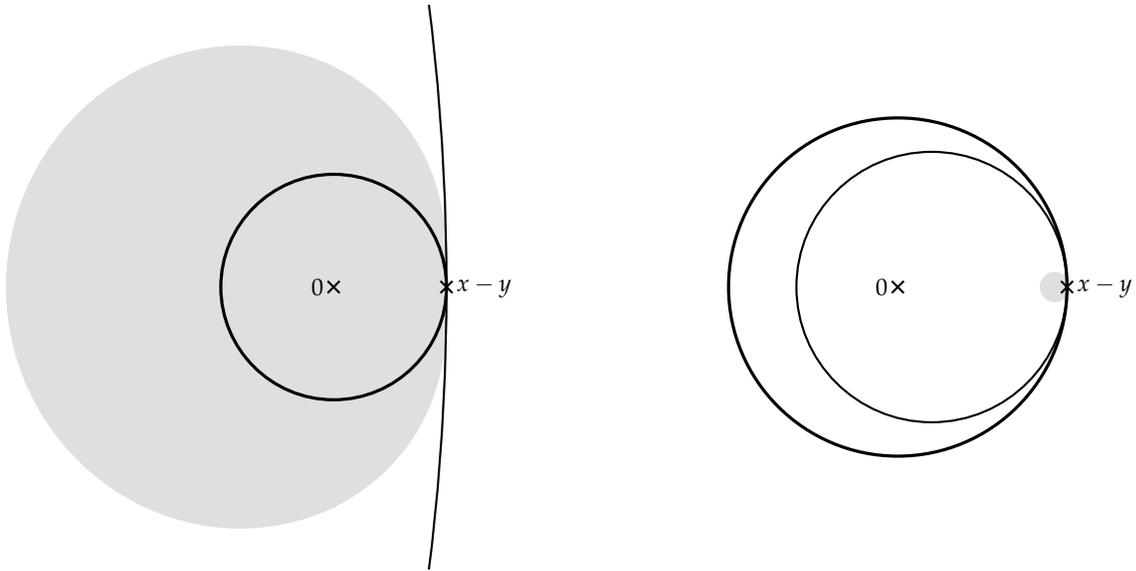
\ifx\showgraphics\one
\begin{center}
  \begin{tikzpicture}[scale=\compopscaling]
    \begin{scope}
      \clip (-3,-2.5) rectangle (2,2.5);
    \composition{\thetavar}{\alphaone}{\betaone}{\alphatwo}{\betatwo}{1}{0}{1}
  \end{scope}
  \xdef\thetavar{\thetavartext}
  \computealphabetacomp{\thetavar}{\alphaone}{\betaone}{\alphatwo}{\betatwo}
\begin{scope}[xshift=5cm,,scale=1.5]
  \clip (-2.5,-2) rectangle (2,2);
  \composition{\thetavar}{\alphaone}{\betaone}{\alphatwo}{\betatwo}{1}{0}{1}
  \end{scope}
\end{tikzpicture}
\end{center}
\fi

\paragraph{Scaled averaged and cocoercive compositions.}

\xdef\compopscaling{3}
\pgfmathsetmacro{\gammavar}{2}
\pgfmathsetmacro{\muvar}{0.3}
\pgfmathsetmacro{\wvar}{0.3}
\pgfmathsetmacro{\thetavar}{1}
\pgfmathsetmacro{\alphaone}{1-\gammavar*(-\muvar+1/(2))}
\pgfmathsetmacro{\betaone}{\gammavar/(2)}
\pgfmathsetmacro{\alphatwo}{1/(2*(1+\gammavar*(\wvar))}
\pgfmathsetmacro{\betatwo}{1/(2*(1+\gammavar*(\wvar))}
\computealphabetacomp{\thetavar}{\alphaone}{\betaone}{\alphatwo}{\betatwo}

Compositions of scaled averaged and 
cocoercive operators are also special cases of 
scaled conic composed with scaled conic operators treated 
in \cref{thm:conic:conic}. 
It covers the forward backward examples in \cref{sec:main:2}, 
where identity is shifted between the operators and the sum is (strongly) monotone. 
The operators in the composition are of the form 
$R_1=\delta_1((1-\alpha_1)\id+\alpha_1 N_1)$ 
and $R_2=\tfrac{\beta_2}{2}(\id+N_2)$, where $\alpha_1\in(0,1)$, $\delta_1>0$, 
and $\beta_2>0$.

In the following example, we consider the forward-backward setting 
in \cref{thm:FB:av:2}. 
The forward backward map is $J_{\gamma B}(\id-\gamma A)$ 
and we let $A+\muvar\id$ be $1$-cocoercive, 
$B$ be maximally $\wvar$-monotone. 
That is, we have shifted $0.3\id$ from $A$ to $B$ 
and the sum is monotone. We use step-length $\gamma=\gammavar$. 
The proof of \cref{thm:FB:av:2} shows that, in our setting, $R_1$ is 
$\pgfmathparse{(1+\gammavar*\muvar)}\pgfmathprintnumber{\pgfmathresult}$-scaled 
$\pgfmathparse{\gammavar/(2*(1+\gammavar*\wvar))}\pgfmathprintnumber{\pgfmathresult}$-averaged 
and that $R_2$ is 
$\pgfmathparse{(1+\gammavar*\wvar)}\pgfmathprintnumber{\pgfmathresult}$-cocoercive. 
\cref{thm:comp:more:gen} implies that the composition is of the form 
$\pgfmathprintnumber{\alphacomptwo}\id+\pgfmathprintnumber[fixed]{\betacomptwo} N$, 
where $N$ is nonexpansive, i.e., it is $\pgfmathprintnumber[fixed]{\betacomptwo}$-averaged.

\ifx\showgraphics\one
\begin{center}
  \begin{tikzpicture}[scale=\compopscaling]
    \begin{scope}
      %\clip (-1.1,-2) rectangle (2,2);
    \composition{1}{\alphaone}{\betaone}{\alphatwo}{\betatwo}{0}{0}{0}
  \end{scope}
  \begin{scope}[xshift=3cm]
    %\clip (-1.1,-2) rectangle (1.5,2);
    \composition{1}{\alphaone}{\betaone}{\alphatwo}{\betatwo}{1}{0}{1}
  \end{scope}
\end{tikzpicture}
\end{center}
\fi

\pgfmathsetmacro{\gammavar}{2}
\pgfmathsetmacro{\muvar}{0.2}
\pgfmathsetmacro{\wvar}{0.3}
\pgfmathsetmacro{\thetavar}{1}
\pgfmathsetmacro{\alphaone}{1-\gammavar*(-(\muvar)+1/(2))}
\pgfmathsetmacro{\betaone}{\gammavar/(2)}
\pgfmathsetmacro{\alphatwo}{1/(2*(1+\gammavar*(\wvar))}
\pgfmathsetmacro{\betatwo}{1/(2*(1+\gammavar*(\wvar))}
\computealphabetacomp{\thetavar}{\alphaone}{\betaone}{\alphatwo}{\betatwo}

The following example considers a similar forward-backward setting, 
but with a strongly monotone sum. We let $A+\muvar\id$ be $1$-cocoercive, 
$B$ be maximally $\wvar$-monotone, which implies that the sum is 
$\pgfmathparse{\wvar-\muvar}\pgfmathprintnumber{\pgfmathresult}$-strongly monotone. 
We keep step-length $\gamma=\gammavar$. The proof of \cref{thm:FB:av:2} shows that, 
in our setting, $R_1$ is 
$\pgfmathparse{(1+\gammavar*\muvar)}\pgfmathprintnumber{\pgfmathresult}$-scaled 
$\pgfmathparse{\gammavar/(2*(1+\gammavar*\wvar))}\pgfmathprintnumber{\pgfmathresult}$-averaged and that $R_2$ is $\pgfmathparse{(1+\gammavar*\wvar)}\pgfmathprintnumber{\pgfmathresult}$-cocoercive. \cref{thm:comp:more:gen} implies that the composition is of the form 
$\pgfmathprintnumber{\alphacomptwo}\id+\pgfmathprintnumber[fixed]{\betacomptwo} N$, where $N$ is nonexpansive, i.e., it is 
$\pgfmathparse{\alphacomptwo+\betacomptwo}\pgfmathprintnumber[fixed]{\pgfmathresult}$-contractive.

\ifx\showgraphics\one
\begin{center}
  \begin{tikzpicture}[scale=\compopscaling]
    \begin{scope}
      %\clip (-1.1,-2) rectangle (2,2);
      \composition{1}{\alphaone}{\betaone}{\alphatwo}{\betatwo}{0}{0}{0}
    \end{scope}
    \begin{scope}[xshift=3cm]
      %\clip (-1.1,-2) rectangle (1.5,2);
      \composition{1}{\alphaone}{\betaone}{\alphatwo}{\betatwo}{1}{0}{1}
    \end{scope}
  \end{tikzpicture}
\end{center}
\fi

\pgfmathsetmacro{\gammavar}{2}
\pgfmathsetmacro{\muvar}{0.4}
\pgfmathsetmacro{\wvar}{0.3}
\pgfmathsetmacro{\alphaone}{1-\gammavar*(-(\muvar)+1/(2))}
\pgfmathsetmacro{\betaone}{\gammavar/(2)}
\pgfmathsetmacro{\alphatwo}{1/(2*(1+\gammavar*(\wvar))}
\pgfmathsetmacro{\betatwo}{1/(2*(1+\gammavar*(\wvar))}
\computealphabetacomp{\thetavar}{\alphaone}{\betaone}{\alphatwo}{\betatwo}

The final example considers a similar forward-backward setting where the sum is not monotone. We let $A+\muvar\id$ be $1$-cocoercive, $B$ be maximally $\wvar$-monotone, which implies 
that the sum is $\pgfmathparse{\wvar-\muvar}\pgfmathprintnumber[fixed]{\pgfmathresult}$-monotone, i.e., it is not monotone. We use step-length $\gamma=\gammavar$. The proof of 
\cref{thm:FB:av:2} shows that, in our setting, $R_1$ is 
$\pgfmathparse{(1+\gammavar*\muvar)}\pgfmathprintnumber{\pgfmathresult}$-scaled 
$\pgfmathparse{\gammavar/(2*(1+\gammavar*\wvar))}\pgfmathprintnumber{\pgfmathresult}$-averaged and that $R_2$ is $\pgfmathparse{(1+\gammavar*\wvar)}\pgfmathprintnumber{\pgfmathresult}$-cocoercive. \cref{thm:comp:more:gen} implies that the composition is of the form $\pgfmathprintnumber{\alphacomptwo}\id+\pgfmathprintnumber[fixed]{\betacomptwo} N$, 
where $N$ is nonexpansive, i.e., it is 
$\pgfmathparse{\alphacomptwo+\betacomptwo}\pgfmathprintnumber[fixed]{\pgfmathresult}$-Lipschitz and not conic, averaged, or contractive.

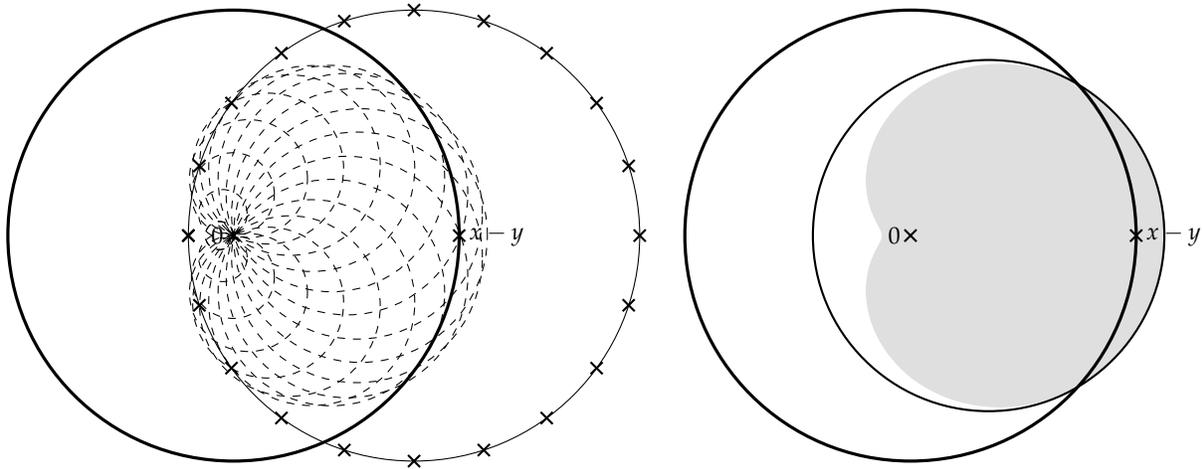
\ifx\showgraphics\one
\begin{center}
  \begin{tikzpicture}[scale=\compopscaling]
    \begin{scope}
      %\clip (-1.1,-2) rectangle (2,2);
      \composition{1}{\alphaone}{\betaone}{\alphatwo}{\betatwo}{0}{0}{0}
    \end{scope}
    \begin{scope}[xshift=3cm]
      %\clip (-1.1,-2) rectangle (1.5,2);
      \composition{1}{\alphaone}{\betaone}{\alphatwo}{\betatwo}{1}{0}{1}
    \end{scope}
  \end{tikzpicture}
\end{center}
\fi

% \begin{center}
%   \begin{tikzpicture}[scale=\compopscaling]

%     \pgfmathsetmacro{\gammavar}{1.4}
%     \pgfmathsetmacro{\alphaone}{1-\gammavar*(-0.3+1/(2))}
%     \pgfmathsetmacro{\betaone}{\gammavar/(2)}
%     \pgfmathsetmacro{\alphatwo}{1/(2*(1+\gammavar*(0.4))}
%     \pgfmathsetmacro{\betatwo}{1/(2*(1+\gammavar*(0.4))}
    
% \composition{1}{-0.9}{1.1}{0.5}{0.5}{1}{0}{1}
% \end{tikzpicture}
% \end{center}

% Wen, B., Chen, X.J., Pong, T.K.: Linear convergence 
% of proximal gradient algorithm with extrapolation
% for a class of nonconvex nonsmooth 
% minimization problems. SIAM J. Optim. 27, 124–145 (2017)

 % of proximal gradient algorithm with extrapolation
% for a class of nonconvex nonsmooth 
% minimization problems. SIAM J. Optim. 27, 124–145 (2017)
\small
\section*{Acknowledgment}
PG was partially supported by the Swedish Research Council.
WMM was partially supported by the Natural Sciences and 
Engineering Research Council of Canada.

\bibliographystyle{plain}
\bibliography{references}

% REFERENCES------------------------------------------------------------------

\begin{appendices}

	% \crefalias{section}{appsec}
	% \section{}
	% \label{app:A0}

	\crefalias{section}{appsec}
	\section{}
	\label{app:A}

\begin{proof}[Proof of \cref{lem:T:gra:AB}]
	Indeed, observe that
	\begin{equation}
		\label{eq:Tlamda}
		R_\lambda
		=(1-2\lambda)\Id +\lambda ((\Id+R_2)R_1-\Id-R_1)
	\end{equation}
	 and 
	 \begin{equation}
		\label{eq:Tlamda:disp}
		\Id-R_\lambda
		=\lambda (\Id+R_1-(\Id+R_2)R_1)=\lambda(\Id-R_2R_1).
	\end{equation}
	% Set
	% \begin{equation}
	% \label{e1:lem1}
	% u\coloneqq (2T_1-\Id)x, 
	% \quad 
	% v\coloneqq (2T_1-\Id)y.
	% \end{equation}
	Set $T_i=\Id+R_i$, $i\in \{1,2\}$.
	In view of \cref{eq:Tlamda} and \cref{eq:Tlamda:disp} 
	we have 
	\begin{subequations}
		\begin{align}
		&\quad\scal{R_\lambda x-R_\lambda y}{
			(\Id-R_\lambda)x-(\Id-R_\lambda)y}
		\\
				&=(1-2\lambda)\scal{x-y}{2\lambda((T_1x-T_2R_1x)-(T_1y-T_2R_1y))}
		\nonumber
		\\
		&\quad
		+\lambda^2\scal{T_1x-T_1y}{(x-T_1x)-(y-T_1y)}
		\nonumber
		\\
		&\quad+\lambda^2\scal{T_2R_1x-T_2R_1y}{(T_1x-T_2R_1x)-(T_1y-T_2R_1y)}
		\nonumber
		\\
		&\quad-\lambda^2\scal{T_2R_1x-T_2R_1y}{(x-T_1x)-(y-T_1y)}
		\\
		&=(1-2\lambda)\scal{x-y}{(\Id-T_\lambda)x-(\Id-T_\lambda)y}
		\nonumber
		\\
		&\quad
		+\lambda^2\scal{T_1x-T_1y}{(x-T_1x)-(y-T_1y)}
		\nonumber
		\\
		&\quad+\lambda^2\scal{T_2R_1x-T_2R_1y}{(R_1x-T_2R_1x)-(R_1y-T_2R_1y)},
		\end{align}
	\end{subequations}
	and the conclusion follows.
\end{proof}

\crefalias{section}{appsec}
\section{}
\label{app:CD}
\begin{proof}[Proof of \cref{lem:T:Lips:to:coco}]
\ref{lem:T:Lips:to:coco:i}:
Because $\tfrac{1}{\beta} A$ is nonexpansive, 
we learn from \cite[Example~20.7]{BC2017}
that $\Id +\tfrac{1}{\beta} A$,
as is ${\beta}\Id+A$, is maximally monotone.
The conclusion now follows in view of 
e.g., \cite[Lemma~2.5]{BMW19}. 
\ref{lem:T:Lips:to:coco:ii}:
This is clear by observing that
$\tfrac{1}{2\beta}\big({\beta}\Id+A\big)
=\tfrac{1}{2}(\Id+\tfrac{1}{\beta} A)$.
\end{proof}

\crefalias{section}{appsec}
\section{}
\label{app:C}
\begin{proof}[Proof of \cref{lem:beta:delta:coco}]
	Indeed, by assumption,
	there exist nonexpansive mappings $N_1\colon X\to X$
	and $N_2\colon X\to X$ such that
	\begin{equation}
	T_1=\tfrac{\beta}{2}\Id+\tfrac{\beta}{2}N_1,
	\quad
	T_2=\tfrac{\delta}{2}\Id+\tfrac{\delta}{2}N_2.
	\end{equation}
	Now, 
	\begin{subequations}
		\begin{align}
		\tfrac{1}{\beta}(T_1-T_2)
		&=\tfrac{1}{\beta} T_1-\tfrac{1}{\beta} T_2
		=\tfrac{1}{2}\Id+\tfrac{1}{2}N_1
		-\tfrac{\delta}{2\beta}\Id-\tfrac{\delta}{2\beta}N_2
		\\
		&=\tfrac{\beta-\delta}{2\beta}\Id
		+\tfrac{1}{2}N_1-\tfrac{\delta}{2\beta}N_2.
		\end{align}
	\end{subequations}
	Using the triangle inequality, one can
	directly verify that 
	$\tfrac{1}{\beta}(T_1-T_2)$ is Lipschitz continuous with a constant
	$\tfrac{\beta-\delta}{2\beta}+\tfrac{1}{2}+\tfrac{\delta}{2\beta}
	=1$. The proof is complete.
\end{proof}

\crefalias{section}{appsec}
\section{}
\label{app:D}

\begin{proof}[Proof of \cref{cor:f1:f2:lips}]
	\ref{cor:f1:f2:lips:i}:
	It follows from \cref{fact:B:Haddad} 
	that 
	$\grad f_1$ (respectively $\grad f_2$)
	is $\tfrac{1}{\beta}$-cocoercive 
	(respectively $\tfrac{1}{\delta}$-cocoercive).
	Now apply \cref{lem:beta:delta:coco} 
	with $(T_1,T_2)$ replaced by 
	$(\grad f_1,\grad f_2)$.
	\ref{cor:f1:f2:lips:ii}:
	Combine \ref{cor:f1:f2:lips:i}
	with \cref{fact:B:Haddad}
	applied with $f$ replaced by 
	$f_1-f_2$.
\end{proof}

\crefalias{section}{appsec}
\section{}
\label{app:E}

\begin{proof}[Proof of \cref{lem:T:av:factor}]
	\ref{lem:T:av:factor:i}:
	Indeed, we have 
	$\delta T=(1-(1-\delta(1-\alpha)))\Id+\delta\alpha N=
	(1-(1-\delta(1-\alpha)))\Id+(1-\delta(1-\alpha))\widetilde{N}$,
	where $\widetilde{N}=((\delta\alpha)/(1-\delta(1-\alpha)))N$.
	Note that 
	$(\delta\alpha)/(1-\delta(1-\alpha)\le 1$, hence 
	$\widetilde{N}$ is nonexpansive and the conclusion follows.
	\ref{lem:T:av:factor:ii}: Clear.
\end{proof}

\end{appendices}

\end{document}